\documentclass{amsart}

\usepackage[paperheight=11in, 
    paperwidth=8.5in,
    outer=1.25in,
    inner=1.25in,
    bottom=1.40in,
    top=1.40in]{geometry} 

\usepackage{amsmath, amsxtra, amsthm, amsfonts, amssymb}
\usepackage{mathrsfs}
\usepackage{graphicx}
\usepackage{url}
\usepackage{color}
\usepackage{bbm} 
\usepackage{tikz-cd}
\usepackage{tikz}

\usepackage[T1]{fontenc}
\usepackage{enumitem}
\usepackage{comment}
\usepackage[pdftex,hidelinks,backref=page]{hyperref} 
\hypersetup{
    colorlinks,
    citecolor=magenta,
    filecolor=magenta,
    linkcolor=blue,
    urlcolor=black
}
\usepackage{cleveref}

\renewcommand{\emptyset}{\varnothing}
\newcommand{\QQ}{\mathbb Q}
\newcommand{\RR}{\mathbb R}
\newcommand{\CC}{\mathbb C}
\newcommand{\PP}{\mathbb P}
\newcommand{\ZZ}{\mathbb Z}
\newcommand{\GG}{\mathbb G}
\newcommand{\M}{\mathrm{M}}

\newcommand{\D}{\mathrm{D}}

\newcommand{\kk}{\mathbbm{k}}
\newcommand{\be}{\mathbf e}

\newcommand{\rank}{\operatorname{rank}}
\newcommand{\id}{\mathrm{id}}
\newcommand{\codim}{\operatorname{codim}}
\newcommand{\ind}[1]{\mathbf 1(#1)} 
\newcommand{\indnoarg}{\mathbf 1} 
\newcommand{\indgp}{\overline{\mathbb I}}
\newcommand{\GP}{\mathsf{GP}}
\newcommand{\ads}{\mathsf{AdS}}
 
\newcommand{\W}{\mathfrak S^B_n}

\newcommand{\face}[2]{\operatorname{face}_{#1}#2}

\newcommand{\phimap}{\phi^B}
\newcommand{\zetamap}{\zeta^B}

\newcommand{\newword}[1]{\textbf{#1}}

\newcommand{\Oo}{\mathcal{O}(1_o)}
\newcommand{\Oinfty}{\mathcal{O}(1_\infty)}

\theoremstyle{definition}
\newtheorem{thm}{Theorem}[section]
\newtheorem{cor}[thm]{Corollary}
\newtheorem{lem}[thm]{Lemma}
\newtheorem{prop}[thm]{Proposition}
\newtheorem{defn}[thm]{Definition}
\newtheorem{propdefn}[thm]{Proposition-Definition}
\newtheorem{conj}[thm]{Conjecture}
\newtheorem{eg}[thm]{Example}
\newtheorem{rem}[thm]{Remark}

\newtheorem{maintheorem}{Theorem}

\numberwithin{equation}{section}

\title{Signed permutohedra, delta-matroids, and beyond}
\author{Christopher Eur, Alex Fink, Matt Larson, Hunter Spink}

\address{Harvard University. Cambridge, MA. USA.}
\email{ceur@math.harvard.edu}

\address{School of Mathematical Sciences, Queen Mary University of London. London E1 4NS. UK}
\email{a.fink@qmul.ac.uk}

\address{Stanford University. Stanford, CA. USA.}
\email{mwlarson@stanford.edu}

\address{University of Toronto.  Toronto, ON. Canada.}
\email{hunter.spink@utoronto.ca}


\begin{document}

\newpage

\maketitle

\begin{abstract}
We establish a connection between the algebraic geometry of the type $B$ permutohedral toric variety and the combinatorics of delta-matroids.
Using this connection, we compute the volume and lattice point counts of type $B$ generalized permutohedra.
Applying tropical Hodge theory to a new framework of ``tautological classes of delta-matroids,'' modeled after certain vector bundles associated to realizable delta-matroids,
we establish the log-concavity of a  Tutte-like invariant for a broad family of delta-matroids that includes all realizable delta-matroids.
Our results include new log-concavity statements for all (ordinary) matroids as special cases.
\end{abstract}

\setcounter{tocdepth}{1}
\tableofcontents

\section{Introduction}

For a nonnegative integer $n$, let $[n] = \{1, \ldots, n\}$.  For a subset $S\subseteq [n]$, let $\be_S = \sum_{i\in S} \be_i \in \RR^n$ be the sum of the standard basis vectors indexed by $S$.
If $n\ge1$, the \textbf{$\boldsymbol{A_{n-1}}$ permutohedral fan} $\Sigma_{A_{n-1}}$ is the complete fan in $\RR^n$ whose maximal cones are the chambers of the arrangement of hyperplanes
\[
H_{\be_i - \be_j} = \{(x_1, \ldots, x_n) \in \RR^n : x_i - x_j = 0\} \quad\text{for all $1\leq i<j\leq n$}.
\]
A polytope $P\subset \RR^n$ is an \textbf{$\boldsymbol{A_{n-1}}$ generalized permutohedron} if its normal fan coarsens the fan $\Sigma_{A_{n-1}}$.
The polyhedral properties of $A_{n-1}$ generalized permutohedra and the algebraic geometry of the toric variety $X_{A_{n-1}}$ associated to $\Sigma_{A_{n-1}}$ (as a fan in $\mathbb R^n/\RR(1,\ldots,1)$) have been well studied as a way to illuminate the structure of several combinatorial objects \cite{Pos09, AA}, 
including graphs, posets, and, notably in recent years, matroids.

\begin{defn}

A \textbf{matroid} $\M$ on $[n]$ is a nonempty collection $\mathcal B$ of subsets of $[n]$, called the \textbf{bases} of $\M$, such that the polytope
\[
P(\M) = \text{the convex hull of } \{\be_B : B\in \mathcal B\} \subset [0,1]^n,
\]
has all edges parallel translates of $\be_i-\be_j$ for various $i,j\in [n]$, or, equivalently, 
such that $P(\M)$ is an $A_{n-1}$ generalized permutohedron with all vertices lying in $\{0,1\}^n$.
\end{defn}

Recently, an interpretation of matroids as elements in the Chow cohomology ring of $X_{A_{n-1}}$ has led to fruitful developments in matroid theory \cite{HK12, AHK18, BST20, LdMRS20}.  Conversely, this interpretation allows matroid theory to inform the geometry of  $X_{A_{n-1}}$ \cite{Ham17, EHL}.  Many of these developments have recently been unified, recovered, and extended under the new framework of ``tautological classes of matroids'' \cite{BEST}, modeled after certain torus-equivariant vector bundles on $X_{A_{n-1}}$.

Meanwhile, the fan $\Sigma_{A_{n-1}}$ generalizes to the fan $\Sigma_\Phi$ 
of the Coxeter arrangement of an arbitrary crystallographic root system $\Phi$, the toric variety $X_{A_{n-1}}$ generalizes to the toric variety $X_{\Phi}$ of $\Sigma_\Phi$, and the combinatorial objects such as graphs, posets, and matroids generalize appropriately to their Coxeter analogues (see \cite[\S4]{ACEP20} and references therein). For instance, in the theory of Coxeter matroids \cite{BGW}, matroids in the usual sense are exactly the type $A$ minuscule Coxeter matroids. 
Several works \cite{Pro90, DL94, Ste94, Kly95} have studied the Chow cohomology ring of $X_\Phi$.
Missing in these previous works is an interaction between Coxeter matroids and the Chow cohomology ring of $X_\Phi$ that generalizes the interaction between matroids and the Chow cohomology ring of $X_{A_{n-1}}$.

We establish here such an interaction when $\Phi$ is a root system of type $B$, noting that the type $B$ minuscule Coxeter matroids are exactly delta-matroids (\Cref{defn:deltamat}).
This interaction interfaces particularly well with the framework of ``tautological classes of delta-matroids'' we develop in \Cref{sec:tauto}, 
which is modeled after toric vector bundles associated to maximal isotropic subspaces that realize delta-matroids.
Some barriers to establishing a uniform treatment for arbitrary Coxeter types can be found in Remark~\ref{rem:barriers}.

\subsection{Main combinatorial consequences}
\begin{defn}\label{defn:Bn}
Let $n\ge0$.
The \textbf{$\boldsymbol{B_n}$ permutohedral fan} $\Sigma_{B_n}$ is the complete fan in $\RR^n$ whose maximal cones are the chambers of the arrangement of hyperplanes
\begin{align*}
H_{\be_i \pm \be_j} &= \{(x_1, \ldots, x_n) \in \RR^n : x_i \pm x_j = 0\} \quad\text{for all $i\neq j \in [n]$, and}\\
H_{\be_i} &= \{(x_1, \ldots, x_n) \in \RR^n : x_i = 0\} \quad\text{for all $i \in [n]$}.
\end{align*}
\end{defn}

The fan $\Sigma_{B_n}$ is the normal fan of the type $B_n$ permutohedron $\Pi_{B_n}$, also called the signed permutohedron, which is the convex hull of $\{w\cdot (n, \ldots, 1) \in \RR^n: w\in \W\}$, where $\W$ is the signed permutation group (see \S\ref{ssec:FanBn}).
A polytope $P\subset \RR^n$ is a \textbf{$\boldsymbol{B_n}$ generalized permutohedron} if its normal fan $\Sigma_P$ coarsens $\Sigma_{B_n}$, or, equivalently, if each edge of $P$ is parallel to $\be_i + \be_j$, $\be_i - \be_j$, or $\be_i$ for various $i,j \in [n]$. $B_n$ generalized permutohedra are also known as bisubmodular polytopes, see \cite[Theorem 1]{FujishigeParametric}.

\medskip
A celebrated result of Postnikov \cite{Pos09} gives a formula for the volumes and lattice point enumerators of $A_{n-1}$ generalized permutohedra in terms of \textbf{transversals} of subsets $S_1, \ldots, S_k$ of $[n]$, i.e., subsets $\tau \subseteq [n]$ such that there exist a bijection $j\colon \{1, \ldots, k\} \to \tau$ with $j(i) \in S_i$ for all $i\in \{1, \ldots, k\}$.
We give a formula for the volumes and lattice point enumerators of $B_n$ generalized permutohedra as follows.

\medskip
Let $[\overline n] = \{\overline 1, \ldots, \overline n\}$, and let $[n,\overline n] = [n] \sqcup [\overline n]$, which is endowed with the obvious involution $\overline{( \cdot  )}$.  For $S\subseteq [n,\overline n]$, we denote $\be_S = \sum_{i\in S} \be_i$, where $\be_{\overline j} := -\be_j$ for $j\in [n]$.
Define the set $\ads$ of \textbf{admissible subsets} of $[n,\overline n]$ to be 
\begin{align*}
\ads &= \{\text{$S\subset [n,\overline n]$ such that $\{i,\overline i\} \not\subseteq S$ for all $i\in [n]$}\},\text{ and define }
\ads_n =\{S\in \ads: |S|=n\}
\end{align*}
to be the set of maximal admissible subsets. 
A \textbf{signed transversal} of $S_1, \ldots, S_n$ is an admissible subset $\tau\in \ads_n$ such that there exists a bijection $j\colon  \{1,\ldots,n\} \to \tau$ with $j(i) \in S_i$ for all $i = 1, \ldots, n$.
For an admissible subset $S\in \ads$, let
\[
\Delta_S^0 = \text{the simplex which is the convex hull of $\{\be_i : i\in S\} \cup \{\mathbf 0\}$ in $\RR^n$}.
\]

\begin{maintheorem}\label{mainthm:BnGP}
Let $P$ be a lattice $B_n$ generalized permutohedron (i.e., $P$ has vertices in $\ZZ^n$).
\begin{enumerate}[label = (\alph*)]
\item\label{BnGP:1} There exists a unique set of integers $\{c_S \in \ZZ : S\in \ads\setminus\{\emptyset\}\}$
such that the signed Minkowski sum $\displaystyle \sum_{S\in \ads\setminus\{\emptyset\}} c_S \Delta_S^0$ equals $P$. Hence we may write $P=P(\{c_S\})$.
\item \label{BnGP:2}For any sequence $(S_1, \ldots, S_n)$ of nonempty admissible subsets of $[n,\overline n]$, one has that
\[
\text{mixed volume of } \{\Delta_{S_1}^0, \ldots, \Delta_{S_n}^0\} = |\{\text{signed transversals of $S_1, \ldots, S_n$}\}|.
\]
In particular, normalizing the volume of the standard simplex $\Delta^0_{[n]}$ to be $1$, one has
\begin{align*}
    \operatorname{Vol}\big(P(\{c_S\}) \big)= \sum_{(S_1, \ldots, S_n)} |\{\text{signed transversals of $S_1, \ldots, S_n$}\}| \cdot c_{S_1}c_{S_2}\cdots c_{S_n}
\end{align*}
where the sum is over all sequences $(S_1, \ldots, S_n)$ of nonempty admissible subsets.
\item \label{BnGP:3}Let $\Psi$ be the linear operator on polynomials that replaces each monomial $x_1^{d_1}\cdots x_m^{d_m}$ in a polynomial $f(x_1, \ldots, x_m)$ by $\frac{d_1!\cdots d_m!}{(d_1+\dots + d_m)!}\binom{x_1}{d_1}\cdots\binom{x_m}{d_m}$.  Let $\square = [0,1]^n$ be the standard unit cube in $\mathbb{R}^n$.  Then, we have
\begin{align*}
    \text{\# lattice points of } \big(P(\{c_S\})-\square\big)=\Psi \Big(\operatorname{Vol}\big(P(\{c_S\}) \big) \Big),
\end{align*}
where $P(\{c_S\}) - \square$ denotes the polytope $P(\{c'_S\})$ with $c'_{S} = c_{S}-1$ if $S = \{i\} \subseteq [n]$ and $c'_{S} = c_{S}$ otherwise.  Here, the volume and lattice point counts are considered as polynomials in the $\{c_S\}$.
\end{enumerate}
\end{maintheorem}

The statements \ref{BnGP:1}, \ref{BnGP:2}, and \ref{BnGP:3} generalize to type $B$ the classical type $A$ results \cite[Proposition 2.4]{ABD10}, \cite[Theorem 9.3]{Pos09}, and \cite[Theorem 11.3]{Pos09}, respectively.  
Hence, \Cref{mainthm:BnGP} fully answers \cite[Question 9.3]{ACEP20} for type $B$.  The statement \ref{BnGP:1} was also shown in \cite{Bas21} via a study of Tits algebras, and a different set of polytopes satisfying the property in \ref{BnGP:1} was obtained in \cite{PPR23} via a study of shard polytopes.  Neither work gives a formula for the volume or lattice point enumerator.
We will deduce \Cref{mainthm:BnGP} via our study of delta-matroids.

\begin{defn}\label{defn:deltamat}
A \textbf{delta-matroid} $\D$ on ground set $[n,\overline n]$ is a nonempty collection $\mathcal F\subseteq \ads_n$ of admissible subsets of $[n, \overline n]$ of cardinality $n$, called the \newword{feasible sets} of $\D$, such that the polytope
\[
P(\D) = \text{the convex hull of } \{ \textstyle  \be_{B \cap [n]} : B\in \mathcal F \} \subset [0,1]^n
\]
has all edges parallel translates of $\be_i+\be_j$, $\be_i - \be_j$, or $\be_i$ for various $i,j\in [n]$, or, equivalently, 
such that $P(\D)$ is a $B_{n}$ generalized permutohedron with all vertices lying in $\{0,1\}^n$.
For $i\in [n]$, we say that $i$ is a \newword{loop}, resp.\ \newword{coloop}, of $\D$ if no, resp.\ every, feasible set contains $i$.
\end{defn}

We often identify a delta-matroid $\D$ with its polytope $P(\D)$. 

\smallskip
Delta-matroids were introduced in \cite{Bou87} by weakening the basis exchange axiom for matroids, to allow cases where not all bases have the same cardinality.
(A basis of~$\D$ is the intersection of a feasible set with~$[n]$).
Several combinatorial settings that give rise to matroids have generalizations to delta-matroids.
As one example, a bipartite graph yields a transversal matroid whose bases come from maximal matchings, as the incident vertices in one part.
Given an arbitrary graph, the sets of vertices incident to matchings of any size are the bases of a delta-matroid \cite{Bou89}.
As another, a connected graph yields a graphic matroid whose bases are the spanning trees.
Given a graph embedded on a surface, the set of spanning ``quasi-trees'' are the bases of a delta-matroid \cite{CMNR19a, CMNR19b}: see Example~\ref{eg:ribbon}.
There is a theory of linear representability for delta-matroids as well: see \Cref{ssec:representable}.
For the equivalence of the definition of delta-matroids in the works cited above and the one given here, see \cite[Ch.~4]{BGW}.

\medskip
A matroid $\M$ on $[n]$ with set of bases $\mathcal  B$ defines a delta-matroid $\D$ in two different ways:
first, by its base polytope $P(\M)$, 
and, second, by its independence polytope
\[
IP(\M) = \text{the convex hull of }(\be_I : I \subseteq [n] \text{ such that } I \subseteq B \text{ for some $B\in \mathcal B$}) \subset [0,1]^n,
\]
whose edges are all of the form $\be_i$ or $\be_i - \be_j$. 
We will frequently use $P(\M)$ and $IP(\M)$ to refer to the delta-matroids obtained from $\M$ as above.

\medskip
We introduce a new invariant of delta-matroids defined by a recursive relation similar to the one satisfied by Tutte polynomials of matroids.
See \Cref{defn:operations} for the deletion $\D\setminus i$, contraction $\D/i$, and projection $\D(i)$ of a delta-matroid $\D$. 

\begin{defn}\label{def:recursiveu}
For a delta-matroid $\D$ on $[n,\overline n]$ with feasible sets $\mathcal F$, the \newword{$\boldsymbol U$-polynomial} $U_{\D}(u,v)$ is the unique bivariate polynomial satisfying the properties: 
\begin{itemize}
\item (Base case) 
If $n=0$, then $U_\D(u,v) = 1$.
\item (Recursive relation) If $n\geq 1$ and $i\in [n]$, then
\[
U_\D(u,v) = \begin{cases}
U_{\D\setminus i}(u,v) + U_{\D/ i}(u,v) + uU_{\D(i)}(u,v) & \text{if $i$ is neither a loop nor a coloop}\\
(u+v+1) \cdot U_{\D\setminus i}(u,v) & \text{if $i$ is a loop or a coloop}.\\
\end{cases}
\]
\end{itemize}
\end{defn}

 \Cref{prop:explicitu} verifies that this recursive definition is well-defined.
Specializing $U_\D(u,v)$ at $u=0$, one obtains the \newword{interlace polynomial} $\operatorname{Int}_{\D}(v)$, introduced in \cite{InterlaceBollobas} for graphs and generalized to delta-matroids in \cite{BrijderInterlace}.
See \cite{MorseInterlace} for a survey on interlace polynomials.\footnote{In our terms the ``interlace polynomial'' defined in \cite{InterlaceBollobas} equals $\operatorname{Int}_{\D}(v-1)$.  Our definition agrees with \cite[Definition 28]{MorseInterlace} and the polynomial denoted $q_1$ in \cite{BrijderInterlace}.}
The invariant $U_\D$ also gives rise to two invariants of (ordinary) matroids.  Let $T_\M$ denote the Tutte polynomial of $\M$.
One computes, as done in Examples~\ref{ex:indeppoly} and~\ref{ex:basepolytope}, that
\[U_{P(\M)}(u,v) = \sum_{T\subseteq S\subseteq [n]} u^{|S-T|} v^{\operatorname{corank}_\M(S) + \operatorname{nullity}_\M(T)},\]
so in particular $\operatorname{Int}_{P(\M)}(v) = T_\M(v+1,v+1)$,
and 
$$U_{IP(\M)}(u,v) = (u + 1)^{n - \operatorname{rank}(\M)}\, T_\M\left(u + 2, \frac{u+v+1}{u+1} \right).$$

\medskip
We establish a log-concavity property for $U$-polynomials of delta-matroids which have an \newword{enveloping matroid} (Definition~\ref{def:enveloping}), a condition necessary for applying tools from the tropical Hodge theory developed in \cite{ADH}. Such delta-matroids include $P(\M)$ and $IP(\M)$ when $\M$ is a matroid (\Cref{prop:matroidenveloping}), and include realizable delta-matroids (\Cref{prop:envelopingtypeB}), in particular the adjacency delta-matroids of graphs (\Cref{eg:adjacency}) and delta-matroids from graphs embedded on surfaces (\Cref{eg:ribbon}).
We say that the coefficients of a homogeneous polynomial $f$ of degree $d$ form a \newword{log-concave unbroken array} if for any $1 \le i < j \le n$ and any monomial $x^{\mathbf{m}}$ of degree $d' \le d$, the coefficients of $\{x_i^k x_j^{d - d' - k} x^{\mathbf{m}}\}$ form a nonnegative log-concave sequence with no internal zeros. 

\begin{maintheorem}\label{mainthm:logconc}
Let $\D$ be a delta-matroid which has an enveloping matroid. Then the polynomials
\begin{equation}\label{eq:isotropiclorentzian}
(y + q)^n\, U_{\D}\left(\frac{x}{y + q}, \frac{y - q}{y + q}\right)  \quad \text{and } 
\end{equation}
\begin{equation}\label{eq:envelopinglorentzian}
(y + w)^n\, U_{\D}\left(\frac{2z + x}{y + w}, \frac{y - z}{y + w}\right)
\end{equation}
have a log-concave unbroken array of coefficients.  In fact, they are denormalized Lorentzian polynomials in the sense of \cite{BH, BLP}.
\end{maintheorem}

In fact, we obtain that (\ref{eq:isotropiclorentzian}) is denormalized Lorentzian by showing that a specialization of a multivariable version of the $U$-polynomial is Lorentzian, which gives stronger log-concavity results. See Theorem~\ref{thm:strengthened}. 
Setting $x=0$ and $q=1$ in \eqref{eq:isotropiclorentzian} implies that the transformation $(y+1)^n\operatorname{Int}_\D(\frac{y-1}{y+1})$ of the interlace polynomial has nonnegative log-concave coefficients with no internal zeros, and hence has unimodal coefficients.  
We note that the interlace polynomial of a realizable delta-matroid can have non-unimodal coefficients (\Cref{eg:notunimodal}); see \Cref{rem:unimodalhistory} for a history of conjectures about unimodality for the interlace polynomial.
\Cref{mainthm:logconc} and Theorem~\ref{thm:strengthened} also yield new log-concavity results for (ordinary) matroids.  For instance, Theorem~\ref{thm:strengthened} implies that the coefficients of $U_{\D}(u, 0)$ are log-concave after multiplying the coefficient of $u^k$ by $k!$, and in particular are strictly log-concave. Taking $\D = P(\M)$ for a matroid $\M$, this implies that if we set
$$a_k = |\{T \subseteq S \subseteq [n] \colon T \text{ independent in $\M$ and $S$ spanning in $\M$, } |S| - |T| = k\}|,$$
then $a_k^2 \ge \frac{k+1}{k} a_{k-1}a_{k+1}$.
See \Cref{cor:logconcave} for more implications of \Cref{mainthm:logconc}. 
See Theorems~\ref{thm:intersectisotropic} and~\ref{thm:intersectenveloping} for the algebro-geometric results underlying the formulas (\ref{eq:isotropiclorentzian}) and~(\ref{eq:envelopinglorentzian}) respectively,
and see \S\ref{sec:logconc} for the derivation of log-concavity from these formulas using tropical Hodge theory. 
\begin{conj}\label{conj:noenveloping}
The hypothesis that $\D$ has an enveloping matroid can be removed in \Cref{mainthm:logconc}.
\end{conj} 
We do not know an easy way to check if a given delta-matroid has an enveloping matroid, so it is difficult to test Conjecture~\ref{conj:noenveloping}. We have checked Conjecture~\ref{conj:noenveloping} for all delta-matroids on at most $5$ elements, which includes some delta-matroids which lack enveloping matroids: see Example~\ref{eg:nonenvelopable}.

\subsection{Underlying geometry}
We obtain Theorems~\ref{mainthm:BnGP} and~\ref{mainthm:logconc} by establishing a new connection between the algebraic geometry of the $B_n$ permutohedral fan $\Sigma_{B_n}$ and the combinatorics of delta-matroids.
The fan $\Sigma_{B_n}$, as a rational fan over $\ZZ^n$, defines a smooth projective toric variety $X_{B_n}$ which we call the \textbf{$\boldsymbol{B_n}$-permutohedral variety}.
We follow the conventions in \cite{Ful93, CLS} for toric varieties and polyhedra, and we work over an algebraically closed field~$\kk$.
The toric variety $X_{B_n}$ is equipped with two well-studied rings, the Chow cohomology ring $A^\bullet(X_{B_n})$ and the Grothendieck ring of vector bundles $K(X_{B_n})$.

\medskip
We construct an isomorphism between the rings $K(X_{B_n})$ and $A^\bullet(X_{B_n})$, different from the classical Hirzebruch--Riemann--Roch  theorem. Recall that the Hirzebruch--Riemann--Roch theorem states that for an arbitrary smooth projective variety $X$, the Chern character map $ch \colon K(X)\otimes \mathbb{Q}\overset\sim\to A^\bullet(X)\otimes \QQ$ is an isomorphism such that 
\[
\chi([\mathcal E]) = \int_X ch([\mathcal E]) \cdot \operatorname{Td}(X) \quad \text{for all $[\mathcal E]\in K(X)$},
\]
where $\chi \colon K(X)\to \mathbb{Z}$ is the sheaf Euler characteristic map, $\int_X$ is the degree map, and $\operatorname{Td}(X)\in A^\bullet(X)\otimes \mathbb{Q}$ is the Todd class of $X$.
  
To state our exceptional Hirzebruch--Riemann--Roch-type theorem, we need the following definitions.  Note that the product fan $(\Sigma_{B_1})^n$, which is the fan induced by the arrangement of  coordinate hyperplanes in $\RR^n$, is a coarsening of $\Sigma_{B_n}$.  Hence, since the toric variety of $\Sigma_{B_1}$ is $\PP^1$, we have a birational toric morphism $X_{B_n} \to (\PP^1)^n$.  Let $\boxplus\mathcal O(1)$ be the vector bundle on $X_{B_n}$ obtained as the direct sum of the pullbacks of $\mathcal O_{\PP^1}(1)$ from each $\PP^1$ factor in the product $(\PP^1)^n$.

\begin{maintheorem}\label{mainthm:HRR}
There exists a ring isomorphism $\phimap\colon  K(X_{B_n}) \to A^\bullet(X_{B_n})$ such that
\[
\chi([\mathcal E]) = \int_{X_{B_n}} \phimap([\mathcal E]) \cdot c(\boxplus\mathcal O(1)) \quad \text{for all $[\mathcal E]\in K(X_{B_n})$},
\]
where $c(\boxplus\mathcal O(1)) = c_0(\boxplus \mathcal{O}(1)) + \dotsb + c_n(\boxplus \mathcal{O}(1))$ denotes total Chern class of $\boxplus\mathcal O(1)$.
\end{maintheorem}

We define the map $\phimap$ and prove \Cref{mainthm:HRR} in \S\ref{sec:HRR}. We note that the map $\phimap$ in \Cref{mainthm:HRR} differs from $ch$ and is an isomorphism integrally, and the class $c(\boxplus \mathcal O(1))$ differs from the Todd class of $X_{B_n}$.  The isomorphism $\phimap$ here is closely related to the type $A$ exceptional Hirzebruch--Riemann--Roch isomorphisms that appeared in \cite{BEST} and \cite{EHL} (see \S\ref{subsec:fakeHRR}).

\medskip
The combinatorial utility of \Cref{mainthm:HRR} is mediated by our \Cref{mainthm:isoms} that describes a basis of the ring $K(X_{B_n})$ in terms of Schubert delta-matroids (\Cref{defn:SchubertDelta}), which correspond to the Bruhat cells of a type $B$ generalized flag variety (\Cref{rem:geometricSchubert}). Recall that there is a standard correspondence between polytopes and base-point-free line bundles on toric varieties \cite[\S6.2]{CLS}.

\begin{maintheorem}\label{mainthm:isoms}
\label{thm:basis} 
The classes of line bundles on $X_{B_n}$ corresponding to the polytopes of Schubert delta-matroids without coloops form a basis for $K(X_{B_n})$. 
\end{maintheorem}

\Cref{mainthm:isoms} is proved in \Cref{sec:D}.
By combining  \Cref{mainthm:HRR} with \Cref{mainthm:isoms}, 
we construct in Corollary~\ref{cor:basisChow} a graded basis for $A^\bullet(X_{B_n})$ indexed by coloop-free Schubert delta-matroids.
By considering the basis elements in  $A^1(X_{B_n})$, we deduce statement \ref{BnGP:1} of \Cref{mainthm:BnGP}.  The rest of \Cref{mainthm:BnGP} is deduced from \Cref{mainthm:HRR} in \S\ref{subsec:volEhr}. Theorem~\ref{mainthm:logconc} is proved by constructing torus-equivariant nef vector bundles on $X_{B_n}$ which are related to delta-matroids; see \S\ref{ssec:isotropic taut} and~\S\ref{ssec:envelopingtaut}. 
The proof of Theorem~\ref{mainthm:logconc} invokes Theorem~\ref{mainthm:HRR} in~\S\ref{ssec:intersection} to compute certain intersection numbers. 
Their log-concavity properties are established using tropical Hodge theory in \S\ref{sec:logconc}. 

\subsection*{Acknowledgments}
We thank Steven Noble for pointing out Example~\ref{eg:nonenvelopable}. We thank the referee for their helpful comments. 
The first author is partially supported by the US National Science Foundation (DMS-2001854).
The third author is supported by an NDSEG fellowship.

\section{Polytope algebras of delta-matroids}\label{sec:D}
In this section, we prove \Cref{mainthm:isoms}, which describes $K(X_{B_n})$ in terms of delta-matroids.
\Cref{ssec:FanBn} sets up preliminaries on the fan $\Sigma_{B_n}$ and signed permutation group $\W$.
The first step of the proof of Theorem~\ref{mainthm:isoms} is that $K(X_{B_n})$ is isomorphic to a combinatorially defined ring, 
the \textbf{polytope algebra} $\indgp(\Sigma_{B_n})$ of indicator functions of lattice $B_n$ generalized permutohedra modulo translation, introduced in \Cref{ssec:polytope algebra}.
This is a special case of the folklore statement that $K(X_\Sigma)$ is isomorphic to a polytope algebra for an arbitrary smooth projective fan $\Sigma$, proven precisely in \cite[Appendix A]{EHL}.
The isomorphism sends the class $[\ind P]$ of the indicator function of a $B_n$ generalized permutohedron $P$ 
to the $K$-class of the corresponding line bundle. 

\Cref{ssec:Schubert} introduces Schubert delta-matroids.
\Cref{ssec:unitcube} contains the bulk of the proof of Theorem~\ref{mainthm:isoms},
and \Cref{ssec:bases from Schubert} assembles it.
The proof proceeds in three main steps. 
Using polyhedral properties special to the unit cube $[0,1]^n$, we show that 
the intersection of a lattice $B_n$ generalized permutohedron with the cube is a delta-matroid polytope (\Cref{prop:intersect with a cube});
tiling by translates of this cube, we conclude that $\indgp(\Sigma_{B_n})$ is generated by classes of delta-matroid polytopes.
Intersecting the cube with the dual of a cone of $\Sigma_{B_n}$ gives a Schubert delta-matroid polytope (\Cref{cor:positive root cone}),
which up to translation may be taken to be coloop-free;
using the Brianchon--Gram theorem, these intersections by themselves generate $\indgp(\Sigma_{B_n})$ (\Cref{thm:disect}). 
The last step is to show that Schubert delta-matroid polytopes satisfy no linear relations (\Cref{prop:ESS} and the sequel).

\subsection{The fan $\Sigma_{B_n}$ and the signed permutation group $\W$}
\label{ssec:FanBn}
Let $n$ be a nonnegative integer.
Recall that the $B_n$ permutohedral fan $\Sigma_{B_n}$ was defined to be the complete  fan in $\RR^n$ whose maximal cones are the chambers of the type $B$ arrangement of hyperplanes, the union of all hyperplanes of the form $\{x_i\pm x_j=0\}$ and $\{x_i=0\}$.
\begin{defn}
The Weyl reflection group corresponding to the real hyperplane arrangement defining~$\Sigma_{B_n}$ is the \newword{signed permutation group} $\W$, which is the subgroup
\[
\W = \{w\in \mathfrak S_{[n,\overline n]} : w(\overline i ) = \overline{w(i)} \text{ for all }i\in [n,\overline n]\} \subset \mathfrak S_{[n,\overline n]},
\]
where $\mathfrak S_{[n,\overline n]}$ denotes the symmetric group on $[n,\overline n]$.

A permutation $\sigma$ of $[n]$ can be extended to a signed permutation of $[n,\overline{n}]$ by setting $\sigma(\overline{i})=\overline{\sigma(i)}$. In this way, the permutation group $\mathfrak S_n$ is naturally a parabolic subgroup of $\W$, viewed as the stabilizer of $[n]\subset [n,\overline{n}]$. Then $\W$ is a semidirect product 
\begin{equation*}\label{eq:semidirect}
\W=\mathfrak S_n \ltimes \{\pm 1\}^n,
\end{equation*}
where $\{\pm 1\}^n\trianglelefteq \W$ is the \newword{sign group} such that the $i\/$th copy of $\{\pm 1\}$ is the subgroup generated by the transposition $(i,\overline{i})$. 
We denote the map to the set of left cosets of~$\mathfrak S_n$ by
$$(\epsilon_1,\ldots,\epsilon_n) \colon \W\to \{\pm 1\}^n,$$
which can also be described by
$$\epsilon_i(w)=\begin{cases}1&i \in w([n])\\ -1 & i \not \in w([n]).\end{cases}$$
\end{defn}
Recall that we have defined $\be_{\bar{i}}=-\be_i\in \mathbb{R}^n$ for $i \in [n]$.
We next fix notation for cones of~$\Sigma_{B_n}$.
\begin{prop}\label{prop:B_n cones}
The maximal cones of $\Sigma_{B_n}$ are given by $$C_w=\operatorname{cone}\{\be_{w(1)},\ldots ,\be_{w(1)}+\cdots +\be_{w(n)}\}$$
for each $w \in \W$.
The cone $C_w$ is the unique maximal cone containing $w\cdot (n,\ldots, 1)$. The dual cones are given by
$$C_w^{\vee}=\operatorname{cone}\{\be_{w(1)},\be_{w(2)}-\be_{w(1)},\ldots, \be_{w(n)}-\be_{w(n-1)}\}.$$
\end{prop}
We describe here the various (left) actions of $\W$ we will consider.
\begin{itemize}
    \item $\W$ acts on $\mathbb{R}^n$ by $w\cdot \be_i=\be_{w(i)}$. This is the geometric definition of the Weyl group as the set of isometries preserving the type $B$ hyperplane arrangement. 
    \item $\W$ acts on the set of maximal cones of $\Sigma_{B_n}$ through its action on $\mathbb{R}^n$ by $w\cdot C_{w'}=C_{ww'}$.
    \item $\W$ acts on the set of delta-matroids $\D$ through the action on the ground set $[n, \overline{n}]$.
    \item $\W$ acts on the set of delta-matroid polytopes $P(\D)$ through its action on the set of delta-matroids. 
    This is \emph{not} induced by the above $\W$-action on $\mathbb{R}^n$ (which does not preserve the cube $[0,1]^n$ containing all delta-matroid polytopes), but rather the $\W$-action on $\mathbb{R}^n$ conjugated by translation by $(-\frac{1}{2},\ldots,-\frac{1}{2})$. Hence $\mathfrak S_n$ acts  in the usual way by permuting coordinates, but the $i$th copy of $\{\pm 1\}$ in the sign group acts by reflection in the $x_i=\frac{1}{2}$ hyperplane.
\end{itemize}

\begin{rem}
The orbit of a delta-matroid under $\mathfrak S_n\le\W$ consists of all isomorphic delta-matroids in the sense usual in the delta-matroid literature.
Its orbit under $\{\pm1\}^n\trianglelefteq\W$ are called its \newword{partial duals} \cite{Chmutov2009}.
So its $\W$-orbit consists of all partial duals of isomorphic delta-matroids.
\end{rem}

\subsection{The polytope algebra}\label{ssec:polytope algebra}
We collect some facts about McMullen's polytope algebra; see \cite[Appendix A]{EHL} for a survey and references.
For a polyhedron $P\subseteq\mathbb R^n$, possibly unbounded,
let $\ind P\colon \mathbb R^n\to\mathbb Z$ be its indicator function, 
defined so that $\ind P(x)$ equals 1 if $x\in P$ and 0 if~not.
Let $\mathscr P$ be a collection of polyhedra in $\mathbb R^n$.

\begin{defn}\label{def:strongly valuative}
The \newword{indicator group} $\mathbb I(\mathscr P)$ is the group of functions from $\mathbb R^n$ to $\mathbb Z$
generated by the indicator functions $\ind P$ for $P\in\mathscr P$.
A function $f\colon \mathscr P\to G$ valued in an abelian group~$G$ is called \newword{strongly valuative} if it factors through the map $\indnoarg\colon \mathscr P\to\mathbb I(\mathscr P)$.
\end{defn}
Let $\mathbb Z^n+\mathscr P = \{m+P : m\in \mathbb Z^n,\ P\in\mathscr P\}$ be the set of lattice translates of polyhedra in $\mathscr P$.

\begin{defn}\label{def:polytope algebra}
The \newword{translation-invariant indicator group} $\indgp(\mathscr P)$ is the quotient
\[\indgp(\mathscr P) = \mathbb I(\mathbb Z^n+\mathscr P)/(\ind{m+P}-\ind P : m\in \mathbb Z^n,P \in\mathscr P).\]
\end{defn}
We write $[f]$ for the class of a function $f\in\mathbb I(\mathbb Z^n+\mathscr P)$ in this quotient.
For a polyhedron $P\in \mathscr P$, we often write $[P]$ for the class $[\ind{P}]$.

\medskip
Suppose now that $\mathscr P$ is the set $\mathscr P_{\ZZ,\Sigma}$ of lattice \textbf{deformations} of a smooth projective fan $\Sigma$ in $\RR^n$, that is,
$
\mathscr P_{\ZZ,\Sigma} = \{ P \subset \RR^n \text{ a lattice polytope whose normal fan coarsens $\Sigma$}\}.
$
In this case, the group $\indgp(\mathscr P_{\ZZ,\Sigma})$ is isomorphic to the subalgebra of McMullen's polytope algebra spanned by polytopes in $\mathscr P_{\ZZ,\Sigma}$ \cite[Proposition A.6]{EHL} (see also \cite{McMullen2009}). 
In particular, $\indgp(\mathscr P_{\ZZ,\Sigma})$ acquires the structure of a unital commutative ring \cite[Lemma 6]{McMullen1989}, with the product induced by $[P] \cdot [Q] = [P+Q]$.

\medskip
The polytope algebra $\indgp(\mathscr P_{\ZZ,\Sigma})$ relates to the geometry of the smooth projective toric variety $X_\Sigma$ of the fan $\Sigma$ as follows.
The standard correspondence between polyhedra and divisors on toric varieties \cite[\S6.2]{CLS} (see also \cite[\S2.4]{ACEP20}) gives a bijection between polytopes $P\in \mathscr P_{\ZZ,\Sigma}$ and 
base-point-free torus-invariant divisors $D_P$ on $X_\Sigma$.  Let $\mathcal O_{X_\Sigma}(D_P)$ denote the corresponding line bundle.  We then have the following folklore isomorphism.

\begin{thm}\label{thm:indK}\cite[Theorem A.10]{EHL} (cf.\ \cite[Theorem 8]{Mor93}).
The assignment $[P] \mapsto [\mathcal O_{X_\Sigma}(D_P)]$ defines an isomorphism of rings $\indgp(\mathscr P_{\ZZ,\Sigma}) \overset\sim\to K(X_\Sigma)$.
\end{thm}

We now specialize to the $B_n$ permutohedral fan.
Let 
\[
\GP_{\mathbb Z,B_n} = \mathscr P_{\ZZ,\Sigma_{B_n}}
\] 
be the set of $B_n$ generalized permutohedra that are lattice polytopes.
Then 
\[
\mathsf{DMat}_n  = \text{the set of all delta-matroids on $[n,\overline n]$}
\]
is identified with the subset of $\GP_{\mathbb Z,B_n}$ consisting of polytopes with vertices in $\{0,1\}^n$.

\subsection{Schubert delta-matroids}\label{ssec:Schubert}

We now describe a special family of delta-matroids that we will use to provide bases for $\mathbb I(\GP_{\ZZ,B_n})$ and $\indgp(\GP_{\ZZ,B_n})$.  
By identifying $w\in \W/\mathfrak{S}_n$ with $w\cdot [n]\in \ads_n$, the Bruhat order provides a partial order on $\ads_n$,
namely the (hyperoctahedral) Gale order of \cite[\S3.1.2]{BGW}, given as follows.
Endow $[n,\overline n]$ with the total order
\begin{equation}\label{eq:order on J}
\overline n<\cdots<\overline 1<1<\cdots<n.
\end{equation}
Then, the Gale order on~$\ads_n$ is the corresponding \newword{dominance order}, which is described in two equivalent ways:
\begin{itemize}
\item\label{i:Gale segment} Given $S, S'\in \ads_n$, we have $S\le S'$ if and only if $|S\cap U|\le|S'\cap U|$ for every upper segment $U$ of the order \eqref{eq:order on J}.
\item In terms of elementwise inequalities, if $S=\{i_1,\ldots,i_n\}$ and $S'=\{j_1,\ldots,j_n\}$ with $i_1<\cdots<i_n$ and $j_1<\cdots<j_n$, then $S\le S'$ if and only if $i_k\le j_k$ for all $k$.
\end{itemize}
\begin{propdefn}\cite[\S6.1.1]{BGW}\label{defn:SchubertDelta}
Each lower interval $[[\overline n],S]$ in the Gale order is the set of feasible sets of a delta-matroid $\Omega_S$.
We call the $\Omega_S$ for $S\in\ads_n$ the \newword{standard Schubert delta-matroids}.
A \textbf{Schubert delta-matroid} is a $\W$-image of a standard Schubert delta-matroid. 
\end{propdefn}

\begin{eg}
For $n=3$, the admissible sets dominated by $\{\overline2,1,3\}$ are
\[\{\overline2,1,3\},\{\overline3,1,2\},\{\overline2,\overline1,3\},\{\overline3,\overline1,2\},\{\overline3,\overline2,1\},\{\overline3,\overline2,\overline1\},\]
so the standard Schubert delta-matroid $\Omega_{\{\overline2,1,3\}}$ 
is the delta-matroid whose polytope is the convex hull of
\[\{\be_{\{1,3\}},\be_{\{1,2\}},\be_{\{3\}},\be_{\{2\}},\be_{\{1\}},\be_{\{\emptyset\}}\}\]
One may also recognize this polytope as the independence polytope of the matroid on $[3]$ whose bases are $\{1,2\}$ and $\{1,3\}$.
\end{eg}

For $S\in \ads_n$, the standard Schubert delta-matroid polytope $P(\Omega_S)$ is the independence polytope of a type~$A$ Schubert matroid in the following way.  The \textbf{standard Schubert matroid} $\Omega_{T}^{A}$ of a subset $T\subseteq [n]$ is the matroid on $[n]$ whose set of bases is
\[
\Omega_{T}^{A}=\{B \subseteq [n] : |B| = |T| \text{ and } B\leq T \text{ in the dominance order}\}
\]
where the dominance order is taken with respect to the ground set ordering $1<\cdots<n$.
\begin{lem}\label{lem:GaleEquivalences}
For $S,S'\in \ads_n$, then the following are equivalent.
\begin{enumerate}
    \item $S\le S'$ in the Gale order;
    \item $|S\cap \{i,\ldots,n\}|\le |S' \cap \{i,\ldots,n\}|$ for all $1\le i \le n$; and
    \item There exists $B\subset [n]$ with $|B|=|S'\cap [n]|$ such that
    $S\cap [n]\subset B\le S'\cap [n]$, where the inequality is taken in the dominance order.
\end{enumerate}
\end{lem}
\begin{proof}
All equivalences are easy to verify directly, so we omit the proof.
\end{proof}

A \textbf{Schubert matroid} is a $\mathfrak{S}_n$-image of a standard Schubert matroid.
From the equivalence of the first and third parts of \Cref{lem:GaleEquivalences}, we see that
$P(\Omega_S)=IP(\Omega^{A}_{S\cap [n]})$,
and so the subset
\[
\mathsf{SchDMat}_n = \text{the set of all Schubert delta-matroids on $[n,\overline n]$}
\]
of $\mathsf{DMat}_n$ is identified with the set of $\W$-images of independence polytopes of Schubert matroids on $[n]$. 
The name ``Schubert (delta-)matroid'' reflects a relationship with Schubert cells explained in \Cref{rem:geometricSchubert}.

\subsection{Intersecting with unit cubes}\label{ssec:unitcube}

We record here some key properties concerning how lattice $B_n$ generalized permutohedra intersect with unit cubes.  We will use them to prove \Cref{mainthm:isoms} and some related isomorphisms in the next subsection.

The natural level of generality of our first proposition, \Cref{prop:intersect with a cube}, is not only lattice $B_n$ generalized permutohedra
but also their unbounded analogues.
A polyhedron $P\subseteq\mathbb R^n$ is \newword{lattice} (over~$\mathbb Z^n$) 
if the affine span $\operatorname{aff}(F)$ of any face $F$ of~$P$ contains a coset of a subgroup of $\mathbb Z^n$ of rank $\dim F$. 
If $P$ is bounded, i.e., $P$ is a polytope, this is equivalent to the vertices of~$P$ being lattice points,
because the differences between vertices of~$F$ generate the subgroup sought for any face~$F$.

\begin{lem}\label{lem:u-rx_1 bounded}
Let $P\subset \mathbb{R}^n$ be a (closed convex) polyhedron and $u\colon \mathbb R^n\to\mathbb R$ a linear functional.
If $P^+=P\cap\{x\in\mathbb R^n:x_1\ge0\}$ is nonempty,
then $u$ is bounded below on $P^+$ if and only if there exists $r\ge0$ such that $u-rx_1$ is bounded below on~$P$.
\end{lem}

\begin{proof}
Suppose $u$ is bounded below on $P^+$.
If $u$ attains its minimum over points $x\in P^+$ at a point with $x_1>0$, then $r=0$ suffices. 
Otherwise take \[r=\limsup_{y\to 0^+}\frac{1}{y}\big(\min\{u(x):x\in P, \, x_1=0\}-\min\{u(x'):x'\in P,\, x'_1=y\}\big).\]
The limit superior exists because finitely many faces on the boundary of $P^+$ contain a minimizer $x'$,
and for each either $y$ is bounded away from~0 or the face also contains a minimizer $x$ and the quantity inside is constant.
The converse is clear because $u\ge u-rx_1$ on~$P^+$.
\end{proof}

\begin{lem}\label{lem:u+re_1 B_n}
Let $\sigma$ be a cone of $\Sigma_{B_n}$, and let $u$ lie in the relative interior of~$\sigma$.
Then both the set of cones of $\Sigma_{B_n}$ which meet $\operatorname{cone}\{u,\be_1\}$
and the order in which $u+\lambda\be_1$ meets these cones as $\lambda\ge0$ increases
are functions of~$\sigma$, independent of~$u$.
\end{lem}
In lieu of a proof of \Cref{lem:u+re_1 B_n} we describe the cones arising. This is easier in the language of total preorders.
Arbitrary cones of $\Sigma_{B_n}$ are in bijection with 
total preorders $\le$ on $[n,\overline n]$ such that for $i,j\in [n,\overline{n}]$, $i\le j$ if (and only if) $\overline j\le\overline i$,
via the map 
\[\mathord\le\mapsto C_\le=\operatorname{cone}\big\{\sum_{j\le i}\be_j : i\in[n,\overline n]\big\}.\]
In \Cref{lem:u+re_1 B_n}, if $\sigma=C_\le$, then the cones whose relative interiors meet $\operatorname{cone}\{u,\be_1\}$ are the $C_\preceq$ for all $\preceq$ such that 
$\le$ and $\preceq$ have the same restriction to $[n,\overline n]\setminus\{1,\overline1\}$,
and for all $i\in[n,\overline n]$, if $1\le i$ then $1\preceq i$.

\begin{prop}\label{prop:intersect with a cube}
Let $P$ be a lattice polyhedron, possibly unbounded, whose normal fan coarsens a subfan of~$\Sigma_{B_n}$.
If $m\in \mathbb{Z}^n$ and $P\cap(m+[0,1]^n)$ is nonempty, then $P\cap(m+[0,1]^n)\in\GP_{\mathbb Z,B_n}$.
\end{prop}

The above result is also proved in \cite{FujishigeParametric,FujishigePatkar}, at least when $P$ is a lattice polytope, using the theory of bisubmodular functions. We include a direct proof. 
The counterpart for type~$A$ generalized permutohedra follows from \cite[(44.70)]{SchrijverB} on intersections with coordinate half-spaces, 
which implies that \Cref{thm:disect} below also holds for type~$A$.

\begin{proof}
By translating we may assume that $m=0$. The cube $\square=[0,1]^n$ is an intersection of coordinate half-spaces.
So we reduce to considering the intersection of $P$ with a coordinate half-space $H^+$, say $\{(x_1, \ldots, x_n) \in \RR^n : x_1\geq 0\}$, and 
showing that if $P\cap H^+$ is nonempty, then it is a lattice polyhedron and has normal fan coarsening a subfan of~$\Sigma_{B_n}$.
Together with the observation that $P\cap\square$ is bounded because $\square$ is, this proves the proposition. 

First, we show that $P\cap H^+$ is lattice. Note that for any face $G$ of $P\cap H^+$, there is a face $F$ of~$P$ such that either
\begin{enumerate}
    \item $G=F\cap H^+$ and $\dim G=\dim F$, or
    \item $G=F\cap H$ and $\dim G=\dim F-1$.
\end{enumerate}
In the former case, $\operatorname{aff}(G)=\operatorname{aff}(F)$.
In the latter case, fix a cone of $\Sigma_{B_n}$ maximal among those normal to~$F$. This cone has the form 
\[\operatorname{cone}\{\be_{w(1)}+\be_{w(2)}+\cdots+\be_{w(i_k)} : k=1,\ldots,m\}\]
for some $w\in\W$ and $\{i_1,\ldots,i_m\}\subseteq[n]$ by \Cref{prop:B_n cones}. Thus
\begin{align}
\operatorname{aff}(F) &= \{x\in\mathbb R^n : x_{w(1)}+\cdots+x_{w(i_k)}=a_{i_k}\mbox{ for all }k=1,\ldots,m\},\notag
\\&= \{x\in\mathbb R^n : x_{w(i_{k-1})+1}+\cdots+x_{w(i_k)}=a_{i_k}-a_{i_{k-1}}\mbox{ for all }k=1,\ldots,m\}\label{eq:aff(F)} 
\end{align}
where the $a_i$ are integers because $P$ is lattice.
The lattice points in $\operatorname{aff}(G)=\operatorname{aff}(F)\cap H$ are those with $x_1=0$,
which form a coset of a subgroup of corank 1 among the lattice points in~$\operatorname{aff}(F)$
because $x_1$ appears in at most one equation in \eqref{eq:aff(F)}.
We have thus shown that $P\cap H^+$ is lattice.

Now we prove that the normal fan of $P\cap H^+$ coarsens a subfan of~$\Sigma_{B_n}$. 
Write $\face uQ$ for the face of a polytope $Q$ on which a linear functional $u\colon \mathbb R^n\to\mathbb R$ attains its minimum;
set $\face uQ=\emptyset$ by convention if no minimum is attained.
The assumption on~$P$ is that for each cone $\sigma$ of~$\Sigma_{B_n}$ with relative interior $\sigma^\circ$, 
it holds that $\face uP=\face vP$ for all $u,v\in\sigma^\circ$.
Our claim is that the same is true of $P\cap H^+$.

Fix a cone $\sigma$ of~$\Sigma_{B_n}$ and $u,v\in\sigma^\circ$.
By \Cref{lem:u-rx_1 bounded}, $\face u{(P\cap H^+)}=\emptyset$ if and only if $u-rx_1$ lies outside the normal fan of~$P$ for all $r\ge0$, where $x_1$ is the first coordinate functional,
and likewise for~$v$.
By \Cref{lem:u+re_1 B_n}, whether this happens depends only on $\sigma$, not on~$u$ or~$v$.
So it remains to handle the case $\face u{(P\cap H^+)}\ne\emptyset$.
If $\face uP$ is not disjoint from $H^+$, we are done, since in this case 
\[\face u{(P\cap H^+)}=(\face uP)\cap H^+=(\face vP)\cap H^+=\face v{(P\cap H^+)}.\]
If they are disjoint, let $r\in\mathbb R$ be minimal such that 
$F:=\face{u-rx_1}P$ intersects $H^+$, where $x_1$ is the first coordinate functional;
some such $r$ exists by our earlier invocation of \Cref{lem:u-rx_1 bounded}.
Note that $r>0$, so $u$ is a positive combination of $x_1$ and $u-rx_1$.
Since $\face{x_1}{(P\cap H^+)}=P\cap H$ and $\face{u-rx_1}{(P\cap H^+)}=F\cap H^+$ intersect in their common face $F\cap H$, this implies $\face u{(P\cap H^+)}=F\cap H$.
Again by \Cref{lem:u+re_1 B_n}, the faces of the form $\face{u-rx_1}P$, and their order they appear in as $r$ varies,
depend only on~$\sigma$, so we have $\face v{(P\cap H^+)}=F\cap H$ also.
\end{proof}

Let 
\[
C = -C_{\operatorname{id}}^{\vee}=\operatorname{cone}\{-\be_1,\be_1-\be_2,\ldots,\be_{n-1}-\be_n\}
  = \{x\in\mathbb R^n:\sum_{i=k}^nx_i\le0\mbox{ for }k\in[n]\}.
\]
This is the type~$B_n$ negative root cone for the choice of positive roots corresponding to our Gale order \cite[\S3.2.2]{BGW}.

\begin{lem}\label{lem:positive root cone 1}
Let $m\in \{0,1\}^n$ and let $S\in \ads_n$ be the size $n$ admissible set such that $m$ is the indicator vector of $S\cap [n]$. Then $P(\Omega_{S}) = (m+C)\cap[0,1]^n$.
\end{lem}

\begin{proof}
The half-space description of $m+C$ is
\begin{equation}\label{eq:m+C}
m+C=\{x\in\mathbb R^n:\sum_{i=k}^nx_i\le\sum_{i=k}^nm_i\mbox{ for }k\in[n]\}.
\end{equation}
By the equivalence of the first and second parts of \Cref{lem:GaleEquivalences}, we see that $x\in (m+C)\cap \{0,1\}^n$ if and only if, for the admissible set $S'\in \ads_n$ such that $x$ is the indicator vector of $S'\cap [n]$, we have $S'\le S$ in the Gale order. Therefore $(m+C)\cap[0,1]^n$ and $\Omega_{S}$ contain the same set of lattice points.
Since $C$ is the dual of a cone of~$\Sigma_{B_n}$, \Cref{prop:intersect with a cube} applies
and shows that $(m+C)\cap[0,1]^n$ is a lattice polytope. 
But $\Omega_{S}$ is also a lattice polytope, so they are equal. 
\end{proof}

\begin{prop}\label{prop:positive root cone}
Let $m\in\mathbb Z^n$. If the intersection $(m+C)\cap[0,1]^n$ is nonempty, then it is a standard Schubert delta-matroid polytope.
\end{prop}

\begin{proof}
Assume that $(m+C)\cap[0,1]^n$ is nonempty. 
We construct a sequence $m^0=m$, $m^1$, \ldots\ of integer vectors so that 
\begin{equation}\label{eq:positive root cone invariant}
(m^j+C)\cap[0,1]^n=(m+C)\cap[0,1]^n.
\end{equation} 
One of the $m^j$ will lie in $\{0,1\}^n$,
whereupon the proposition follows from \Cref{lem:positive root cone 1}.

Denote the generators of~$C$, the negative simple roots, by $\alpha_1=-\be_1$ and $\alpha_i=\be_{i-1}-\be_i$ for $i=2,\ldots,n$.
An arbitrary lattice point of $m^j+C$ has the form $x=m^j+\sum_{i=1}^n a_i\alpha_i$ for nonnegative integers $a_i$.
If $m^j_i>1$ then we let $m^{j+1} = m^j+(m^j_i-1)\alpha_i$. 
In this case $x_i\le1$ only if $a_i>m^j_i-1$, so $m^j+C$ and $m^j+(m^j_i-1)\alpha_i+C$ have the same intersection with $[0,1]^n$ and \eqref{eq:positive root cone invariant} holds.
Similarly, if $m^j_i<0$, then we let $m^{j+1} = m^j+(-m^j_i)\alpha_{i+1}$,
and \eqref{eq:positive root cone invariant} holds because $x_i\ge0$ only if $a_{i+1}>-m_i$ (note that $i<n$ in this case, which follows from $(m+C)\cap[0,1]^n$ being nonempty).

The sequence $(\sum_{i=1}^nim^j_i)_{j\ge0}$ is decreasing by construction, 
and bounded below by~0, because if $\sum_{i=1}^nim^j_i<0$ the functional $\sum_{i=1}^nix_i$ takes negative values on $m^j+C$ and nonnegative values on $[0,1]^n$, implying $(m^j+C)\cap[0,1]^n=\emptyset$.
So it is finite, i.e., the case $m^j\in\{0,1\}$ happens after finitely many steps.
\end{proof}

\begin{cor}\label{cor:minors of Schubert}
The set $\mathsf{SchDMat}_n$ is closed under nonempty intersections with faces of $[0,1]^n$.
\end{cor}

\begin{proof}
By the $\W$ symmetry and iteration, it's enough to prove that if $P=P(\D)$ for $\D$ a standard Schubert delta-matroid and $F$ is a facet of $[0,1]^n$, then $P\cap F\in\mathsf{SchDMat}_n$.
Write $P=(m+C)\cap[0,1]^n$ as in \Cref{prop:positive root cone},
and $F=H\cap[0,1]^n$ for a hyperplane $H=\{x\in\mathbb R^n:x_i=s\}$ where $i\in[n]$ and $s\in\{0,1\}$. 
Then $P\cap F = (m+C)\cap H\cap[0,1]^n$.
Let $\pi\colon H\to\mathbb R^{n-1}$ be the map omitting the $i\/$th coordinate.
Using \eqref{eq:m+C} and its counterpart for $B_{n-1}$,
one can check that $(m+C)\cap H$ is identified by~$\pi$ with a translate of 
the cone $-C_{\operatorname{id}}^\vee$ which is dual to a cone in~$\Sigma_{B_{n-1}}$.
Therefore $\pi$ takes $P\cap F$ to a type~$B_{n-1}$ standard Schubert delta-matroid polytope.
This implies that $P\cap F$ is a Schubert delta-matroid polytope, as follows.
In the case $H=\{x\in\mathbb R^n:x_n=0\}$, 
if $\pi(P\cap F)=P(\Omega_S)$ for $S$ a maximal admissible subset of~$[n-1]$,
then $P\cap F=P(\Omega_{S\cup\{\overline n\}})$ by \Cref{lem:positive root cone 1}.
The other possible choices of~$H$ are $\W$ images of this one,
so in general $P\cap F$ is a $\W$ image of $P(\Omega_{S\cup\{\overline n\}})$.
\end{proof}

\begin{cor}\label{cor:positive root cone}
Let $\square'$ be a face of $[0,1]^n$, and $\sigma$ be a cone of $\Sigma_{B_n}$.
For $m\in\mathbb Z^n$, if the intersection $(m+\sigma^\vee)\cap\square'$ is nonempty, then it is in $\mathsf{SchDMat}_n$.
\end{cor}

\begin{proof}
If $\sigma$ is a maximal cone of~$\Sigma_{B_n}$, then $\sigma^\vee$ is a Weyl image of the cone $C= -C_{\operatorname{id}}^{\vee}$ above,
and the result follows from \Cref{prop:positive root cone} and \Cref{cor:minors of Schubert}.

For an arbitrary cone $\sigma$, we reduce to the preceding case.
The cone $\sigma$ is a face of a maximal cone $\tau$ of~$\Sigma_{B_n}$, so $\sigma^\vee$ is a tangent cone of $\tau^\vee$,
that is, $\sigma^\vee=-F+\tau^\vee$ for a face $F\subset\tau^\vee$.  Now for $m'\in-F\cap\mathbb Z^n$, we have
\[\sigma^\vee\supseteq(-F\cap(m'+F))+\tau^\vee=m'+\tau^\vee.\]
If $m'$ is chosen deep enough in the interior of $-F$,
the defining halfspaces of $m+m'+\tau^\vee$ will all contain $\square'$,
so $m+\sigma^\vee$ and $m+m'+\tau^\vee$ will have the same intersection with~$\square'$.
\end{proof}

\subsection{Bases from Schubert delta-matroids}\label{ssec:bases from Schubert}

We are now ready to prove the following intermediate step for the proof of \Cref{mainthm:isoms}.

\begin{thm}\label{thm:disect}
One has
\[\mathbb I(\mathbb Z^n+\mathsf{SchDMat}_n) = \mathbb I(\mathbb Z^n+\mathsf{DMat}_n) = \mathbb I(\GP_{\mathbb Z,B_n}).\]
\end{thm}

\begin{proof}
Let $P\subset \RR^n$ be a lattice $B_n$ generalized permutohedron. 
We will write $\ind P$ as a sum of indicator functions of lattice translates of Schubert delta-matroid polytopes.
This will prove that $\mathbb I(\GP_{\mathbb Z,B_n})\subset\mathbb I(\mathbb Z^n+\mathsf{SchDMat}_n)$,
and the left-to-right inclusions in the theorem are clear.

Recall the signed permutohedron $\Pi_{B_n}$.
By the Brianchon--Gram theorem  
applied to $P+\varepsilon \Pi_{B_n}$ in the pointwise limit $\varepsilon\to0^+$, we have
\[\ind P = \sum_{\sigma\in\Sigma_{B_n}}(-1)^{\codim\sigma}\,\ind{P+\sigma^\vee}.\]
Note that $P+\sigma^\vee$ is a lattice translate of $\sigma^\vee$.

Tile $\mathbb{R}^n$ by lattice translates of Boolean cubes $[0,1]^n$. Let $\mathcal C$ be the set of all such cubes that meet $P$, together with their common internal faces, so that we have an inclusion-exclusion relation 
\[\ind{\bigcup_{F\in\mathcal C}F} = \sum_{F\in\mathcal C}(-1)^{\codim(F)}\,\ind F.\]
Then
\[\ind P = \sum_{F\in\mathcal C}(-1)^{\codim(F)}\,\ind{P\cap F}
= \sum_{F\in\mathcal C}\sum_{\sigma\in\Sigma_{B_n}}(-1)^{\codim(F)+\codim(\sigma)}\,\ind{(P+\sigma^\vee)\cap F}.\]
By \Cref{cor:positive root cone}, the right hand side is in $\mathbb I(\mathbb Z^n+\mathsf{SchDMat}_n)$.
\end{proof}

We remark that the second equality of the theorem could have been proved using the tiling by Boolean cubes and \Cref{prop:intersect with a cube} without invoking the Brianchon--Gram theorem.

\begin{cor}\label{cor:indgp isomorphisms}
One has
\[\indgp(\mathsf{SchDMat}_n) = \indgp(\mathsf{DMat}_n) = \indgp(\GP_{\mathbb Z,B_n}).\]
\end{cor}

\label{sec:poly isoms}
\begin{proof}
What is left to prove after \Cref{thm:disect} is that the three groups of relations are equal.
These are generated by $\ind{m+P}-\ind P$ where $m\in\mathbb Z^n$ and $P\in\mathsf{SchDMat}_n$, $\mathsf{DMat}_n$, and $\GP_{\mathbb Z,B_n}$ respectively.
If $P\in\GP_{\mathbb Z,B_n}$, then another use of~\Cref{thm:disect} gives us a finite expression
\begin{align*}
\ind{m+P}-\ind P &= \sum_{Q\in\mathsf{SchDMat}_n,v\in\mathbb Z^n}a_{Q,v}\,(\ind{m+v+Q}-\ind{v+Q})
\\&= \sum_{Q\in\mathsf{SchDMat}_n,v\in\mathbb Z^n}a_{Q,v}\,\big((\ind{m+v+Q}-\ind Q)-(\ind{v+Q}-\ind Q)\big).
\end{align*}
So the relations for $\indgp(\GP_{\mathbb Z,B_n})$ are also relations for $\indgp(\mathsf{SchDMat}_n)$,
and the other containments are obvious.
\end{proof}

We prepare for the proof of \Cref{mainthm:isoms} by proving the analogous fact for $\mathbb I(\mathsf{DMat}_n)$.

\begin{prop}\label{prop:ESS}
The set $\{\ind{P}:P\in\mathsf{SchDMat}_n\}$ is a basis for $\mathbb I(\mathsf{DMat}_n)$.
\end{prop}

\begin{proof}
The first equality in \Cref{thm:disect} implies that every $\ind{P}$ for $P$ a delta-matroid polytope can be expressed as a linear combination of indicator functions of Schubert delta-matroid polytopes.
Here we note that a lattice translate of a Schubert delta-matroid polytope $P(\D)$, provided it is contained in the unit cube, is again a Schubert delta-matroid polytope because it is a $\W$-image of~$P(\D)$.

For linear independence, suppose we have a nontrivial relation
\[
\sum_{i = 1}^k a_{i} \ind{P_i} = 0 \qquad \text{with } k\geq 1 \text{ and } a_1, \ldots, a_k \neq 0
\]
where $P_1, \ldots, P_k$ are Schubert delta-matroids.
By \Cref{prop:positive root cone}, there exists $w\in \W$ and $m\in \ZZ^n$ such that $P_1 = [0,1]^n \cap (m + w\cdot C)$. 
Without loss of generality, we may assume that $P_1$ does not contain $P_i$ for all $i>1$.    In particular, no $P_i$ for $i>1$ is contained in $m + w\cdot C$.  Now, \cite[Theorem 2.3]{ESS} implies that the assignment
\[
P \mapsto \begin{cases}
1 &\text{if $P \subset m + w\cdot C$ and $P \cap m \neq \emptyset$}\\
0 &\text{otherwise} 
\end{cases}
\]
defines a strongly valuative function on $\GP_{\ZZ^n,B_n}$.  Applying this function to both sides of the relation $\sum_{i = 1}^k a_{i} \ind{P_i} = 0$ then yields $a_1 = 0$, a contradiction.
\end{proof}

We are ready to prove \Cref{thm:basis}.
Converted to a statement about polyhedra by using \Cref{thm:indK}, the theorem asserts that a basis of~$\indgp(\GP_{\ZZ,B_n})$ is
\[
{\mathsf{SchDMat}}_n^{\mathsf{clf}} := \{\D\in \mathsf{SchDMat}_n:\D\text{ has no coloops}\}.
\]
The superscript $^{\mathsf{clf}}$ stands for ``coloop-free.''
We verify that, among the polytopes of the delta-matroids in ${\mathsf{SchDMat}}_n^{\mathsf{clf}}$,
there is exactly one translate of any Schubert delta-matroid polytope.
For any $\D\in\mathsf{SchDMat}$, changing any coloops $\D$ may have to loops gives a translate in ${\mathsf{SchDMat}}_n^{\mathsf{clf}}$.
If for two delta-matroids $\D$ and $\D'$ we have
$P(\D') = m+P(\D)$ for some $m\in\mathbb Z^n$,
then $m\in\{-1,0,1\}^n$; if for some $i$ we have $m_i=1$, then $P(\D')\subseteq\{x\in\mathbb R^n:x_i=1\}$ and $P(\D)\subseteq\{x\in\mathbb R^n:x_i=0\}$, and if $m_1=-1$ then these containments hold vice versa, 
so not both $\D$ and $\D'$ are coloop-free.

Our method for proving \Cref{thm:basis} can also be used to deduce the counterpart of the theorem in type~$A$, 
i.e., that coloop-free Schubert matroids are a basis for the translation-invariant polytope algebra of lattice type $A$ generalized permutohedra.  
Another proof of the type~$A$ theorem can be assembled from \cite[Theorem D]{BEST} and the analogous theorem for the cohomology ring in type~$A$ appearing in \cite{Ham17}.

\begin{proof}[Proof of \Cref{mainthm:isoms}]
\Cref{thm:disect} shows that $\{[P]:P\in\mathsf{SchDMat}_n^{\mathsf{clf}}\}$ generates $\indgp(\mathsf{SchDMat}_n)$. So we must prove linear independence.

We first show translates of coloop-free Schubert delta-matroids are linearly independent in $\mathbb I(\mathbb Z^n+\mathsf{SchDMat}_n)$. Suppose we are given a finite relation
\[\sum_{P\in\mathsf{SchDMat}_n^{\mathsf{clf}},m\in\ZZ^n}a_{P,m}\,\ind{m+P}=0.\]
Let $V\subseteq\ZZ^n$ be the set of vectors $v$ such that, for some $(P,m)$ with $a_{P,m}\ne0$,
$m+P$ intersects the translate $[0,1)^n+v$ of the half-open cube.
Our objective is to prove $V$ empty. Suppose otherwise, and let $v\in V$ be lexicographically minimum.
Restricting our relation to the closed cube $v+[0,1]^n$ gives
\[\sum_{P\in\mathsf{SchDMat}_n^{\mathsf{clf}},m\in\ZZ^n}a_{P,m}\,\ind{(m+P)\cap(v+[0,1]^n)}=0.\]
If $(m+P)\cap(v+[0,1]^n)$ is nonempty, then it has the form $v+Q$ for some $Q\in\mathsf{SchDMat}_n$ by \Cref{cor:minors of Schubert}.
Letting 
\[J(Q)=\{(P,m):(m+P)\cap(v+[0,1]^n) = v+Q\},\]
we collect identical translates:
\[\sum_{Q\in\mathsf{SchDMat}_n}\left(\sum_{(P,m)\in J(Q)}a_{P,m}\right)\ind{v+Q}=0.\]
By \Cref{prop:ESS}, every inner sum is zero.  
For any $Q\in\mathsf{SchDMat}_n^{\mathsf{clf}}$, minimality of~$v$ implies that the only possibly nonzero summand in this inner sum is the one indexed by $(P,m)=(Q,v)$, so $a_{Q,v}=0$.  But this contradicts $v\in V$. 

Now, a linear dependence in $\indgp(\GP_{\mathbb Z,B_n})$,
\[\sum_{P\in\mathsf{SchDMat}_n^{\mathsf{clf}}} a_P\,[\ind P] = 0,\]
lifts to $\mathbb I(\GP_{\mathbb Z,B_n})$ as a relation
\[\sum_{P\in\mathsf{SchDMat}_n^{\mathsf{clf}}} a_P\,\ind P + \sum_{Q,m\in\mathbb Z^n\setminus\{0\}} b_{Q,m}\,(\ind{m+Q} - \ind Q) = 0\]
over some family of lattice $B_n$ generalized permutohedra $Q$, where finitely many $b_{Q,m}$ are nonzero.
Applying \Cref{thm:disect} to these $Q$, this can be rewritten
\[\sum_{P\in\mathsf{SchDMat}_n^{\mathsf{clf}}} a_P\,\ind P + \sum_{P\in\mathsf{SchDMat}_n,m\ne0} c_{P,m}\,(\ind{m+P} - \ind{P}) = 0.\]
Every $P\in\mathsf{SchDMat}_n$ has a lattice translate $P' \in\mathsf{SchDMat}_n^{\mathsf{clf}}$, and we can use the relation
$\ind{m+Q} - \ind Q = (\ind{m+Q} - \ind{Q'}) - (\ind{Q} - \ind{Q'})$ for any polytopes $Q, Q'$
to rewrite the second sum:
\[\sum_{P\in\mathsf{SchDMat}_n^{\mathsf{clf}}} a_P\,\ind P + \sum_{P' \in\mathsf{SchDMat}_n^{\mathsf{clf}},m\ne0} d_{P',m}\,(\ind{m+P'} - \ind{P'}) = 0.\]
The earlier lifted linear independence statement implies that each polytope in the above sum has a zero coefficient, i.e., $d_{P,m}=0$ for all $m\ne0$ and $a_P-\sum_{m\ne0}d_{P,m}=0$. Therefore $a_P=0$ for all $P\in\mathsf{SchDMat}_n^{\mathsf{clf}}$.
\end{proof}

\section{The exceptional Hirzebruch--Riemann--Roch-type theorem}
\label{sec:HRR}

We prove \Cref{mainthm:HRR}, relating the Grothendieck ring of vector bundles $K(X_{B_n})$ to the Chow cohomology $A^\bullet(X_{B_n})$, in two parts.  
In \S\ref{subsec:exceptIsom}, we establish the isomorphism $\phimap\colon  K(X_{B_n}) \to A^\bullet(X_{B_n})$ via localization methods in torus-equivariant geometry.  Then, in \S\ref{subsec:fakeHRR}, we establish the formula involving the sheaf Euler characteristic by relating the isomorphism $\phimap$ to a similar isomorphism for stellahedral varieties established in \cite{EHL}.

\subsection{$K$-rings and Chow rings of $X_{B_n}$}\label{subsec:eqvprep}
Let $T = \mathbb{G}_m^n$ be the torus embedded in $X_{B_n}$, and let $K_T(X_{B_n})$ be the $T$-equivariant $K$-ring of $X_{B_n}$, which is the Grothendieck ring of $T$-equivariant vector bundles on $X_{B_n}$, and let $A_T^{\bullet}(X_{B_n})$ be the $T$-equi\-variant Chow ring in the sense of \cite{EG1998}. We describe the equivariant and non-equivariant $K$ and Chow rings of $X_{B_n}$. 
We will make use of descriptions of $K_T(X_{B_n})$ and $A_T^{\bullet}(X_{B_n})$ coming from equivariant localization. See \cite[Section 2]{EHL} for a review of equivariant localization. 

\medskip
We first set up some notation. To describe the adjacent maximal cones in $\Sigma_{B_n}$, we use the following special involutions in~$\W$:
\begin{itemize}
    \item $\tau_{i,i+1}=(i,i+1)(\overline{i},\overline{i+1})$ for $1\le i \le n-1$, and
    \item $\tau_n=(n,\overline{n})$.
\end{itemize}
Then $C_w$ is adjacent to $C_{w'}$ exactly if $w=w'\tau_{i,i+1}$ for some $i$, in which case the common facet normal is $\pm (\be_{w(i)}-\be_{w(i+1)})$, or $w=w'\tau_n$, in which case the common facet normal is $\pm \be_{w(n)}$. 
Recall that $K_T(\mathrm{pt}) = \mathbb{Z}[T_1^{\pm 1}, \dotsc, T_n^{\pm 1}]$ and $A_T^{\bullet}(\mathrm{pt}) = \mathbb{Z}[t_1, \dotsc, t_n]$. Let $T_{\bar{i}} = T_i^{-1}$ and $t_{\bar{i}} = -t_i$ for $i \in [n]$.
\begin{thm}\label{thm:Klocalization}\cite{VezzosiVistoli,PayneCohomology} The following hold. 
\begin{enumerate}
\item  The injective localization map $K_T(X_{B_n}) \to K_T(X_{B_n}^T)=\bigoplus_{w \in \W} K_T(\mathrm{pt})$ identifies $K_T(X_{B_n})$ with the set of collections of elements  $(f_w)_{w \in \W} \in \bigoplus_{w \in \W} \mathbb{Z}[T_1^{\pm 1},\ldots,T_n^{\pm 1}]$ such that
\begin{itemize}
    \item if $w\tau_{i,i+1}=w'$ for $1\le i \le n-1$, then $f_w \equiv f_{w'}\text{ mod }1-T_{w(i)}T_{w(i+1)}^{-1}$, and
    \item if $w\tau_n=w'$ then $f_w \equiv f_{w'}\text { mod }1-T_{w(n)}$.
\end{itemize}
The diagonal embedding of  $\mathbb{Z}[T_1^{\pm 1},\ldots, T_n^{\pm 1}]$ into $\bigoplus_{w \in \W} K_T(\mathrm{pt})$ identifies $\mathbb{Z}[T_1^{\pm 1},\ldots, T_n^{\pm 1}]$ with a subring of $K_T(X_{B_n})$, and the $K$-ring $K(X_{B_n})$ is given by $$K(X_{B_n})=K_T(X_{B_n})/(T_1-1,\ldots,T_n-1).$$
\item The injective localization map $A_T^{\bullet}(X_{B_n}) \to A_T^{\bullet}(X_{B_n}^T)=\bigoplus_{w \in \W} A_T^{\bullet}(\mathrm{pt})$ identifies $A_T^{\bullet}(X_{B_n})$ with the set of collections of elements  $(f_w)_{w \in \W} \in \bigoplus_{w \in \W} \mathbb{Z}[t_1, \dotsc, t_n]$ such that
\begin{itemize}
    \item if $w\tau_{i,i+1}=w'$ for $1\le i \le n-1$, then $f_w \equiv f_{w'}\text{ mod }t_{w(i)}-t_{w(i+1)}$, and 
    \item if $w\tau_n=w'$ then $f_w\equiv f_{w'}\text { mod }t_{w(n)}$.
\end{itemize}
The diagonal embedding of  $\mathbb{Z}[t_1, \dotsc, t_n]$ into $\bigoplus_{w \in \W} A_T^{\bullet}(\mathrm{pt})$ identifies $\mathbb{Z}[t_1, \dotsc, t_n]$ with a subring of $A_T^{\bullet}(X_{B_n})$, and the Chow ring $A^{\bullet}(X_{B_n})$ is given by $$A^{\bullet}(X_{B_n})=A_T^{\bullet}(X_{B_n})/(t_1,\ldots,t_n).$$
\end{enumerate}
\end{thm}

There is an action of $\W$ by automorphisms on $X_{B_n}$, so we functorially obtain an action of $\W$ on $K(X_{B_n})$ and $A^{\bullet}(X_{B_n})$. 
We now describe $\W$-actions on $K_T(X_{B_n})$ and $A_T^{\bullet}(X_{B_n})$,
the latter being the type~$B_n$ case of Tymoczko's dot action \cite{Tym08}.
To do so, we prepare with some generalities on maps between torus-equivariant $K$-rings for actions of potentially different tori.
For $i = 1,2$, let $T_i$ be a torus and $X_i$ a smooth projective $T_i$-variety.
Suppose we have a map of tori $\varphi\colon  T_1 \to T_2$ and a map $\overline\varphi\colon  X_1 \to X_2$ with the commuting diagram
\[
\begin{tikzcd}
&T_1\times X_1 \ar[r, "\varphi\times\overline\varphi"]\ar[d] &T_2 \times X_2 \ar[d]\\
&X_1 \ar[r,"\overline\varphi"]&X_2,
\end{tikzcd}
\]
where the two vertical maps are the torus actions.  Then, by treating $X_2$ as a $T_1$-variety via $\overline \varphi$, we have the induced maps
\begin{equation}\label{eq:internal note maps}
K_{T_2}(X_2) \to K_{T_1}(X_2) \overset{\overline\varphi^*}\to K_{T_1}(X_1)
\end{equation}
where the first map is the ``forgetful map'' and the second map is the pullback map.  We similarly have induced maps of equivariant Chow rings.

\medskip
In our situation, we will have $T_1 = T_2 = T$ and $X_1 = X_2 = X_{B_n}$ in the following way.
An element $w\in \W$ acts on $\RR^n$ by $\be_i \mapsto \be_{w(i)}$.  We consider $\RR^n$ as the real vector space $\operatorname{Cochar}(T) \otimes \mathbb{R}$ that contains the fan $\Sigma_{B_n}$.
This $\W$-action defines an automorphism $\varphi_w \colon T \to T$ given by $T_i \mapsto T_{w^{-1}(i)}$.  Since the $\W$-action maps $\Sigma_{B_n}$ isomorphically onto itself, the map $\varphi_w$ extends to an automorphism $\overline\varphi_w \colon X_{B_n} \to X_{B_n}$.  The map $\overline \varphi_w$ is not a $T$-equivariant map, but it fits into the commuting diagram
\[
\begin{tikzcd}
&T\times X_{B_n} \ar[r, "\varphi_w\times \overline\varphi_w"]\ar[d] &T \times X_{B_n} \ar[d]\\
&X_{B_n} \ar[r, "\overline\varphi_w"]&X_{B_n}.
\end{tikzcd}
\]
Hence, we have the maps
\[
\psi_w\colon K_{T}(X_{B_n}) \to K_{T}(X_{B_n}) \overset{\overline\varphi_w^*}\to K_{T}(X_{B_n})
\]
as in \eqref{eq:internal note maps}, and similarly for $A_T^{\bullet}(X_{B_n})$. The assignments $w\mapsto \psi_{w^{-1}}$ give a $\W$-action descending to the usual $\W$-action on $K(X_{B_n})$ and $A^{\bullet}(X_{B_n})$.
In terms of the localization description of $K_T(X_{B_n})$ and $A_T^{\bullet}(X_{B_n})$ in Theorem~\ref{thm:Klocalization}, the action has the following explicit description:
\begin{enumerate}
\item An element $w \in \W$ acts on $f \in K_T(X_{B_n})$ by $(w \cdot f)_{w'} = f_{w^{-1}w'}(T_{w(1)}, \dotsc, T_{w(n)})$.
\item An element $w \in \W$ acts on $f \in A^{\bullet}_T(X_{B_n})$ by $(w \cdot f)_{w'} = f_{w^{-1}w'} (t_{w(1)}, \dotsc, t_{w(n)})$.
\end{enumerate}

\subsection{The exceptional isomorphism}\label{subsec:exceptIsom}
Recall the map $\epsilon \colon \W \to \{\pm 1\}^n$ from Section~\ref{ssec:FanBn}. 
\begin{thm}\label{thm:Psiiso}
There is an injective ring map
$$\phimap_T \colon K_T(X_{B_n}) \to A_T^{\bullet}(X_{B_n})[1/(1 \pm t_i)]:=A_T^{\bullet}(X_{B_n})[\{ \textstyle\frac{1}{1-t_i},\frac{1}{1+t_i}\}_{1 \le i \le n}]$$
obtained by
$$(\phimap_T(f))_w(t_1,\ldots,t_n)=f_w(h_{\epsilon_1(w)}(t_1),\ldots,h_{\epsilon_n(w)}(t_n))$$
where
$$h_\epsilon(t)=(1+\epsilon t)^{\epsilon} := \begin{cases}1+t&\epsilon=+1\\ \frac{1}{1-t}&\epsilon=-1\end{cases}.$$
This equivariant map $\phimap_T$ descends to a non-equivariant isomorphism $\phimap \colon K(X_{B_n}) \overset\sim\to A^{\bullet}(X_{B_n})$. Finally, $\phimap$ and $\phimap_T$ are $\W$-equivariant in the sense that they intertwine the above $\W$-actions:
$$\phimap_T(w\cdot f)=w\cdot \phimap_T(f),\text{ and }\phimap(w\cdot f)=w\cdot \phimap(f).$$
\end{thm}
\begin{proof}
We first check that $\phi^B_T$ is $\W$-equivariant. For $f \in K_T(X_{B_n})$, we have that
$$ (\phi_T^B(w \cdot f))_{w'} = f_{w^{-1}w'}(h_{\epsilon_1(w')}(T_{w(1)}), \dotsc, h_{\epsilon_n(w')}(T_{w(n)})), \text{ and}$$
$$(w \cdot \phi_T^B(f))_{w'} = f_{w^{-1}w'}((1 + \epsilon_1(w') t_{w(1)})^{\epsilon_1(w')}, \dotsc, (1 + \epsilon_n(w') t_{w(n)})^{\epsilon_n(w')}),$$
which are equal.
We now check the congruence conditions.
First, we check for $w'=w\tau_{i,i+1}$ that
$$(\phi^B_T(f))_{w} \equiv (\phi^B_T(f))_{w'} \mod {t_{w(i)}-t_{w(i+1)}}.$$
By $\W$-equivariance, this is equivalent to
$$(\phi^B_T(w^{-1}\cdot f))_{\id} \equiv (\phi^B_T(w^{-1}\cdot f))_{\tau_{i,i+1}} \mod {t_{i}-t_{i+1}},$$
which by definition of $\phi^B_T$, and the fact that $\epsilon_j(\id)=\epsilon_j(\tau_{i,i+1})=1$ for all $j$, is equivalent to
$$(w^{-1}\cdot f)_{\id}(t_1+1,\ldots,t_n+1) \equiv (w^{-1}\cdot f)_{\tau_{i,i+1}}(t_1+1,\ldots,t_n+1) \mod {t_{i}-t_{i+1}}.$$
Since $w^{-1}\cdot f\in K_T(X_{B_n})$, we have $ ((w^{-1}\cdot f)_{\id}(T_1,\ldots,T_n) \equiv (w^{-1}\cdot f)_{\tau_{i,i+1}}(T_1,\ldots,T_n)) \mod {1 - T_i^{-1}T_{i+1}}$, and the result follows from replacing $T_j$ with $t_j+1$ for all $j$. Now, we check for $w'=w\tau_{n}$ that
$$(\phi^B_T(f))_{w} \equiv (\phi^B_T(f))_{w'} \mod t_{w(n)} .$$
Indeed, this similarly follows from the fact that $w\cdot f\in K_T(X_{B_n})$ and the compatibility
$$(w^{-1}\cdot f)_{\id}(T_1,\ldots,T_n) \equiv (w^{-1}\cdot f)_{\tau_n}(T_1,\ldots,T_n) \mod {T_n - 1}.$$
As we now know that $\phi^B_T$ is well-defined, from the defining formula it is trivial to check that it is an injective ring map.

We now check that the map $\phimap_T$ descends non-equivariantly to a map $\phimap \colon K(X_{B_n})\to A^{\bullet}(X_{B_n})$. Note that under the map $A_T^{\bullet}(X_{B_n})\to A^{\bullet}(X_{B_n})$ we have $1\pm t_i\mapsto 1$, so there is an induced map $A_T^{\bullet}(X_{B_n})[\frac{1}{1\pm t_i}]\to A^{\bullet}(X_{B_n}).$ 
To obtain the map $\phimap$, we have to show that under the composite $K_T(X_{B_n})\to A_T^{\bullet}(X_{B_n})[\frac{1}{1\pm t_i}]\to A^{\bullet}(X_{B_n})$, the ideal $(T_1-1,\ldots,T_n-1)$ gets mapped to $0$. Indeed, $\phimap_T(T_i-1)=t_i \cdot r_i$ where $(r_i)_w$ is $1$ if $\epsilon_i(w)=1$ and $\frac{1}{1-t_i}$ if $\epsilon_i(w)=-1$. Therefore $\phimap_T(T_i - 1)$ is zero under the map $A_T^{\bullet}(X_{B_n})[\frac{1}{1\pm t_i}]\to A^{\bullet}(X_{B_n})$ because $t_i$ maps to $0$.

The $\W$-equivariance of $\phi^B$ follows immediately from the $\W$-equivariance of $\phimap_T$, so it remains to check that $\phimap$ is an isomorphism. For this, we identify the image of $\phi^B_T$. 
Note that $\phimap(K_T(X_{B_n}))$ lies in the subring $R\subset A_T^{\bullet}(X_{B_n})[\frac{1}{1\pm t_i}]$ consisting of those $g$ where $g_{w}$ lies in the ring $K_T(\mathrm{pt})[\frac{1}{1+\epsilon_1(w)t_1},\ldots, \frac{1}{1+\epsilon_n(w)t_n}]$ for all $w$. Define
$$h_\epsilon^{-1}(T)=\epsilon(T^{\epsilon}-1) := \begin{cases}T-1 & \epsilon=+1\\ 1-T^{-1}&\epsilon=-1.\end{cases}$$
It is easy to see that for $g\in R$ we have
$g_w(h_{\epsilon_1(w)}^{-1}(t_1),\ldots, h_{\epsilon_n(w)}^{-1}(t_n))\in K_T(\mathrm{pt})$ for all $w$, and, arguing as before, we see that
$$w\mapsto g_w(h_{\epsilon_1(w)}^{-1}(t_1),\ldots, h_{\epsilon_n(w)}^{-1}(t_n))$$
gives a preimage of $g$ under $\phi^B_T$. Hence $\phimap_T \colon K_T(X_{B_n})\to R$ is an isomorphism. Now, note that the $r_i$ constructed above has the property that $r_i\in R^{\times}$, so the ideal $(T_1-1,\ldots,T_n-1)\subset K_T(X_{B_n})$ maps under $\phimap_T$ to the ideal $(t_1,\ldots,t_n)\subset R$. Hence because
$$A_T(X_{B_n})\subset R\subset A_T(X_{B_n})\left [\frac{1}{1\pm t_i} \right]$$
and $\frac{1}{1\pm t_i}$ gets sent to $1$ after quotienting by $(t_1,\ldots,t_n)$,
we conclude that $\phimap$ induces an isomorphism 
\[K(X_{B_n})\cong R/(t_1,\ldots,t_n)= A_T(X_{B_n})\left [\frac{1}{1\pm t_i} \right]/(t_1,\ldots,t_n)=A^\bullet(X_{B_n}).\qedhere\]
\end{proof}

\subsection{Stellahedral geometry}\label{subsec:fakeHRR}

We show that the isomorphism $\phimap$ of \Cref{thm:Psiiso} satisfies
\[
\chi([\mathcal E]) = \int_{X_{B_n}} \phimap([\mathcal E]) \cdot c(\boxplus \mathcal O(1))
\]
for any $[\mathcal E]\in K(X_{B_n})$, thereby completing the proof of \Cref{mainthm:HRR}.  While one can prove this via the Atiyah-Bott localization formula, as in \cite{BEST}, we present a more geometric proof that explains how our result relates to a previous exceptional Hirzebruch--Riemann--Roch-type theorem given in \cite{EHL} for stellahedral varieties.
Note that $(\Sigma_{B_1})^n$ is a fan in $\RR^n$ whose cones are
\[
\operatorname{Cone}(\be_i : i \in S) \quad \text{for $S$ an admissible subset of $[n,\overline n]$}.
\]

\begin{defn}\label{defn:stella}
The \textbf{stellahedral fan} $\Sigma_{St_n}$ is a fan in $\RR^n$ obtained from $(\Sigma_{B_1})^n$ by iteratively performing stellar subdivisions on all faces of the nonpositive orthant
$\operatorname{Cone}(\be_i : i \in [\bar{n}])$
starting with the maximal face. 
\end{defn}

Note that the $B_n$ permutohedral fan $\Sigma_{B_n}$ is obtained by performing such iterated stellar subdivisions on all the orthants.  In other words, the fan $\Sigma_{B_n}$ is the common refinement of the $2^n$ different ``copies'' of the stellahedral fan:
For each admissible subset $\tau \in \ads_n$, we have the ``copy'' of the stellahedral fan obtained from $(\Sigma_{B_1})^n$ by performing the iterated stellar subdivision on the orthant $\operatorname{Cone}(\be_i : i \in \tau)$.  See \Cref{fig:stella} for an illustration when $n = 2$.

\begin{figure}[h]
\includegraphics[height = 30mm]{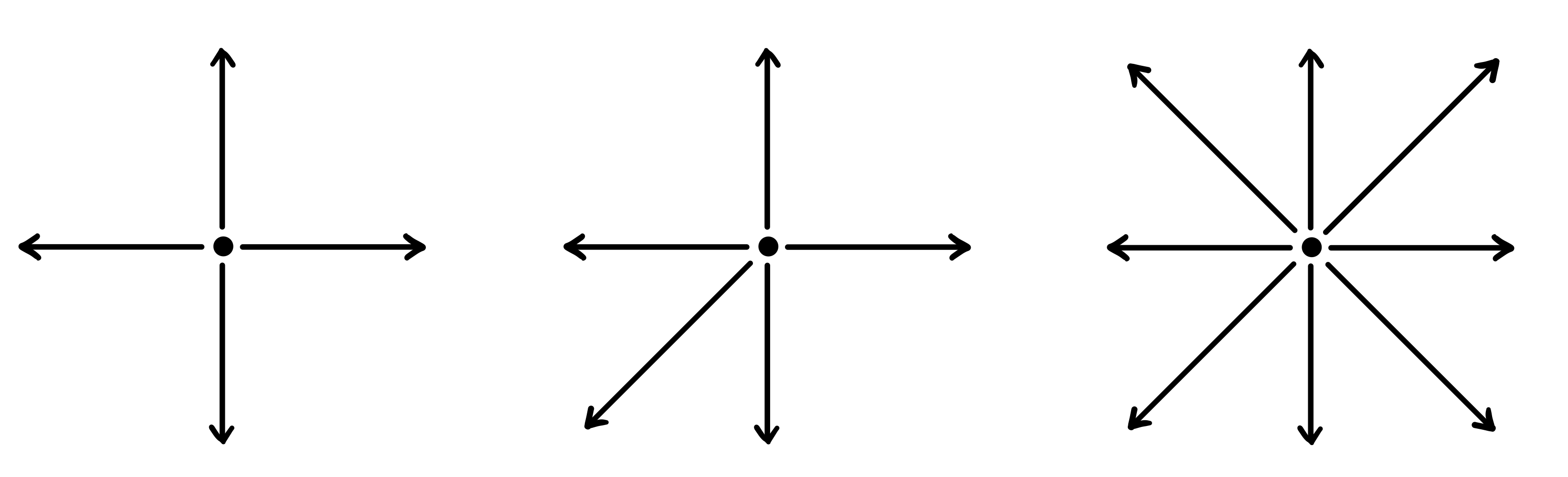}
\caption{The fans $(\Sigma_{B_1})^2$ (left), $\Sigma_{St_2}$ (middle), and $\Sigma_{B_2}$ (right)}
\label{fig:stella}
\end{figure}

The \textbf{stellahedral variety} $X_{St_n}$ is the toric variety associated to the fan $\Sigma_{St_n}$.
Since the fans $\Sigma_{B_n}$, $\Sigma_{St_n}$, and $(\Sigma_{B_1})^n$ form a sequential coarsening, we have a natural sequence of maps $X_{B_n} \to X_{St_n} \to (\PP^1)^n$ of toric varieties. The map $X_{B_n} \to X_{St_n}$ is also considered in \cite{Clader}. Recall that $\boxplus \mathcal O(1)$ denotes the vector bundle on $X_{B_n}$ that is the direct sum of the pullbacks of $\mathcal O_{\PP^1}(1)$ from each $\PP^1$ factor in~$(\PP^1)^n$.  
We reuse the notation $\boxplus \mathcal O(1)$ for the similar vector bundle pulled back only to~$X_{St_n}$.

\medskip
Stellahedral varieties play a central role in the proof the top-heavy conjecture and the nonnegativity of Kazhdan-Lusztig polynomials of matroids \cite{BHMPW20a, BHMPW20b}.  
The connection between stellahedral varieties and matroids was further developed in \cite{EHL}.
In our case, we will need the following exceptional Hirzebruch--Riemann--Roch-type theorem for stellahedral varieties.

\begin{thm}\label{thm:stellaiso} \cite[Theorem 1.9 \& Theorem 6.1]{EHL}
There is an isomorphism $\phi_T \colon K_T(X_{St_n}) \to A_T^{\bullet}(X_{St_n})[1/(1 - t_i)]$ defined by
\[
f_x(T_1, \ldots, T_n) \mapsto f_x( \textstyle\frac{1}{1-t_1}, \ldots, \frac{1}{1-t_n})
\]
where $f_x(T_1, \ldots, T_n)\in \ZZ[T_1^{\pm 1}, \ldots, T_n^{\pm 1}]$ is the localization value of a $K$-class $f\in K_T(X_{St_n})$ at a $T$-fixed point $x$ of $X_{St_n}$.
It descends to an isomorphism $\phi \colon K(X_{St_n}) \to A^{\bullet}(X_{St_n})$ which satisfies
\[
\chi([\mathcal E]) = \int_{X_{St_n}} \phi([\mathcal E]) \cdot c(\boxplus \mathcal O(1)) \quad\text{for any $[\mathcal E]\in K(X_{St_n})$}.
\]
\end{thm}

The isomorphism $\phimap$ of \Cref{thm:Psiiso} is an extension of this isomorphism $\phi$ as follows. 

\begin{lem}\label{lem:stellcommute}
Let $p \colon X_{B_n} \to X_{St_n}$ be the toric morphism described above.  The following diagram commutes:
\begin{center}
\begin{tikzcd}
K(X_{St_n}) \arrow[r, "\phi"] \arrow[d, "p^*"]
& A^{\bullet}(X_{St_n}) \arrow[d, "p^*"] \\
K(X_{B_n}) \arrow[r, "\phimap"]
& A^{\bullet}(X_{B_n}). 
\end{tikzcd}
\end{center}
\end{lem}

\begin{proof}
For a matroid $\M$ on $[n]$, its independence polytope $IP(\M)$ is a deformation of $\Sigma_{St_n}$ and hence defines a class $[IP(\M)]$ in the polytope algebra $\indgp(\mathscr P_{\ZZ,\Sigma_{St_n}})$  \cite[Example 3.15]{EHL}.  Moreover, the set $\{[IP(\M)] : \M \text{ a matroid on $[n]$}\}$ spans $\indgp(\mathscr P_{\ZZ,\Sigma_{St_n}})$ as an abelian group \cite[Proposition 7.4]{EHL}, which is isomorphic to $K(X_{St_n})$ via \Cref{thm:indK}.  Hence, it suffices to show the commutativity of the diagram on the spanning set $\{[IP(\M)] : \text{ $\M$ a matroid on }[n]\}$.
Now, for $i\in [n]$ and any maximal cone $\sigma$ of $\Sigma_{B_n}$ containing $\be_i$, the $T$-equivariant localization value of $[IP(\M)]$ at $\sigma$ is a Laurent polynomial in the variables $T_j$ for $j\ne i$, 
because the vertex of $IP(\M)$ minimizing the standard pairing with a vector in the interior of $\sigma$ has zero as its $i$th coordinate.  By the descriptions of the maps $\phi_T$ and $\phimap_T$, this implies that $p^* \phi_T([IP(\M)]) = \phimap_T([IP(\M)])$ for any matroid $\M$ on $[n]$.
\end{proof}

We caution that the torus-equivariant analogue of the above diagram does not commute.
We can now finish the proof of Theorem~\ref{mainthm:HRR}.

\begin{proof}[Proof of Theorem~\ref{mainthm:HRR}]

We have shown that $\phi^B$ is an isomorphism in \Cref{thm:Psiiso}.
It remains to show the Hirzebruch--Riemann--Roch-type formula
\[
\chi([\mathcal E]) = \int_{X_{St_n}} \phimap([\mathcal E]) \cdot c(\boxplus \mathcal O(1)) \quad\text{for any $[\mathcal E]\in K(X_{B_n})$}.
\]
\Cref{thm:basis} implies that $K(X_{B_n})$ is generated as an abelian group by Weyl images of independence polytopes of matroids.  Hence, it suffices to check the Hirzebruch--Riemann--Roch-type formula for Weyl images of independence polytopes of matroids. 
Moreover, by Weyl-equivariance of $\phimap$, it suffices to check this for independence polytopes of matroids.
Then this follows from the projection formula, Theorem~\ref{thm:stellaiso}, and Lemma~\ref{lem:stellcommute}.
\end{proof}

\begin{rem}\label{rem:barriers}
There are two obstructions to establishing analogues of Theorems \ref{mainthm:HRR} and~\ref{mainthm:isoms} for arbitrary root systems.
First, Propositions \ref{prop:intersect with a cube} and~\ref{prop:positive root cone} about intersections with the unit cube, which were essential to our proof of \Cref{mainthm:isoms}, 
no longer hold when the unit cube is replaced by (minuscule) weight polytopes of types other than $A$ and $B$, for instance in type $D$.  See \cite[Remark 3.15]{ESS}.
Second, the useful feature of $\Sigma_{B_n}$ in the construction of the map $\phimap_T$ in \Cref{thm:Psiiso} and in the proof of \Cref{mainthm:HRR} is that $\Sigma_{B_n}$ can be viewed as a common refinement of $2^n$ ``copies'' of the stellahedral fan $\Sigma_{St_n}$.  For arbitrary crystallographic root systems $\Phi$, we do not know whether $K(X_{\Phi})$ and $A^{\bullet}(X_{\Phi})$ are integrally isomorphic.
\end{rem}

In Section~\ref{ssec:intersection}, we will make use of the following ``dual'' version of $\phimap$. For a variety $X$, define the ring involution $D_K \colon K(X) \to K(X)$ by $[\mathcal{E}] \mapsto [\mathcal{E}^{\vee}]$ and the ring involution $D_A \colon A^\bullet(X) \to A^\bullet(X)$ by multiplication by $(-1)^d$ in degree $d$. Define the ``dual'' isomorphism $\zetamap \colon K(X_{B_n}) \to A^{\bullet}(X_{B_n})$ by $D_A \circ \phimap \circ D_K$. Similarly define $\zeta^B_T$. The isomorphism $\zetamap$ satisfies the following Hirzebruch--Riemann--Roch-type formula.  To state it, let $\gamma \in A^1(X_{B_n})$ be the divisor class on $X_{B_n}$ corresponding to the $n$-dimensional \textbf{cross polytope}, which is the $B_n$ generalized permutohedron $\Diamond = \operatorname{Conv}(\be_i : i \in [n,\overline n])\subset \RR^n$.

\begin{prop}\label{prop:otherHRR}
For any $[\mathcal E]\in K(X_{B_n})$, one has
\[
\chi([\mathcal E]) = \int_{X_{B_n}} \zetamap([\mathcal E]) \cdot c(\boxplus \mathcal O(-1)) \cdot (1 + \gamma + \cdots + \gamma^n).
\]
\end{prop}

\begin{proof}
A primitive vector in a ray of $\Sigma_{B_n}$ is $\be_S$ for some nonempty admissible subset $S$ of $[n,\overline n]$.  We note that the minimum of the standard pairing $\langle x, \be_S\rangle$ for $x \in \Diamond$ is $-1$.  
Under the standard correspondence between polytopes and base-point-free divisors on toric varieties that we have been using, this means that $\gamma$ is the sum of all boundary divisors on $X_{B_n}$.  In other words, by \cite[Theorem 8.1.6]{CLS}, the line bundle $\mathcal O(-\gamma)$ is the canonical bundle of $X_{B_n}$.
Applying Serre duality along with $\phimap = D_A \circ \zetamap \circ D_K$ to \Cref{mainthm:HRR}, we have that
\begin{align*}
\chi([\mathcal E]) &= (-1)^n \chi \big([\mathcal O(-\gamma)] \cdot D_K([\mathcal E]) \big) \\
&=  (-1)^n \int_{X_{B_n}} \phimap\big([\mathcal O(-\gamma)] \cdot D_K([\mathcal E])\big) \cdot c(\boxplus \mathcal O (1)) \\
&= (-1)^n \int_{X_{B_n}} D_A\big( \zetamap([\mathcal O(\gamma)]\cdot [\mathcal E]) \cdot c(\boxplus \mathcal O(-1))\big)\\
&= \int_{X_{B_n}} \zetamap([\mathcal O(\gamma)]) \cdot \zetamap([\mathcal E]) \cdot c(\boxplus \mathcal O(-1)).
\end{align*}

It suffices now to show that $\zetamap([\mathcal{O}(\gamma)]) = 1 + \gamma + \cdots + \gamma^n$. For this, we compute using torus-equivariant localization. 
For $w \in \W$ such that $\face{v}{\Diamond} = -\be_i$ for any $v \in C_w^{\circ}$, we have that $[\mathcal{O}(\gamma)]_w = T_i$. For such $w$, we must have that $i \in w([n])$, so this maps to $1/(1 - t_i)$ under $\zeta^B_T$. If $\face{v}{\Diamond} = \be_i$,  $[\mathcal{O}(\gamma)]_w = T_i^{-1}$, and we must have $i \not \in w([n])$, so this maps to $1/(1 + t_i)$ under $\zeta^B_T$.
We thus see that $\zetamap([\mathcal{O}(\gamma)])= c(\mathcal{O}(-\gamma))^{-1}=1 + \gamma + \cdots + \gamma^n$, as desired. 
\end{proof}

We now introduce a set of equivariant $K$-classes that is inspired by \cite[Definition 10.4]{BEST}. Say that a class $[\mathcal{E}] \in K_T(X_{B_n})$ has ``nice Chern roots'' if, on the maximal cone corresponding to $w \in \W$, we have $[\mathcal{E}]_{w} = a_{w, 0} + \sum_{i \in w([n])} a_{w, i} T_i^{-1} - \sum_{i \not \in w([n])} a_{w, i} T_i.$ 

We first define some notation. For $[\mathcal{E}] \in K_T(X_{B_n})$, let $c^T([\mathcal{E}], u) = c_0^T([\mathcal{E}]) + c_1^T([\mathcal{E}])u + \dotsb \in A^{\bullet}_T(X_{B_n})[u]$ be the equivariant Chern polynomial. The  equivariant Segre power series $s^T([\mathcal{E}], u)=s_0^T([\mathcal{E}])+s_1^T([\mathcal{E}])u+\cdots \in A^{\bullet}_T(X_{B_n})[[u]]$ is defined by $s^T([\mathcal{E}], u):=c^T([\mathcal{E}], u)^{-1}$.  Recall that the map that assigns a vector bundle $\mathcal{E}$ to its rank extends to a map $\operatorname{rk} \colon K(X_{B_n}) \to \mathbb{Z}$. If we write $[\mathcal{E}]_{w} = \sum_{i=1}^{k_{w}} a_{w, i}T^{m_{w, i}}$, then, with $u$ a formal variable, we have that
$$\sum_{j=0}^{\infty}  \textstyle\bigwedge^j \displaystyle[\mathcal{E}]_{w} u^j = \prod_{i=1}^{k_{w}}(1 + T^{m_{w, i}}u)^{a_{w, i}}, \text{ and } \sum_{j=0}^{\infty} \operatorname{Sym}^j[\mathcal{E}]_{w}u^j = \prod_{i=1}^{k_{w}} \left(\frac{1}{1 - T^{m_{w, i}}u}\right )^{a_{w, i}}.$$

\begin{prop}\label{prop:niceChern}
If $[\mathcal{E}]$ has nice Chern roots, then 
\begin{align*}\sum_{i \ge 0} \zetamap_T(\textstyle\bigwedge^i[ \mathcal{E}]) u^i &= (u + 1)^{\operatorname{rk}(\mathcal{E})} \,c^T\left ([\mathcal{E}], \frac{u}{u + 1} \right),\\
\sum_{i \ge 0} \phimap_T(\textstyle\bigwedge^i[ \mathcal{E}]) u^i &= (u + 1)^{\operatorname{rk}(\mathcal{E})} \,s^T ([\mathcal{E}]^{\vee}) \,c^T\left([\mathcal{E}]^{\vee}, \frac{1}{u + 1}\right),\\
\sum_{j \ge 0} \zetamap_T (\operatorname{Sym}^j [\mathcal{E}])u^j &=  \frac{1}{(1 - u)^{\operatorname{rk}(\mathcal{E})}} \,s^T\left([\mathcal{E}], \frac{u}{u - 1}\right), \text{ and}\\
\sum_{j \ge 0} \phimap_T (\operatorname{Sym}^j [\mathcal{E}])u^j &= \frac{c^T([\mathcal{E}]^{\vee})}{(1 - u)^{\operatorname{rk}(\mathcal{E})}} \,s^T\left ([\mathcal{E}]^{\vee}, \frac{1}{1 - u}\right ).\end{align*}
\end{prop}

\begin{proof}
We prove the formulas involving $\phimap$; the formulas involving $\zetamap$ are similar. Consider a maximal cone corresponding to $w \in \W$, and write 
\[[\mathcal{E}]_{w} = a_{w, 0} + \sum_{i \in w([n])} a_{w, i} T_i^{-1} - \sum_{i \not \in w([n])} a_{w, i} T_i.\] 
Then
\begin{align*}
\sum_{i \ge 0} \phimap_T(\textstyle\bigwedge^i[ \mathcal{E}])_{w} u^i & = (u + 1)^{a_{w, 0}}\prod_{i \in w([n])}(1 + \phimap_T(T_i^{-1})u)^{a_{w, i}} \prod_{i \not \in w([n])}(1 + \phimap_T(T_i)u)^{a_{w, i}} \\ 
&=(u + 1)^{a_{w, 0}} \prod_{i \in w([n])}(1 + (1 + t_i)^{-1}u)^{a_{w, i}} \prod_{i \not \in w([n])}(1 + (1 - t_i)^{-1}u)^{a_{w, i}} \\ 
&= (u + 1)^{\operatorname{rk}(\mathcal{E})} \prod_{i \in w([n])} (1 + t_i)^{-a_{w, i}}\left (1 + \frac{t_i}{u+1} \right) \prod_{i \not \in w([n])} (1 - t_i)^{-a_{w, i}} \left(1 - \frac{t_i}{u+1}\right) \\
&= (u + 1)^{\operatorname{rk}(\mathcal{E})} s^T([\mathcal{E}]^{\vee})_w \,c^T\left ([\mathcal{E}]^{\vee}, \frac{1}{u + 1} \right)_w.
\end{align*}
Similarly, we compute
\begin{align*}
\sum_{j \ge 0} \phimap_T (\operatorname{Sym}^j [\mathcal{E}])_wu^j &= \frac{1}{(1 - u)^{a_{w, 0}}}\prod_{i \in w([n])} \left ( \frac{1}{1 - \phimap_T(T_i^{-1})u} \right)^{a_{w, i}} \rlap{$\displaystyle\prod_{i \not \in w([n])} \left ( \frac{1}{1 - \phimap_T(T_i)u} \right)^{a_{w, i}}$} \\
&=  \frac{1}{(1 - u)^{a_{w, 0}}} \prod_{i  \in w([n])}  \frac{1}{(1 - (1 + t_i)^{-1}u)^{a_{w, i}}} \rlap{$\displaystyle\prod_{i \not \in w([n])}  \frac{1}{(1 - (1 - t_i)^{-1}u)^{a_{w, i}}}$} \\ 
&= \frac{1}{(1 - u)^{\operatorname{rk}(\mathcal{E})}} \prod_{i \in w([n])} \frac{ 1 + t_i}{1 + t_i/(1 -u)}  \prod_{i \not \in w([n])} \frac{ 1 - t_i}{1 - t_i/(1 -u)} \\ 
&= \frac{c^T([\mathcal{E}]^{\vee})}{(1 - u)^{\operatorname{rk}(\mathcal{E})}} \,s^T\left ([\mathcal{E}]^{\vee}, \frac{1}{1 - u} \right).\qedhere
\end{align*}
\end{proof}

\section{The Chow cohomology ring of $X_{B_n}$}
In this section, we first combine Theorems \ref{mainthm:HRR} and~\ref{mainthm:isoms} to obtain a basis for the Chow cohomology ring $A^\bullet(X_{B_n})$. 
We then prove Theorem~\ref{mainthm:BnGP} by using the Hirzebruch--Riemann--Roch-type formula that $\phi^B$ satisfies.

\subsection{A Schubert basis}
We now describe the structure of the Chow cohomology ring $A^\bullet(X_{B_n})$ in terms of ``augmented Bergman classes'' of matroids.
Let $\M$ be a matroid of rank $r$ on $[n]$. The \textbf{augmented Bergman fan} of $\M$ is a subfan $\Sigma_\M$ of the stellahedral fan $\Sigma_{St_n}$ obtained by gluing together the order complex of lattice of flats and the independence complex of $\M$; for a precise definition see \cite[Definition 2.4]{BHMPW20a}.  
Assigning weight 1 to each of its maximal cones defines a Minkowski weight $[\Sigma_\M]$, called the \textbf{augmented Bergman class} of $\M$, which can be considered as an element in $A^{n-r}(X_{St_n})$. Augmented Bergman classes are nef Chow classes, and they span extremal rays of the cone of nef classes in $A^{n-r}(X_{St_n})$ \cite[Proposition 2.8]{BHMPW20a}. 

We will consider the pullbacks of augmented Bergman classes to $X_{B_n}$ under the morphism $p \colon X_{B_n} \to X_{St_n}$ described in Section~\ref{subsec:fakeHRR}. These pullbacks continue to span extremal rays of the cone of nef classes in $A^{\bullet}(X_{B_n})$. We will also refer to these pulled back classes as augmented Bergman classes. For a matroid $\M$, let $\M^{\perp}$ be the dual matroid.  Only two properties of augmented Bergman classes will be used in the rest of the paper: 
\begin{enumerate}
\item For any matroid $\M$, the class $[\Sigma_{\M}]$ is nonzero.
\item When $\M$ has rank $n - 1$, the class $[\Sigma_{\M}]$ is the first Chern class of the line bundle corresponding to the simplex $IP(\M^{\perp})$. 
\end{enumerate}

We now introduce some terminology. Say that a delta-matroid $\D$ with feasible sets $\mathcal F$ is \newword{standard cornered}
if, whenever $B\in\mathcal F$ and $i\in B\cap[n]$, then $B\setminus\{i\}\cup\{\overline i\}\in\mathcal F$.
For example, delta-matroids of the form $IP(\M)$ are standard cornered. In fact this is the only example. 
\begin{lem}\label{lem:saturated}
Any standard cornered delta-matroid is of the form $IP(\M)$ for a matroid~$\M$.
\end{lem}

\begin{proof}
We show the matroid independent set axioms for $\mathcal I=\{B\cap[n]:B\in\mathcal F\}$.
By assumption, $\mathcal I$ is a nonempty family of sets closed under taking subsets,
so we must prove the independent set augmentation axiom.
Let $A,B\in\mathcal F$ with $|A\cap[n]|<|B\cap[n]|$. 
Let $F$ be the smallest face of $[0,1]^n$ containing $a=\be_{A\cap[n]}$ and $b=\be_{B\cap[n]}$.
We have that $P(\D)\cap F$ is a delta-matroid polytope.
Let $C$ be the vertex cone of~$a$ in $P(\D)\cap F$ (with the apex of~$C$ at the origin).
Then $C$ contains $b-a$ and is generated by type~$B_n$ roots.
Because $b-a$ has strictly positive sum of coordinates, 
$C$ must have a generator with strictly positive sum of coordinates, either $\be_i$ or $\be_i+\be_j$ for some $i,j\in[n]$.
So either $a+\be_i$ or $a+\be_i+\be_j$ lies in $P(\D)\cap F$;
because $\D$ is standard cornered, the latter case implies the former one.
By the choice of~$F$, the element $i$ lies in $B\setminus A$,
and hence $(A\cap [n])\cup\{i\}\in\mathcal I$.
\end{proof}
Say that a delta-matroid $\mathrm{C}$ is \textbf{cornered} if there is $w \in \W$ such that $w \cdot \mathrm{C}$ is standard cornered. We now develop some properties of cornered delta-matroids. 

\begin{lem}\label{lem:degrees}
Let $\M$ be a matroid of rank $r$ on $[n]$.  Then the degree $i$ part of $\phi^B([IP(\M)])$ vanishes for $i > r$,  is equal to $[\Sigma_{\M^{\perp}}]$ in degree $r$, and is $1$ in degree $0$. 
\end{lem}

\begin{proof}
That $\phi \colon K(X_{St_n}) \to A^{\bullet}(X_{St_n})$ has this property follows from \cite[Lemma 5.9]{EHL}. Then the result follows from Lemma~\ref{lem:stellcommute}. 
\end{proof}

\begin{lem}\label{lem:rankdefined}
Let $\M_1, \M_2$ be matroids on $[n]$, and suppose that $w_1 \cdot [IP(\M_1)] = w_2 \cdot [IP(\M_2)]$ for some $w_1, w_2 \in \W$. Then the rank of $\M_1$ is equal to the rank of $\M_2$, and $w_1 \cdot [\Sigma_{\M_1^{\perp}}] = w_2 \cdot [\Sigma_{\M_2^{\perp}}]$. 
\end{lem}

\begin{proof}
By the $\W$-equivariance of $\phi^B$, we must have that $w_1 \cdot [\Sigma_{\M_1^{\perp}}] = w_2 \cdot [\Sigma_{\M_2^{\perp}}]$. Lemma~\ref{lem:degrees} identifies the rank of $\M$ as the degree of the top nonzero piece of $\phi^B([IP(\M)])$. 
\end{proof}

In particular, if $\mathrm{C}=w\cdot IP(\M)$ is a cornered delta-matroid, then we define the \textbf{cornered rank}  $\operatorname{rk}_{\rm cor}(\mathrm{C})$ as the rank of $\M$, which is independent of the choice of $\M$ and $w$, and we define  
$$[\Sigma_{\mathrm{C}}]:=w \cdot [\Sigma_{\M^{\perp}}].$$  Note that $[\Sigma_{IP(\M^{\perp})}]=[\Sigma_\M]$. The following is an immediate consequence of \Cref{lem:degrees}.
\begin{lem}
\label{lem:phiweaksat}
Let $\mathrm{C}$ be a cornered delta-matroid. Then the degree $i$ part of $\phi^B([\mathrm{C}])$ vanishes for $i>\operatorname{rk}_{\rm cor}(\mathrm{C})$, is equal to $[\Sigma_{\mathrm{C}}]$ in degree $\operatorname{rk}_{\rm cor}(\mathrm{C})$, and is $1$ in degree $0$.
\end{lem}
Now we construct our basis for $A^\bullet(X_{B_n})$, noting that Schubert delta-matroids are cornered.

\begin{cor}\label{cor:basisChow}
For any $0\leq r \leq n$,
$$\{[\Sigma_{\mathrm{C}}]: \mathrm{C}\in \mathsf{SchDMat}_n^{\mathsf{clf}} \text{ and } \operatorname{rk}_{\rm cor}(\mathrm{C})=r\}$$
is a basis for $A^r(X_{B_n})$. 
\end{cor}

\begin{proof}
Endow $K(X_{B_n})$ with a grading by declaring the $r$th graded piece to be generated by the elements $\{[P(\mathrm{C})]\colon \mathrm{C} \in \mathsf{SchDMat}_n^{\mathsf{clf}}, \operatorname{rk}_{\rm cor}(\mathrm{C})=r\}$; this is well-defined by Theorem~\ref{mainthm:isoms}.
Combining \Cref{thm:basis} with \Cref{mainthm:HRR}, we have that $\{\phimap([P(\mathrm{C})]) : \mathrm{C}\in \mathsf{SchDMat}_n^{\mathsf{clf}}\}$ is a basis of $A^\bullet(X_{B_n})$. By \Cref{lem:phiweaksat}, $\phi^B$ is lower-triangular with respect to the gradings on $K(X_{B_n})$ and $A^{\bullet}(X_{B_n})$ and the degree $r$ part of $\phi^B([\mathrm{C}])$ is $[\Sigma_{\mathrm{C}}]$, so we conclude.
\end{proof}

Setting $r = 1$ in the corollary yields \Cref{mainthm:BnGP}\ref{BnGP:1} as follows.

\begin{proof}[Proof of \Cref{mainthm:BnGP}\ref{BnGP:1}]
The polytope of a delta-matroid in $\mathsf{SchDMat}_n^{\mathsf{clf}}$ of cornered rank $1$
is a translate of a simplex $\Delta_S^0$ for $S\in\ads\setminus\{\emptyset\}$, and vice versa.
Namely, $P(\Omega_{[\overline n]\setminus\{\overline i\}\cup\{i\}}) = \Delta_{\{1,\ldots,i\}}^0$, 
and if $\D=w\cdot\Omega_{[\overline n]\setminus\{\overline i\}\cup\{i\}}$, then $P(\D) = w\cdot P(\Omega_{[\overline n]\setminus\{\overline i\}\cup\{i\}})$
differs from $\Delta_{w\cdot\{1,\ldots,i\}}^0 = w\cdot\Delta_{\{1,\ldots,i\}}^0$
only by the translations that distinguish the $\W$-action on delta-matroid polytopes from the $\W$-action on $\mathbb R^n$ in \Cref{ssec:FanBn}.
No two simplices $\Delta_S^0$ are translations of each other except for the pairs of line segments $\{\Delta_{\{i\}}^0, \Delta_{\{\overline i\}}^0\}$.
Hence, setting $r = 1$ in \Cref{cor:basisChow}, we have that the set
\[
\{\text{the divisor class associated to $\Delta_S^0$} : \text{$S \in \ads\setminus\{\emptyset\}$ and $S\ne\{\overline i\}$ for $i\in[n]$}\}
\]
is a basis of $A^1(X_{B_n})$.
Thus, up to translation by a vector in $\ZZ^n$, every $B_n$ generalized permutohedron is a signed Minkowski sum of the simplices $\Delta_S^0$ in the displayed set.  
Since $\Delta_{\{\overline i\}}^0 =  \Delta_{\{i\}}^0 - \be_i$, reinserting the segments $ \Delta_{\{\overline i\}}^0$ into the set accounts for the translations.
\end{proof}

\begin{rem}
The $h$-vector of the Coxeter complex $\Sigma_\Phi$ of a root system $\Phi$, or, equivalently, the sequence of dimensions of the graded pieces of $A^\bullet(X_\Phi)$, is equal to the vector of $\Phi$-Eulerian numbers \cite{Bjo84,Bre94}, which are defined in terms of the descents of elements in the Coxeter group associated to $\Phi$.
Concretely, in type $B$ the set of descents of an element $w\in \W$ is
\[
\operatorname{des}(w) = \{i\in [n] : w(i-1) > w(i)\},
\]
where we define $w(0) = 0$ to fit into the total order as $\overline n < \cdots <\overline 1 < 0 < 1 < \cdots < n$.  The $r$th $B_n$ Eulerian number is then
\[
h_r(B_n) := |\{w\in \W : \operatorname{des}(w) = r\}|.
\]
In particular, \Cref{cor:basisChow} implies that the $B_n$ Eulerian numbers count the coloop-free Schubert delta-matroids of cornered rank~$r$.
An analogous statement for type $A$ was shown in \cite{Ham17}.  In neither type~$A$ nor type~$B$ do we know of a natural bijection between the set of Weyl group elements with a fixed number of descents and the corresponding set of coloop-free Schubert (delta-)matroids.
\end{rem}

\subsection{Volumes and lattice point enumerators}\label{subsec:volEhr}

We now compute volumes and lattice point counts
of $B_n$ generalized permutohedra by using \Cref{mainthm:HRR}.
We will use the following observation throughout.
For an admissible subset $S\in \ads$, let $h_S$ be the divisor class on $X_{B_n}$ associated to the simplex $\Delta_S^0$.
Because simplices are Weyl images of the independence polytopes of standard Schubert matroids of cornered rank 1, \Cref{lem:phiweaksat} implies that $\phimap([\Delta_S^0]) = 1 + h_S$.

\begin{proof}[Proof of \Cref{mainthm:BnGP}\ref{BnGP:2}]
For a sequence $(S_1, \ldots, S_n)$ of $n$ admissible subsets, standard results in toric geometry \cite[\S5.4]{Ful93} imply that the mixed volume of the corresponding simplices is the intersection product $\int_{X_{B_n}} h_{S_1} \cdots h_{S_n}$, which is equal to
\[
\int_{X_{B_n}} (1+h_{S_1}) \cdots (1+h_{S_n}) = \int_{X_{B_n}} \phimap([\Delta_{S_1}^0] \cdots [\Delta_{S_n}^0]) = \int_{X_{B_n}} \phimap([\Delta_{S_1}^0 + \cdots +\Delta_{S_n}^0]) .
\]
Let $P$ be the Minkowski sum $\Delta_{S_1}^0 + \cdots +\Delta_{S_n}^0$.  
By construction, the polytope $P$ is ``saturated towards the origin'' in the following sense:  For any subset $S \subseteq [n]$, let $\mathrm{Orth}_S =  \RR^S_{\geq 0}\times \RR^{[n]\setminus S}_{\leq 0}$.  If $u\in P \cap \operatorname{Orth}_S$, then any $v\in \operatorname{Orth}_S$ such that $u-v \in \operatorname{Orth}_S$ is also in $P$.
We tile $\RR^n$ by lattice translates of the unit cube $\square = [0,1]^n$, and express
\begin{multline*}
[P] = \Big(\sum_{m\in \ZZ^n} [P\cap (m+\square)]\Big) \\
+ \text{a linear combination of $\{[P \cap (m+F)] : m\in \ZZ^n,\ \text{$F$ a proper face of $\square$}\}$}
\end{multline*}
Every intersection $P\cap (m+\square)$ or $P\cap (m+F)$ in the expression is a translate of a delta-matroid polytope by \Cref{prop:intersect with a cube}.
Because $P$ is saturated towards the origin, these delta-matroid polytopes are cornered by Lemma~\ref{lem:saturated}. For such a delta-matroid $\mathrm{C}$, by \Cref{lem:phiweaksat} we have $\int_{X_{B_n}}\phimap([P(\mathrm{C})])=0$ when $P(\mathrm{C})\ne \square$.
When $P(\mathrm{C})=\square$ we have
\[
\int_{X_{B_n}}\phimap([\square]) = \int_{X_{B_n}}\phimap([\Delta_{\{1\}}^0]\cdots[\Delta_{\{n\}}^0]) = \int_{X_{B_n}} (1 + h_{\{1\}}) \cdots (1 + h_{\{n\}}) = 1.
\]
We have thus reduced to counting the number of $m$ such that $P \cap (m + \square) = m + \square$.  This happens only when $m+\square$ contains the origin, since each simplex is contained in the cross-polytope $\Diamond$, so $P\subset n\Diamond$, and every integral translate of $\square$ contained in $n\Diamond$ contains the origin.  In other words, we are counting the set of cardinality-$n$ admissible subsets $\tau \in \ads_n$ such that $\be_\tau \in P$.  This set, by the construction of $P$, is in bijection with the set of signed transversals of $(S_1, \ldots, S_n)$.
\end{proof}

\begin{proof}[Proof of \Cref{mainthm:BnGP}\ref{BnGP:3}]
Denote by $\ads^{\notin [n]}$ the subset $\{S\in \ads : |S|>1 \text{ or } S = \{\overline i \} \subset [\overline n]\}$ of admissible subsets of $[n,\overline n]$.  Note that the divisor class on $X_{B_n}$ corresponding to the cube $\square=[0,1]^n$ is $h_{\{1\}} + \cdots + h_{\{n\}}$.  By standard results in toric geometry \cite[\S3.5]{Ful93}, the quantity 
\[\Big(\text{\# lattice points of } \big(P(\{c_S\})-\square\big)\Big)\]
is computed by the Euler characteristic
\[
\chi\Big(\big[\sum_{S\in \ads^{\notin [n]}} c_S \Delta_S^0+\sum_{i \in [n]} (c_{i} - 1)\Delta_{\{i\}}^0\big] \Big).
\]
Noting that
$c(\boxplus\mathcal O(1)) = \prod_{i\in [n]} (1+h_{\{i\}})$, we apply \Cref{mainthm:HRR} to obtain
\begin{align*}
\chi\Big( & [\sum_{S\in \ads^{\notin [n]}} c_S \Delta_S^0 +\sum_{i \in [n]} (c_{i} - 1)\Delta_{\{i\}}^0] \Big) \\
&= \int_{X_{B_n}} \prod_{S\in \ads^{\notin [n]}} \phimap([\Delta_S^0])^{c_S}\cdot \prod_{i\in [n]} \phimap([\Delta_{\{i\}}^0])^{c_i - 1}\cdot c(\boxplus\mathcal O(1)) \\
&= \int_{X_{B_n}} \prod_{S\in \ads^{\notin [n]}} (1+h_S)^{c_S}\cdot \prod_{i\in [n]} (1+h_{\{i\}})^{c_i}\\
&= \int_{X_{B_n}} \prod_{S\in \ads\setminus\{\emptyset\}} \left(\sum_{k=0}^n \binom{c_S}{k} h_S^{c_S}\right)
\\
&= \Psi\Big( \operatorname{Vol}\big(\sum_{S\in \ads\setminus \{\emptyset\}} c_S \Delta_S^0\big)\Big),
\end{align*}
as desired.
\end{proof}

Finally, we note that the mixed volume computation above can be generalized to arbitrary cornered delta-matroids as follows. 
\begin{thm}
Let $\mathrm{C}_1,\ldots,\mathrm{C}_k$ be cornered delta-matroids with $\sum \operatorname{rk}_{\rm cor}(\mathrm{C}_i)=n$, and write $\mathrm{C}_i=w_i\cdot IP(\M_i)$. Then we have
\begin{multline*}
\int_{X_{B_n}} [\Sigma_{\mathrm{C}_1}] \cdots [\Sigma_{\mathrm{C}_k}] = \\
\left|\left\{\tau\in \ads_n \ \middle| \ \begin{matrix} \tau \text{ a signed transversal of $(w_1\cdot B_1, \ldots, w_1\cdot B_1, \ldots, w_k\cdot B_k, \ldots, w_k\cdot B_k)$}\\ \text{where $B_i$ is a basis of $\M_i$ and $w_i\cdot B_i$ is repeated $\operatorname{rk}_{\rm cor}(\mathrm{C}_i)$ times}\end{matrix}\right\}\right|.
\end{multline*}
\end{thm}
\begin{proof}
The argument is similar to the proof of \Cref{mainthm:BnGP}\ref{BnGP:2}, so we sketch only the main steps.
By \Cref{mainthm:HRR} and \Cref{lem:phiweaksat}, we have
\[
\int_{X_{B_n}} [\Sigma_{\mathrm{C}_1}] \cdots [\Sigma_{\mathrm{C}_k}] = |\{m \in \ZZ^n : \mathrm{C}_1+\cdots+\mathrm{C}_k \supseteq (m + \square)\}|
\]
where $\square = [0,1]^n$. Write $w_iIP(\M_i)$ for the image of the polytope $IP(\M_i)$ under the isometry associated to $w_i$ for the standard geometric action of $\W$ on $\mathbb{R}^n$. Then  $P(\mathrm{C}_1)+\cdots+P(\mathrm{C}_k)$ is an integral translate of $P=w_1 IP(\M_1)+\cdots+w_kIP(\M_k)$, so we may equivalently compute $|\{m \in \ZZ^n : P \supseteq (m + \square)\}|$. Because $w_i IP(\M_i)\subset \operatorname{rk}_{\rm cor}(\mathrm{C}_i)\Diamond$ for the cross-polytope $\Diamond$, we have $P\subseteq (\sum \operatorname{rk}_{\rm cor}(\mathrm{C}_i))\Diamond=n\Diamond$. Hence for $P \supseteq (m+\square)$, we must have that $n\Diamond\supset m+\square$ so $m+\square$ contains the origin. Hence, we are counting the number of $\tau\in \ads_n$ such that $\be_\tau \in P$.  The desired formula follows.
\end{proof}

\begin{cor}
For a matroid $\M$ of rank $r$ and admissible subsets $S_1, \ldots, S_r \in \ads$, we have
\[
\int_{X_{B_n}} [\Sigma_\M] \cdot h_{S_1} \cdots h_{S_r} = 
\left|\left\{\tau\in \ads_n \ \middle| \ \begin{matrix}\text{$\tau$ a signed transversal of $(S_1, \ldots, S_r, B, \ldots, B)$}\\ \text{ for some basis $B$ of $\M^\perp$}\end{matrix}\right\}\right|.
\]
\end{cor}

\section{Tutte-like invariants of delta-matroids}\label{sec:invardelta}

We first recall some combinatorial operations on delta-matroids.
In the context of multi-matroids, these operations can be found in \cite{BouShelter}.

\begin{defn}\label{defn:operations}
Let $\D$ be a delta-matroid on $[n, \bar{n}]$, and let $i \in [n]$.
We define three delta-matroids on $[n, \bar{n}] \setminus \{i, \bar{i}\}$ obtained from $\D$ as follows:
\begin{enumerate}
\item If $i$ is not a loop, the \newword{contraction} $\D/ i$ is the delta-matroid with feasible sets $B \setminus i$ for $B$ a feasible set of $\D$ containing $i$.
\item If $i$ is not a coloop, the \newword{deletion} $\D \setminus i$ is the delta-matroid with feasible sets $B \setminus \bar{i}$ for $B$ a feasible set of $\D$ containing $\bar{i}$. 
\item We define the \newword{projection} $\D(i)$ as the delta-matroid with feasible sets $B \setminus \{i, \bar{i}\}$ for $B$ a feasible set of $\D$. 
\item If $i$ is a loop (resp.\ coloop), we define $\D/ i = \D \setminus i$ (resp.\ $\D \setminus i = \D/ i$), so that $\D/i = \D\setminus i = \D(i)$. 
\end{enumerate}
\end{defn}

If $i$ is not a loop (resp.\ a coloop), then $P(\D/i)$ (resp.\ $P(\D\setminus i)$) is obtained by intersecting $P(\D)$ with the hyperplane $x_i = 0$ (resp.\ $x_i = 1$).
We obtain $P(\D(i))$ by taking the orthogonal projection of $P(\D)$ onto $x_i = 0$. Therefore projections commute with each other and commute with deletion and contraction. For $I \subseteq [n]$, we write $\D(I)$ for the delta-matroid obtained by successively projecting along each $i \in I$, and similarly define $D/I$ and $D \setminus I$. 

\medskip 
In the introduction, we defined  the $U$-polynomial $U_{\D}(u,v)$ and its specialization, the interlace polynomial $\operatorname{Int}_\D(v) = U_\D(0,v)$, via a recursion involving deletion, contraction, and projection, similar to the deletion-contraction recursion for the Tutte polynomial of a matroid.
Like the Tutte polynomial of a matroid, the $U$-polynomial and the interlace polynomial also admit a non-recursive formula in the following way.
For a delta-matroid $\D$ with feasible sets $\mathcal{F}$ and $S \in \ads_n$, let 
\[
d_{\D}(S) = \frac{1}{2}\min_{B \in \mathcal{F}} | B \mathbin\triangle S|, \text{ the lattice distance between $\be_{S\cap [n]}$ and $P(\D)$}.
\]

\begin{prop}\label{prop:explicitu}
For a delta-matroid $\D$ on $[n,\overline n]$, define polynomials $\operatorname{Int}'_{\D}(v)$ and $U'_{\D}(u,v)$ by
\begin{align*}
\operatorname{Int}'_\D(v) &= \sum_{S\in \ads_n} v^{d_\D(S)}, \text{ and }
U_{\D}'(u, v) =\sum_{I \subseteq [n]} u^{|I|} \operatorname{Int}'_{\D(I)}(v).
\end{align*}
Then $U'_\D(u,v)$ satisfies the recursion for $U_\D(u,v)$ in \Cref{def:recursiveu}.
In particular, $U'_\D = U_\D$ and $\operatorname{Int}'_\D = \operatorname{Int}_\D$, and the recursive definition of $U_\D$ is independent of the element $i\in[n]$ chosen.
\end{prop}

\begin{proof}
We first show that $\operatorname{Int}'_\D(v)$ satisfies the recursive property in \Cref{def:recursiveu} with $u = 0$.
Then \cite[Theorem 30]{BrijderInterlace} states that if $i\in [n]$ is neither a loop nor coloop, then
$\operatorname{Int}'_{\D}(v) = \operatorname{Int}'_{\D/i}(v) + \operatorname{Int}'_{\D \setminus i}(v)$, and that if every element is a loop or a coloop, then $\operatorname{Int}'_\D(v) = (1+v)^n$.
If $i$ is a loop or a coloop of $\D$, then it continues to be so in $\D/J$ and $\D\setminus J$ for $J\subseteq [n]$ not containing $i$. 
Thus, we conclude that $\operatorname{Int}'_\D$ satisfies the desired recursive relation, and hence that $\operatorname{Int}'_\D = \operatorname{Int}_\D$.

For the $U$-polynomial, we have that
$$U'_{\D}(u, v) =  u U'_{\D(i)}(u, v) + \sum_{J \not \ni i} u^{|J|} \operatorname{Int}_{\D(J)}(v) .$$
If $i$ neither a loop nor coloop of $\D$, then $i$ is neither a loop nor coloop of $\D(J)$ for any $J$ not containing $i$.  The defining recursion for the interlace polynomial gives that
$$\sum_{J \not \ni i} u^{|J|} \operatorname{Int}_{\D(J)}(v) = \sum_{J \not \ni i} u^{|J|} \left( \operatorname{Int}_{\D(J)/i}(v) + \operatorname{Int}_{\D(J) \setminus i}(v) \right) = U_{\D/ i}'(u, v) + U_{\D \setminus i}'(u,v).$$
Combining these yields $U'_{\D}(u, v) = U'_{\D/i}(u, v) + U'_{\D \setminus i}(u, v) + uU'_{\D(i)}(u,v)$ if $i$ is not a loop or coloop of $\D$.  If $i$ is a loop or a coloop of $\D$, then it continues to be so in $\D(J)$ for $J\subseteq [n]$ not containing $i$.  Hence, if $i$ is a loop or a coloop, we have
\[
U'_\D(u,v) = \sum_{J\not\ni i}u^{|J|+1}\operatorname{Int}_{\D(J\cup i)}(v) +u^{|J|}  \operatorname{Int}_{\D(J)}(v) = \sum_{J\not\ni i} u^{|J|}\Big(u\operatorname{Int}_{\D(J\cup i)}(v) +  (v+1)\operatorname{Int}_{\D(J\cup i)}(v)\Big),
\]
and hence $U'_\D(u,v) = (u+v+1) U'_{\D\setminus i}(u,v)$.
\end{proof}

Given two delta-matroids $\D_1$, $\D_2$ on disjoint ground sets, let $\D_1 \times \D_2$ be the delta-matroid on the union of the ground sets whose feasible sets are $B_1 \sqcup B_2$ for $B_i$ feasible in $\D_i$. Observe that $d_{\D_1}(S_1)d_{\D_2}(S_2) = d_{\D_1 \times \D_2}(S_1 \sqcup S_2)$ and that projections commute with products, so \Cref{prop:explicitu} implies the following.

\begin{cor}\label{cor:multu}
For two delta-matroids $\D_1$ and $\D_2$ on disjoint ground sets, we have
\[
U_{\D_1 \times \D_2}(u, v) = U_{\D_1}(u, v) U_{\D_2}(u, v).
\]
\end{cor}

We also note the following property of $U_\D$ for future use.

\begin{lem}\label{lem:sumprojection}
We have that
$$\sum_{I \subseteq [n]} a^{|I|} U_{\D(I)}(u, v) = U_{\D}(u + a, v).$$
\end{lem}

\begin{proof}
We claim that, if $i$ is not a loop or coloop, then 
$$\sum_{I \subseteq [n]} a^{|I|} U_{\D(I)}(u, v) = U_{\D/i}(u + a, v) + U_{D \setminus i}(u + a, v) + (u+a) U_{D(i)}(u + a, v).$$
We induct on the size of the ground set. Note that
$$\sum_{i \in I \subseteq [n]} a^{|I|} U_{\D(I)}(u, v) = a \cdot \sum_{J \in [n] \setminus i} a^{|J|} U_{\D(i \cup J)} = a \cdot U_{\D(i)}(u + a, v), \text{ and}$$
\begin{equation*}\begin{split}
\sum_{i \not \in I \subseteq [n]} a^{|I|} U_{\D(I)}(u, v) &= \sum_{i \not \in J \subset [n]} a^{|J|} (U_{\D/i(J)}(u, v) + U_{\D \setminus i (J)}(u, v) + u U_{\D(i \cup J)}(u,v)) \\ 
&= U_{\D/i}(u + a, v) + U_{\D \setminus i}(u + a, v) + u U_{\D(i)}(u + a, v).
\end{split}\end{equation*}
Summing these gives the claim. When $i$ is a loop or coloop, it follows from the multiplicativity of the $U$-polynomial (\Cref{cor:multu}) that the left-hand side satisfies the expected product formula. This shows that the left-hand side satisfies the defining recursion of the right-hand side.
\end{proof}

We now compute the $U$-polynomials of delta-matroids arising from matroids.

\begin{eg}\label{ex:indeppoly}
We compute $U_{\D}$ for $\D = IP(\M)$, where $\M$ is a matroid on $[n]$ of rank $r$. An element $i \in [n]$ is a loop of $\D$ if $i$ is a loop of $\M$, and $i$ is never a coloop of $\D$. Then $\D(i)$ and $\D/i$ are both $IP(\M/i)$, and $\D \setminus i$ is $IP(\M \setminus i)$. Hence, $U_{IP(\M)}$ is a Tutte--Grothendieck invariant, which implies that
$$U_{IP(\M)}(u, v) = (u + 1)^{n - r} T_\M\left(u + 2, \frac{u + v + 1}{u + 1}\right).$$
\end{eg}

\begin{eg}\label{ex:basepolytope}
We compute $U_{P(\M)}$ for a matroid $\M$ on $[n]$. Let $\operatorname{corank}_{\M}(S) = \operatorname{rk}_{\M}([n]) - \operatorname{rank}_{\M}(S)$ be the corank and $\operatorname{nullity}_{\M}(S) = |S| - \operatorname{rk}_{\M}(S)$ the nullity of a subset $S$ in $\M$. Then we claim that
\[U_{P(\M)}(u,v) = \sum_{T\subseteq S\subseteq [n]} u^{|S-T|} v^{\operatorname{corank}_{\M}(S) + \operatorname{nullity}_{\M}(T)} .\]
Let $I \subseteq [n]$, and 
fix some $S \subseteq [n] \setminus I$.  Then $d_{P(\M)(I)}(S) = \min_{S \subseteq S' \subseteq S \cup I} d_{P(\M)}(S')$, and 
\begin{equation*}\begin{split}
d_{P(\M)}(S') &= \operatorname{corank}_{\M}(S') + \operatorname{nullity}_{\M}(S') \\ 
&= ( \operatorname{corank}_{\M|_{S \cup I}/S}(S') + \operatorname{corank}_{\M}(S \cup I) ) + ( \operatorname{nullity}_{\M|_{S \cup I}/S}(S') + \operatorname{nullity}_{\M}(S) ).
\end{split}\end{equation*}
The summand $\operatorname{corank}_{\M|_{S \cup I}}(S') + \operatorname{nullity}_{\M|_{S \cup I}}(S')$ achieves its minimum value $0$ when $S'$ is a basis of the minor $M|_{S \cup I}/S$.  The other summand is the constant $\operatorname{corank}_{\M}(S \cup I) + \operatorname{nullity}_{\M}(S)$. 
The claim then follows from Proposition~\ref{prop:explicitu}.
\end{eg}

It would be interesting to compute the $U$-polynomial of other families of delta-matroids such as those arising from graphs and ribbon graphs (see Examples~\ref{eg:adjacency}~and~\ref{eg:ribbon}).
Theorem~\ref{mainthm:logconc} applies to these delta-matroids, and therefore gives log-concavity results.

\medskip
We conclude this section by recording a multivariable version of the $U$-polynomial in the variables $u_1, \dotsc, u_n, v$.
Because this multivariable version will arise naturally in our intersection computations on $X_{B_n}$, it will be useful for proving log-concavity results.
For $I \subseteq [n]$, set $u^{I} = \prod_{i \in I} u_i$.
Following the formula in Proposition~\ref{prop:explicitu}, we define
$$U_{\D}(u_1, \dotsc, u_n, v) := \sum_{I \subseteq [n]} u^I \operatorname{Int}_{D(I)}(v).$$
Note that we recover the usual $U$-polynomial by setting $u = u_1 = \dotsb = u_n$.

\section{Representability and enveloping matroids}
We now discuss representability of delta-matroids and prepare for the construction of vector bundles associated to realizations of delta-matroids in Section~\ref{sec:tauto}.

\subsection{Torus-orbit closures}\label{subsec:torusorbit}
We will discuss representability of delta-matroids using polytopes and torus-orbit closures.
Let us prepare with generalities on torus-orbit closures in projective spaces and associated polytopes.

\medskip
Let $H$ be a torus with character lattice $\operatorname{Char}(H)$.
For a finite dimensional representation $V$ of $H$ and a point $x\in \PP(V)$, we define the \textbf{moment polytope} $P(\overline{H\cdot x})$ of its orbit closure $\overline{H\cdot x}$ as follows.
Let $V \simeq \bigoplus_{i = 0}^N V_i$ be the canonical decomposition into $H$-eigenspaces, where $H$ acts on each $V_i$ with character $a_i \in \operatorname{Char}(H)$. 
For a representative $v\in V$ of $x\in \PP(V)$, let $\mathscr A$ be the set 
\[
\mathscr A = \Big\{a_i :
v_i \neq 0 \text{ in the expression $v= \sum_{i=0}^N v_i$, where $v_i \in V_i$ for all $i=0, \ldots, N$}\Big\}
\]
which is independent of the choice of~$v$.
We define
\[
P(\overline{H\cdot x}) = \text{the convex hull of } \mathscr A \subset \operatorname{Char}(H) \otimes \RR.
\]
Over $\CC$, this agrees with the classical notion of moment polytopes; see for instance \cite[\S4.2]{Ful93} and \cite[\S8]{Sot03}.  Let us record the following basic facts.

\begin{prop}\label{prop:moment}
With notation as above:
\begin{enumerate}[label = (\arabic*)]
\item\label{orbits} The ($k$-dimensional) $H$-orbits of $\overline{H\cdot x}$ are in bijection with the ($k$-dimensional) faces of $P(\overline{H\cdot x})$ (for all $0\leq k \leq \dim H$).  The character lattice of the quotient of $H$ by the stabilizer of the orbit corresponding a face $F$ is the sublattice $\ZZ\{F\cap \mathscr A\}$ of $\operatorname{Char}(H)$.  (Here $F\cap \mathscr A$ is translated appropriately to contain the origin.)
\item\label{sub} If $\iota\colon  H' \hookrightarrow H$ is an inclusion of a subtorus $H'$ with the corresponding linear projection $\iota^\#\colon  \operatorname{Char}(H)_\RR \to \operatorname{Char}(H')_\RR$, then $P(\overline{H'\cdot x})$ equals the projection $\iota^\# P(\overline{H\cdot x})$.
\end{enumerate}
\end{prop}

\begin{proof}
The orbit closure $\overline{H\cdot x}$ is isomorphic to the $H$-variety
\[
X_{\mathscr A} = \text{the closure of the image of $H\to \PP^{|\mathscr A|-1}$ defined by $h \mapsto (h^a)_{a\in \mathscr A}$}.
\]
The first statement is then \cite[Corollary 3.A.6]{CLS}.  The second statement follows by construction because the  $H$-eigenspace $V_i$ with weight $a_i\in \operatorname{Char}(H)$ is an $H'$-eigenspace with weight $\iota^\# a_i \in \operatorname{Char}(H')$.
\end{proof}

\subsection{Representable delta-matroids}\label{ssec:representable}
For a delta-matroid $\D$ with feasible sets $\mathcal{F}$, let
\[
\widehat{P(\D)} = 2P(\D) - \be_{[n]} = \text{the convex hull of $\{ \be_B : B \in \mathcal{F}\}$} \subset [-1,1]^n.
\]
When $P(\D) = P(\M)$ or $P(\D) = IP(\M)$, we set $\widehat{P(\M)} := \widehat{P(\D)}$ and $\widehat{IP(\M)} := \widehat{P(\D)}$ respectively.
We now describe representability of $\D$ in terms of the polytope $\widehat{P(\D)}$ and torus-orbit closures in a type $B$ Grassmannian.

\medskip
The \newword{standard $(2n+1)$-dimensional quadratic space} is $\kk^{2n+1}$, whose coordinates are labelled $\{1, \ldots, n, \overline 1, \ldots, \overline n, 0\}$, and which is equipped with the quadratic form
\[
q(x_1, \dotsc, x_n, x_{\bar 1}, \ldots, x_{\bar n},x_{0}) = x_1 x_{\bar{1}} + \dotsb + x_{n} x_{\bar{n}} + x_{0}^2.
\]
A maximal isotropic subspace $L\subset \kk^{2n+1}$ is an $n$-dimensional subspace for which the restriction $q|_{L}$ is identically zero.
The \newword{maximal orthogonal Grassmannian}, denoted $OGr(n; 2n+1)$, is a variety whose $\kk$-valued points are in bijection with maximal isotropic subspaces of the standard $(2n+1)$-dimensional quadratic space $\kk^{2n+1}$.
By definition, $OGr(n; 2n+1)$ is a closed subvariety of the Grassmannian $Gr(n; 2n+1)$ with the Pl\"ucker embedding $Gr(n;2n+1) \hookrightarrow \PP^{\binom{2n+1}{n} -1}$.
The torus $\mathbb{G}_{\rm m}^{2n+1}$ acts on $Gr(n;2n+1)$ by its standard action on $\kk^{2n+1}$.
The torus $T = \mathbb {G}_m^n$ embeds into $\mathbb{G}_{\rm m}^{2n+1}$ by $(t_1, \dotsc, t_n) \mapsto (t_1, \dotsc, t_n, t_1^{-1}, \dotsc, t_n^{-1}, 1)$, and the induced action of $T$ on $Gr(n;2n+1)$ preserves $OGr(n; 2n+1)$.  We thus treat $OGr(n;2n+1)$ as a $T$-variety with the $T$-equivariant Pl\"ucker embedding in $\PP^{\binom{2n+1}{n}-1}$.

\begin{prop}\label{prop:GS}
For $L \subset \kk^{2n+1}$ maximal isotropic, the set of admissible subsets
\[
\mathcal F = \{S \in \ads_n : \text{the composition $L \hookrightarrow \kk^{2n+1} \twoheadrightarrow \kk^S$ is an isomorphism}\}
\]
is the set of feasible sets of a delta-matroid $\D$, and the moment polytope $P(\overline{T\cdot [L]})$ of the orbit closure of $[L]$ as a point in $\PP^{\binom{2n+1}{n}-1}$ is equal to $\widehat{P(\D)}$.
\end{prop}

In this case, we say that $L$ is a \newword{$B_n$ representation} of $\D$. We say that $\D$ is \newword{$B_n$ representable} if it has a $B_n$ representation.
Over $\CC$, the proposition is \cite[Section 7, Theorem 1]{Gelfand1987a}.  A type $C$ analogue of this statement for the Lagrangian Grassmannian, without the assertion about moment polytopes, appears in \cite[Theorem 3.4.3]{BGW}.

\begin{proof}
Index the coordinates of $\PP^{\binom{2n+1}{n}-1}$ by size $n$ subsets of $[n,\overline n] \cup \{0\}$.  One verifies that:
\begin{itemize}
\item The $T$-fixed points of $OGr(n;2n+1)$ correspond to admissible subsets $B\in \ads_n$ of size $n$, where $B$ gives a point in $\PP^{\binom{2n+1}{n}-1}$ whose Pl\"ucker coordinates are all zero except at $B$.
\item The $T$-invariant closed curves of $OGr(n;2n+1)$ correspond to pairs of $T$-fixed points such that, writing $B$ and $B'$ for the corresponding admissible subsets, $\be_B - \be_{B'}$ is parallel to $\be_i$, $\be_i+\be_j$, or $\be_i - \be_j$ for some $i,j\in [n]$.
\end{itemize}
The proposition now follows from \Cref{prop:moment}\ref{orbits}.
\end{proof}

\begin{eg}\label{rem:geometricSchubert}
Schubert delta-\linebreak[0]matroids are $B_n$ representable, and their representations explain their name as follows.
The closed cells $X_v$ of the Schubert stratification of $OGr(n;2n+1)$ are indexed by $v\in \W/\mathfrak S_n$, and the containment relation among the $X_v$ is given by the reversed Bruhat order.
If $x$ is a general point of~$X_v$, then the delta-matroid represented by the corresponding isotropic subspace is the standard Schubert delta-matroid $\Omega_{v\cdot[\overline n]}$.  In particular, they are certain generalized Bruhat interval polytopes corresponding to Schubert cells \cite{TW15}.
This is analogous to the relationship between Schubert matroids on $[n]$ of rank~$r$ and the Schubert stratification of $Gr(r;n)$.
\end{eg}

A maximal isotropic subspace $L$ of $\kk^{2n}$ with the quadratic form $q(x_1, \ldots, x_n, x_{\overline 1}, \ldots, x_{\overline n}) = x_1x_{\overline 1} + \cdots + x_n x_{\overline n}$ yields a maximal isotropic subspace $L\oplus \{0\}$ in $\kk^{2n+1}$, and hence a $B_n$ representation of a delta-matroid $\D$.
In such case, we say that $L$ is a \newword{$D_n$ representation} of $\D$.  Such a delta-matroid is an \newword{even delta-matroid}, meaning that the parity of $|B\cap [n]|$ for any feasible set $B$ is the same \cite[Theorem 3.10.2]{BGW}.

\medskip
In the literature, there are two prominent constructions of delta-matroids from graphs. Both constructions yield even delta-matroids with $D_n$ representations.

\begin{eg}\label{eg:adjacency}
Let $G$ be a simple graph on vertex set $[n]$, and let $A_G$ be its adjacency matrix with entries considered as elements of $\mathbb F_2$.  As the matrix $A_G$ is skew-symmetric, the row-span of the $n\times 2n$ matrix $[I_n | A_G]$ is an isotropic subspace of $\mathbb F_2^{2n}$, and hence defines an even delta-matroid $\D(G)$.
The interlace polynomial was originally defined and studied as a graph invariant.  See \cite{Duchamp, InterlaceBollobas, AvdH04}.
\end{eg}

\begin{eg}\label{eg:ribbon}
A graph $\Gamma$ embedded in a surface, also known as a ribbon graph, with edges labeled by~$[n]$
defines a delta-matroid $\D(\Gamma)$ whose feasible sets are the ``spanning quasi-trees'' of $\Gamma$, i.e., the spanning subgraphs whose small neighborhood has just one boundary component.
Note that for a planar graph, this coincides with the usual graphical matroid of the graph.
See \cite{CMNR19a} for a history and proofs, and \cite{CMNR19b} for further connection between delta-matroids and ribbon graphs generalizing the connection between matroids and graphs.  \cite[Theorem 4.3.5]{BGW} shows that such a delta-matroid has a $D_n$ representation (see also \cite{BBGS00}).
\end{eg}

\subsection{Enveloping matroids}
The notion of an \textbf{enveloping matroid} of a delta-matroid will play a crucial role when we construct ``tautological classes of delta-matroids'' in \S\ref{sec:tauto} and when we apply tools from tropical Hodge theory to prove Theorem~\ref{mainthm:logconc} in \S\ref{sec:logconc}. 

\medskip
Let $\operatorname{env} \colon \RR^{2n} \to \RR^n$ be the map given by $\operatorname{env}(x_1, \dotsc, x_n, x_{\bar 1}, \ldots, x_{\bar{n}}) = (x_1 - x_{\bar{1}}, \dotsc, x_n - x_{\bar{n}})$.
To avoid confusion with our notation that $\be_{\overline i} = -\be_i \in \RR^n$, we use $\mathbf{u}_1, \dotsc, \mathbf{u}_n, \mathbf{u}_{\bar{1}}, \dotsc, \mathbf{u}_{\bar{n}}$ to refer to the standard basis of $\mathbb{R}^{2n}$. For $S \subset [n, \bar{n}]$, let $\mathbf{u}_S = \sum_{i \in S} \mathbf{u}_i$. If $S \in \ads$, then $\operatorname{env}(\mathbf{u}_S) = \be_S$.

\begin{defn}\label{def:enveloping}
Let $\M$ be a matroid on $[n, \bar{n}]$, and let $\D$ be a delta-matroid on $[n, \bar{n}]$. Then $\M$ is an \newword{enveloping matroid}
of $\D$ if the image of $P(\M)$ under $\operatorname{env}$ is $\widehat{P(\D)}$.
\end{defn}

\begin{rem}\label{rem:shelter}
In \cite[Section 4]{BouShelter}, Bouchet considers matroids $\M$ on $[n, \bar{n}]$ whose independent sets which are admissible are the subsets of the feasible sets of a delta-matroid $\D$. He calls such a matroid a \textbf{sheltering matroid} of $\D$. It follows from \cite[Section 3.3]{LarRank} that $\M$ is a sheltering matroid if and only if $\operatorname{env}(IP(\M)) = P(\D) + \square - \be_{[n]}$, so Lemma~\ref{lem:envelopingindep} will show that enveloping matroids are sheltering matroids. 

In \cite[Exercise 3.12.6]{BGW}, the authors consider matroids whose bases which are admissible are the feasible sets of $\D$.
They call such a matroid also an enveloping matroid, which disagrees with Definition~\ref{def:enveloping}.

Let $\D$ be the delta-matroid on $[2, \bar{2}]$ with feasible sets $\{1, 2\}$ and $\{1, \bar{2}\}$. The matroid on $[2, \bar{2}]$ with bases $\{1, 2\}, \{1, \bar{2}\}$, and $\{2, \bar{2}\}$ is a sheltering matroid for $\D$, but it is not an enveloping matroid. The matroid with bases $\{1, 2\}, \{1, \bar{2}\}$, and $\{1, \bar{1}\}$ is an enveloping matroid in the sense of \cite[Exercise 3.12.6]{BGW}, but it is not a sheltering matroid. 
\end{rem}

Our main examples of delta-matroids with enveloping matroids are $B_n$ representable delta-matroids (\Cref{prop:envelopingtypeB}), which in particular includes delta-matroids arising from graphs and graphs embedded on surfaces by Examples~\ref{eg:adjacency} and~\ref{eg:ribbon},
and delta-matroids arising from matroids (\Cref{prop:matroidenveloping}).

\medskip
Existence of enveloping matroids behaves well with respect to operations on delta-matroids as follows.  Let $\M$ be an enveloping matroid of a delta-matroid $\D$ on $[n,\overline n]$.
\begin{itemize}
\item For $w\in \W$, the $\W$-action on $[n, \bar{n}]$ makes $w \cdot \M$ an enveloping matroid of $w \cdot \D$.
\item For $i\in [n]$, the matroid minor $\M/i\setminus \bar{i}$ (resp.\ $\M\setminus i / \bar{i}$) is an enveloping matroid for $\D/i$ (resp.\ $\D\setminus i$).
\item If $\M'$ is an enveloping matroid of another delta-matroid $\D'$ on ground set disjoint from that of $\D$, then $\M \oplus \M'$ is an enveloping matroid for $\D\times \D'$.
\item The \textbf{dual delta-matroid} $\D^{\perp}$ is the delta-matroid with feasible sets $\{\overline{B} \colon B \text{ a feasible set of }\D\}$. Then the dual matroid $\M^{\perp}$ is an enveloping matroid for $\D^{\perp}$.
\end{itemize}

For future use in \S\ref{sec:logconc}, we record an observation that loops and coloops of $\D$ and $\M$ are compatible.

\begin{lem}\label{lem:envelopingloop-free}
Let $\D$ be a delta-matroid with an enveloping matroid $\M$, and let $i\in [n]$.
Then $i$ is a loop (resp.\ coloop) in $\D$ if and only if $i$ is a loop and $\overline i$ a coloop (resp.\ $i$ is a coloop and $\overline i$ a loop) in $\M$.
In particular, if $\D$ is loop-free and coloop-free, then so is $\M$. 
\end{lem}

\begin{proof}
Let us prove the statement for when $i$ is a loop, i.e., the polytope $\widehat{P(D)}\subset \RR^n$ is contained in the hyperplane $x_i = -1$.
If a basis $B$ of $\M$ contains $i$ or does not contain $\overline i$, then $\operatorname{env}(\mathbf u_B)$ lies in $x_i \geq 0$.  Hence $i$ is a loop and $\overline i$ a coloop of $\M$. The other direction is similar.
\end{proof}

\begin{prop}\label{prop:envelopingtypeB}
Let $L \subset \kk^{2n+1}$ be a $B_n$ representation of a delta-matroid $\D$, and let $L'$ denote the image of $L$ under the projection to $\kk^{2n}$ forgetting the $x_{0}$-coordinate. Then the matroid that $L'$ represents is an enveloping matroid of $\D$.  In particular, every $B_n$ representable delta-matroid has an enveloping matroid.
\end{prop}

\begin{proof}
Let $\M$ be the matroid that $L$ represents.  As a point in $OGr(n;2n+1) \subset Gr(n; 2n + 1) \subset \PP^{\binom{2n+1}{n}-1}$, the moment polytope of $\overline{\mathbb{G}_{\rm m}^{2n+1} \cdot [L]}$ is $P(\M)$, whereas the moment polytope of $\overline{T \cdot [L]}$ is $\widehat{P(\D)}$ by \Cref{prop:GS}.  Then \Cref{prop:moment}\ref{sub} implies that the image of $P(\M)$ under the composition $\operatorname{env} \circ \pi_0$ is $\widehat{P(\D)}$, where $\pi_0\colon  \RR^{2n+1} \to \RR^{2n}$ is the projection forgetting the $0$th coordinate.  Note that $L'$ is a representation of $\M \setminus 0$, and $\operatorname{env}(P(\M\setminus 0))$ is contained in $\operatorname{env} \circ \pi_0(P(\M)) = \widehat{P(\D)}$. Each feasible set of $\D$ is a basis of $\M$ which does not contain $0$, and hence is a basis of $\M \setminus 0$, which proves that $\operatorname{env}(P(\M \setminus 0)) = \widehat{P(\D)}$. 
\end{proof}

\begin{rem}\label{rem:shelterable}
Because the Weyl groups of type $B$ and $C$ root systems coincide, one may consider delta-matroids as type $C$ Coxeter matroids, and consequently consider $C_n$ representability in terms of Lagrangian subspaces in a $2n$-dimensional space with a symplectic form.  See \cite{BGW98} or \cite[\S3.4]{BGW}.
The proof of Proposition~\ref{prop:envelopingtypeB} shows that $C_n$ representable delta-matroids also have enveloping matroids.
\end{rem}

\begin{prop}\label{prop:matroidenveloping}
Let $\M$ be a matroid on $[n]$. Then the delta-matroids $P(\M)$ and $IP(\M)$ have enveloping matroids. 
\end{prop}

\begin{proof}
For $P(\M)$, we show that $\M \oplus \overline \M^\perp$ is an enveloping matroid, 
where $\overline \M^\perp$ is the isomorphic image of $\M^\perp$ under $\overline{(\cdot)}:[n]\to[\bar n]$. 
Minkowski sums commute with linear projections, so
\begin{align*}
    \operatorname{env}(P(\M \oplus \overline \M^\perp)) 
  &=\operatorname{env}(P(\M)+P(\overline \M^\perp))
\\&=P(\M)+ (-P(\M^\perp))
\\&=P(\M) +(P(\M)-\be_{[n]})=\widehat{P(\M)}. 
\end{align*}

For $IP(\M)$ we take the free product $\M \mathbin\square \overline \M^\perp$ of \cite{CrapoSchmitt},
whose bases are the sets $S\cup\overline T$ of size $\rank \M+\rank \M^\perp=n$ with $S,T\subseteq[n]$
such that $S$ is independent in~$\M$ and $T$ is spanning in~$\M^\perp$. 
Write $SP(\mathrm{N})$ for the spanning set polytope of a matroid $\mathrm{N}$, so $SP(\mathrm{N}^\perp) = -IP(\mathrm{N})+\be_{[n]}$. We show that
\[
P(\M \mathbin\square \overline \M^\perp) = (IP(\M)+SP(\overline \M^\perp))\cap H,
\]
where $H$ is the hyperplane $\{v\in\RR^{2n}:\sum_{i\in[n,\bar n]}v_i=n\}$.
For a polytope $Q$, any vertex of $Q\cap H$ is of the form $F\cap H$, where $F$ is a vertex or edge of~$Q$.
The polytope $IP(\M)+SP(\overline \M^\perp)$ is a lattice polytope whose edge directions all have the form $\mathbf{u}_i$ or $\mathbf{u}_i-\mathbf{u}_j$ for $i,j\in[n,\bar n]$ because each edge of a Minkowski sum is parallel to an edge of one of the two summands.
As $\sum_{i\in[n,\bar n]}v_i$ takes values 0 or 1 on all of these direction vectors,
if $H$ intersects an edge of $IP(\M)+SP(\overline \M^\perp)$ transversely, then the intersection is a lattice point.
Therefore $(IP(\M)+SP(\overline \M^\perp))\cap H$ is a lattice polytope as well.
By definition of the free product, $P(\M \mathbin\square \overline \M^\perp)$ and this intersection have the same set of lattice points, so they are equal.
Now as above
\begin{align*}
    \operatorname{env}(P(\M \mathbin\square \overline \M^\perp)) 
&\subseteq 
    \operatorname{env}(IP(\M)+SP(\overline \M^\perp))
\\&=IP(\M)+(-SP(\M^\perp))
\\&=IP(\M)+(IP(\M)-\be_{[n]})=\widehat{IP(\M)}.
\end{align*}
The containment is an equality because every vertex of $\widehat{IP(\M)}$ has the form $\be_S-\be_{\overline{E\setminus S}}$ for $S$ an independent set of~$\M$, and this vertex has the preimage $(\mathbf{u}_S,\mathbf{u}_{E\setminus S})$ in $P(\M \mathbin\square \overline \M^\perp)$.
\end{proof}

\begin{eg}\label{eg:nonenvelopable}
In \cite[Section 4]{BouShelter}, Bouchet gives the example, which he attributes to Duchamp, of the delta-matroid with the set of feasible sets
\begin{equation*}\begin{split}
\mathcal{F} = \{\{\bar{1}, \bar{2}, \bar{3}, \bar{4}\}, &\{\bar{1}, \bar{2}, \bar{3},4\}, \{\bar{1}, 2, 3, \bar{4}\}, \{1, \bar{2}, 3, \bar{4}\},\{1, 2, \bar{3}, \bar{4}\},\\ 
&  \{\bar{1}, 2, 3, 4\}, \{1, \bar{2}, 3, 4\}, \{1, 2, \bar{3}, 4\}, \{1, 2, 3, 4\} \}.
\end{split}\end{equation*}
There is no matroid on $[4, \bar{4}]$ whose set of bases which are admissible is $\mathcal{F}$. In particular, this delta-matroid does not have an enveloping matroid. 
\end{eg}

\section{Vector bundles and $K$-classes}\label{sec:tauto}
We now define two types of equivariant vector bundles associated to realizations of delta-matroids, which we call \textbf{isotropic tautological bundles} and \textbf{enveloping tautological bundles} respectively. The isotropic tautological bundles are analogous to the bundles used in \cite{BEST}, and the enveloping tautological bundles are analogous to the bundles used in \cite{EHL}. The construction of an isotropic tautological bundle depends on the choice of a $B_n$ representation of a delta-matroid, and the construction of an enveloping tautological bundle depends on the choice of a realization of an enveloping matroid.
The $K$-classes of the bundles will only depend on the delta-matroid, which leads to the construction of \textbf{isotropic tautological classes} and \textbf{enveloping tautological classes} for all delta-matroids, not necessarily with a $B_n$ representation or a representable enveloping matroid.

\medskip
In both cases, we will construct a $T$-equivariant map from $X_{B_n}$ to a Grassmannian and define the bundles as pullbacks of certain universal bundles. Let us therefore prepare with a discussion of maps from $X_{B_n}$ to Grassmannians.
The discussion can be easily adapted to replace $X_{B_n}$ with any smooth projective toric variety, but such generality won't be needed here.

\subsection{Maps into Grassmannians}
Let $L \subset \kk^N$ be a linear space of dimension $r$, corresponding to a point $[L]$ of $Gr(r; N)$ and representing a matroid $\M$ of rank $r$ on $[N]$. Let $\iota \colon T \to \mathbb{G}_{\rm m}^N$ be an inclusion of $T$ into the torus acting on $Gr(r; N)$, and let $\iota^{\#} \colon \operatorname{Char}(\mathbb{G}_{\rm m}^N) \to \operatorname{Char}(T)$ be the pullback map on character lattices. Then $\iota^{\#}P(\M)$ is a lattice polytope in $\operatorname{Char}(T) \otimes \RR$. Suppose that $\Sigma_{B_n}$ refines the normal fan of $\iota^{\#}(P(\M))$.
For each $w \in \W$ and any $v$ in the interior of $C_w$, let $B_w$ be any basis of $\M$ such that the corresponding vertex of $P(\M)$ maps under $\iota^\sharp$ into the $v$-minimal vertex $\face{v}{\iota^{\#}P(\M)}$. 

\begin{prop}\label{prop:toricgrass}
With the set-up as above,  there is a unique $T$-equivariant morphism $\varphi_L \colon X_{B_n} \to Gr(r; N)$ such that the identity of $T \subset X_{B_n}$ is sent to $[L]$.   The pullback $\varphi_L^*(\mathcal{S}_{\mathrm{univ}})$ of the tautological subbundle on $Gr(r; N)$ is a $T$-equivariant vector bundle on $X_{B_n}$ such that, for each $w\in \W$, the $T$-equivariant $K$-class localizes to
$$[\varphi_{L}^*(\mathcal{S}_{\mathrm{univ}})]_w = \sum_{i \in B_w} \iota^{\#}T_i.$$
\end{prop}

\begin{proof}
The moment polytope (taken with respect to the Pl\"{u}cker embedding of the Grassmannian) of the $\mathbb{G}_{\rm m}^N$-orbit closure $\overline{\mathbb{G}_{\rm m}^N \cdot [L]} \subset Gr(r; N)$ is $P(\M)$, so, by Proposition~\ref{prop:moment}\ref{sub}, the moment polytope of the $T$-orbit closure $\overline{T \cdot [L]}$ is $\iota^{\#} P(\M)$. Note that $\overline{T \cdot [L]}$ is a (possibly non-normal) toric variety whose embedded torus is $T/\operatorname{Stab}_{T}([L])$. The normalization of $\overline{T \cdot [L]}$ is a toric variety whose fan is the normal fan of $\iota^{\#}P(\M)$ (considered in $\operatorname{Cochar}(T) \otimes \mathbb{R}$, possibly with lineality space),  and whose lattice may be finer than the lattice in $\Sigma_{B_n}$. We therefore have a unique morphism $X_{B_n} \to \overline{T \cdot [L]} \hookrightarrow Gr(r; N)$ such that the identity of $T$ is sent to $[L]$. 

To compute the localization of $[\varphi_{L}^*(\mathcal{S}_{\mathrm{univ}})]$ to a fixed point of $X_{B_n}$corresponding to $w \in \W$, we consider the image of this fixed point, $x_w \in Gr(r; N)$. Because pullbacks commute with pullbacks, it suffices to compute the pullback of $[\mathcal{S}_{\mathrm{univ}}]$ to $x_w$ in $T$-equivariant $K$-theory.  
Note that $x_w$ is a $T$-fixed point, which implies that $\overline{\mathbb{G}_{\rm m}^N \cdot x_w}$ is acted on trivially by $T$, so $K_T(\overline{\mathbb{G}_{\rm m}^N \cdot x_w}) = K(\overline{\mathbb{G}_{\rm m}^N \cdot x_w}) \otimes \mathbb{Z}[T_1^{\pm 1}, \dotsc, T_n^{\pm 1}]$. 
Therefore the pullback in $T$-equivariant $K$-theory of $[\mathcal{S}_{\mathrm{univ}}]$ to any point of $\overline{\mathbb{G}_{\rm m}^N \cdot x_w}$ is the same element of $\mathbb{Z}[T_1^{\pm 1}, \dotsc, T_n^{\pm 1}]$. 
The $\mathbb{G}_{\rm m}^N$-fixed points of $\overline{\mathbb{G}_{\rm m}^N \cdot x_w}$ are exactly the vertices of $P(\M)$ in the preimage of $\face{v}{\iota^{\#}P(\M)}$. The pullback in $\mathbb{G}_{\rm m}^N$-equivariant $K$-theory of $[\mathcal{S}_{\mathrm{univ}}]$ to a $\mathbb{G}_{\rm m}^N$-fixed point of $Gr(r; N)$ corresponding to $B_w \subset [N]$ is $\sum_{i \in B_w} T_i$. Applying $\iota^{\#}$ implies the result. 
\end{proof}

For using \Cref{prop:toricgrass}, we set up some notation for a delta-matroid $\D$ and $w\in \W$:
\begin{itemize}
\item Let $B_w(\D)$ be the \newword{$w$-minimal} \newword{feasible set of }$\D$, i.e., the feasible set corresponding to the vertex $\face{v}P(\D)$ of $P(\D)$ on which any linear functional $v$ in $C_w^{\circ}$ achieves its minimum.
\item Likewise, let $B_w^{\max}(\D)$ be the \newword{$w$-maximal} feasible set corresponding to the vertex of $P(\D)$ on which any linear functional in the interior of $C_w$ achieves its maximum.
\end{itemize}
Note that $\overline{B_w^{\max}(\D)} = B_w(\D^\perp)$.
We omit $(\D)$ and simply write $B_w$ if no confusion is expected.

\subsection{Construction of isotropic tautological bundles}\label{ssec:isotropic taut}

Let $\mathcal{O}_{OGr(n; 2n+1)}^{\oplus 2n+1}$ be the rank $2n+1$ trivial bundle on $OGr(n;2n+1)$, which is equipped with the standard quadratic form, and which is a $T$-equivariant vector bundle with the action
\begin{equation}\label{eq:T-action on fiber}
(t_1, \dotsc, t_n) \cdot (x_1, \dotsc,  x_n, x_{\overline 1}, \ldots, x_{\overline n}, x_{0}) = (t_1 x_1, \dotsc, t_n x_n, t_1^{-1} x_{\bar{1}}, \dotsc, t_{n}^{-1}x_{\bar{n}}, x_{0}).
\end{equation}
Let $\mathcal{I}_{\rm{univ}}$ be the universal isotropic subbundle of $\mathcal{O}_{OGr(n; 2n+1)}^{\oplus 2n+1}$, whose fiber over a point of $OGr(n;\linebreak[0] 2n+1)$ corresponding to the maximal isotropic subspace $L \subset \kk^{2n+1}$ is $L$.  Under the inclusion $OGr(n; 2n+1) \subset Gr(n; 2n+1)$, the bundle $\mathcal{I}_{\rm{univ}}$ is the $T$-equivariant subbundle of $\mathcal{O}_{OGr(n; 2n+1)}^{\oplus 2n+1}$ obtained as the restriction of the universal subbundle on $Gr(n; 2n+1)$. Then the following proposition follows from Proposition~\ref{prop:toricgrass} and the fact that $OGr(n; 2n+1)$ is a $T$-fixed subvariety of $Gr(n; 2n+1)$.

\begin{prop}\label{prop:map}
For each $B_n$ representation $L \subset \kk^{2n+1}$ of a delta-matroid $\D$, we have a $T$-equivariant map
\[
X_{B_n} \to \overline{T\cdot [L]} \hookrightarrow OGr(n;2n+1)
\]
such that the identity of $T$ is sent to $[L]$.
For each $w\in \W$, the pullback of $\mathcal{I}_{\mathrm{univ}}$ localizes to $\sum_{i \in B_w} T_i$ at the $T$-fixed point of $X_{B_n}$ corresponding to $w$. 
\end{prop}

Note our continued use of the convention that $T_{\overline i} = T_i^{-1}$ for $i\in [n]$.

\begin{defn}
Let $L$ be a $B_n$ representation of a delta-matroid $\D$. Then the \newword{isotropic tautological bundle} $\mathcal{I}_L$ on $X_{B_n}$ is the pullback of $\mathcal{I}_{\rm{univ}}$ under the map $ X_{B_n} \to OGr(n; 2n+1)$ in \Cref{prop:map}.
\end{defn}

Let $\mathcal{O}_{X_{B_n}}^{\oplus 2n+1}$ be the rank $2n+1$ trivial bundle with a  $T$-equivariant structure given by the action of $T$ on~$\kk^{2n+1}$ in \eqref{eq:T-action on fiber}.
Note that $\mathcal{I}_L$ is the unique $T$-equivariant subbundle of $\mathcal{O}_{X_{B_n}}^{\oplus 2n+1}$ whose fiber at the identity of $T\subset X_{B_n}$ is the isotropic subspace $L$.
In particular, its dual $\mathcal I_L^\vee$ is globally generated, and $\mathcal{I}_L$ is an anti-nef vector bundle.
The equivariant $K$-class of $\mathcal I_L$ depends only on the delta-matroid $\D$.  Moreover, we show that this $K$-class is well-defined for any delta-matroid, not necessarily representable.

\begin{prop}\label{prop:welldefined}
For any delta-matroid $\D$ on $[n, \bar{n}]$, there is a class $[\mathcal{I}_{\D}] \in K_T(X_{B_n})$ defined by 
$$[\mathcal{I}_{\D}]_w = \sum_{i \in B_w} T_i.$$
\end{prop}

We define the \newword{isotropic tautological class} $[\mathcal{I}_{\D}]$ of $\D$ by the above formula.  Proposition~\ref{prop:map} implies that $[\mathcal{I}_{\D}] = [\mathcal{I}_L]$ if $L$ is a $B_n$ representation of $\D$.

\begin{proof}
We need to check that the above formula satisfies the compatibility condition in Theorem~\ref{thm:Klocalization}. Let $w \in \W$, and set $w' = w \tau_{i, i+1}$. Then the cones corresponding to $w$ and $w'$ share a hyperplane whose normal vector is $\be_{w(i)} - \be_{w({i+1})}$. As the normal fan of $\widehat{P(\D)}$ coarsens $\Sigma_{B_n}$, the $w$-minimal and $w'$-minimal vertices of  $\widehat{P(\D)}$ either coincide or differ by an edge parallel to $\be_{w(i)} - \be_{w({i+1})}$. This implies that $[\mathcal{I}_{\D}]_w - [\mathcal{I}_{\D}]_{w'}$ is either $0$ or $\pm (T_{w(i)}-T_{w(i+1)})$, which is divisible by $1 - T_{w(i)}T_{w(i+1)}^{-1}$. 

Now set $w' = w \tau_n$. Then the cones corresponding to $w$ and $w'$ share a hyperplane whose normal vector is $\be_{w(n)}$.  Again, that the normal fan of $\widehat{P(\D)}$ coarsens $\Sigma_{B_n}$ implies that either $[\mathcal{I}_{\D}]_w = [\mathcal{I}_{\D}]_{w'}$ or $[\mathcal{I}_{\D}]_w - [\mathcal{I}_{\D}]_{w'} =\pm (1-T_{w(n)})$ is divisible by $1 - T_{w(n)}$.
\end{proof}

\begin{rem}
We could also consider the quotient bundles $\mathcal{O}_{X_{B_n}}^{\oplus 2n+1}/\mathcal{I}_L$. However, one can verify that $[\mathcal I_L] + [\mathcal I_L]^\vee  = [\mathcal O^{\oplus 2n + 1}]$, and so $c([\mathcal{I}_L]^{\vee}) = c(\mathcal{O}_{X_{B_n}}^{\oplus 2n+1}/\mathcal{I}_L)$. Therefore, studying the quotient bundle does not give any new elements of $A^{\bullet}(X_{B_n})$.
\end{rem}

\subsection{Construction of enveloping tautological bundles}\label{ssec:envelopingtaut}

From each realization $L \subset \kk^{2n}$ of an enveloping matroid $\M$ of a delta-matroid $\D$, we construct the enveloping tautological bundles $\mathcal{S}^E_L$ and $\mathcal{Q}^E_L$.
Let $\pi_i \colon X_{B_n} \to \mathbb{P}^1$ denote the composition $X_{B_n} \to (\mathbb{P}^1)^n \to \mathbb{P}^1$, where the latter map is the projection onto the $i$th factor.
Let us treat $\PP^1$ as the toric variety of the fan in $\RR$ consisting of the positive ray, negative ray, and the origin.  $\PP^1$ has two torus-fixed divisors $\infty$ and $o$ that correspond respectively to the negative ray and the positive ray. These torus-fixed divisors correspond respectively to the intervals $[0,1]$ and $[-1,0]$ under the standard correspondence between polytopes and base-point-free divisors on toric varieties \cite[Chapter 6]{CLS}.
Let $\Oinfty$ and $\Oo$ be the respective toric line bundles isomorphic to $\mathcal O_{\PP^1}(1)$, and define
\[
\mathcal{M} = \bigoplus_{i \in [n]} \pi_i^* \Oinfty \oplus \pi_i^* \Oo.
\]
We now show the existence of vector bundles $\mathcal S_L^E$ and $\mathcal Q_L^E$ on $X_{B_n}$ that fit into a short exact sequence of $T$-equivariant vector bundles
$$0 \to \mathcal{S}^E_L \to \mathcal{M} \to \mathcal{Q}^E_L \to 0,$$
which is characterized by the property that the fiber over of the identity point of $T$ is $0 \to L \to \kk^{2n} \to \kk^{2n}/L \to 0$.  We prepare with a combinatorial lemma.  Recall that $\square$ denotes the cube $[0, 1]^n$, and 
the standard basis of $\mathbb{R}^{2n}$ is denoted $\mathbf{u}_1, \dotsc, \mathbf{u}_n, \mathbf{u}_{\bar{1}}, \dotsc, \mathbf{u}_{\bar{n}}$.

\begin{lem}\label{lem:envelopingindep}
Let $\M$ be an enveloping matroid of a delta-matroid $\D$. Then
\[
\operatorname{env}(IP(\M)) = P(\D) + \square - \be_{[n]}.
\]
\end{lem}

\begin{proof}
First we note that $\operatorname{env}(IP(\M))$ is contained in $P(\D) + \square - \be_{[n]}$. Every vertex of $\operatorname{env}(IP(\M))$ can be written as $\frac{1}{2}\operatorname{env}(\mathbf{u}_{B}) + \frac{1}{2}\operatorname{env}(-\mathbf{u}_S)$ for some basis $B$ of $\M$ and $S \subset B$. Then $\frac{1}{2}\operatorname{env}( \mathbf{u}_B) \in P(\D) - (\frac{1}{2}, \dotsc, \frac{1}{2})$ and $\frac{1}{2}\operatorname{env}(-\mathbf{u}_S) \in \square - (\frac{1}{2}, \dotsc, \frac{1}{2})$. 

Now it suffices to show that every vertex of $P(\D) + \square - \be_{[n]}$ is contained in $\operatorname{env}(IP(\M))$. Let $v$ be a vector in the interior of $C_w$. Then
\begin{align*}
\face{v}{(P(\D) + \square - \be_{[n]})} & = \textstyle \face{v}{(P(\D) - \frac{1}{2} \be_{[n]})} + \face{v}{(\square - \frac{1}{2}\be_{[n]})}\\
& \textstyle =  \frac{1}{2} \be_{B_w} + \frac{1}{2} \be_{w([\bar n])}\\
&=  \textstyle \frac{1}{2} \be_{B_w} - \frac{1}{2} \be_{w([n])}.
\end{align*}
Because the normal fan of $P(\D) + \square - \be_{[n]}$ is a coarsening of $\Sigma_{B_n}$, every vertex is of the form $\frac{1}{2} \be_{B_w} - \frac{1}{2} \be_{w([n])}$ for some $w \in \W$. We see that this is equal to $\operatorname{env}(\mathbf{u}_{B_w} - \mathbf{u}_{B_w \cap w([n])})$. Because $B_w \cap w([n]) \subset B_w$, this is contained in $IP(\M)$. 
\end{proof}

We first construct the dual of the vector bundle $\mathcal Q_L^E$.  Let $L^\perp$ be the dual space $(\kk^{2n}/L)^\vee$,  considered as a subspace of $\kk^{2n}$ under the isomorphism $(\kk^{2n})^\vee \simeq \kk^{2n}$.
It represents the dual matroid of the matroid represented by $L$.
Let the torus $T$ act on $\kk^{4n} = \kk^{2n}\times \kk^{2n}$ by the usual action $(t_1x_1, \ldots, t_nx_n, t_1^{-1}x_{\overline 1}, \ldots, t_n^{-1}x_{\overline n})$ on the first $\kk^{2n}$ factor and trivially on the second $\kk^{2n}$ factor.  We let $T$ act on $Gr(n;4n)$ accordingly.

\begin{prop}\label{prop:map2}
For a representation $L$ of an enveloping matroid $\M$ of a delta-matroid $\D$, let $E_L \subset \kk^{4n}$ be the image of $L^{\perp}$ under the diagonal embedding $\kk^{2n} \hookrightarrow \kk^{4n}$.  Then there is a composition of $T$-equivariant maps
\[
\varphi_L\colon  X_{B_n} \to \overline{T\cdot [E_L]} \hookrightarrow Gr(n;4n).
\]
\end{prop}

We define the \newword{enveloping tautological quotient bundle} $\mathcal Q_L^E$ to be the \emph{dual} of the pullback of the universal subbundle on $Gr(n;4n)$ via the map $\varphi_L$.

\begin{proof}
Let $\widetilde T$ be the $2n$-dimensional torus $\GG^{2n}_m$ with the action on $Gr(n;4n)$ induced by
$$(t_1, \dotsc, t_{2n}) \cdot (x_1, \dotsc, x_{4n}) = (t_1 x_1, \dotsc, t_{2n}x_{2n}, x_{2n+1}, \dotsc, x_{4n}).$$
By \cite[Proposition 3.16]{EHL}, the moment polytope of $\overline{\widetilde{T} \cdot [E_L]}$ is $IP(\M^{\perp})$.
By \Cref{prop:moment}\ref{sub}, the moment polytope of $\overline{T \cdot [E_L]}$ is $\operatorname{env}(IP(\M^{\perp})) = P(\D^{\perp})+ \square - \be_{[n]}$.  Note that the normal fan of $P(\D^{\perp})+ \square - \be_{[n]}$ coarsens $\Sigma_{B_n}$, so we conclude by Proposition~\ref{prop:toricgrass}.
\end{proof}

By construction, we have a surjection $\mathcal{O}_{X_{B_n}}^{\oplus 4n} \to \mathcal{Q}_L^E$. There is also a surjection $\mathcal{O}_{X_{B_n}}^{\oplus 4n} \to \mathcal{M}$, given by taking the direct sum over all $i = 1, \ldots, n$ of the surjections 
$$\mathcal{O}^{\oplus 4}_{X_{B_n}} \simeq H^0(\mathbb{P}^1, \Oinfty \oplus \Oo) \otimes \mathcal{O}_{X_{B_n}} \to \pi_i^* \Oinfty \oplus \pi_i^*\Oo,$$
whose kernel is $\pi_i^* (-1_{\infty}) \oplus \pi_i^* (-1_{o})$.

\begin{prop}
The composition 
$$\bigoplus_{i \in [n]} \pi_i^* (-1_{\infty}) \oplus \pi_i^* (-1_{o}) \to \mathcal{O}_{X_{B_n}}^{\oplus 4n} \to \mathcal{Q}_L^E$$
is zero, 
so there is a map $\mathcal{M} \to \mathcal{Q}_L^E$.
\end{prop}

We define the \newword{enveloping subbundle} $\mathcal{S}_L^E$ to be the kernel of the map $\mathcal{M} \to \mathcal{Q}_L^E$.

\begin{proof}
It suffices to check this on the dense open torus $T \subset X_{B_n}$.  By considering each factor of $T = \mathbb{G}_{\rm m}^n$ separately, the computation reduces to the case $n=1$. Over a point $t \in \mathbb{G}_{\rm m}$, the fiber of $\pi_i^* (-1_{\infty}) \oplus \pi_i^* (-1_{o}) \subseteq \mathcal{O}_{\mathbb{P}^1}^{\oplus 4}$ is the subspace $\{(ta, t^{-1}b, a, b) \colon (a, b) \in k^2\} \subseteq k^4$. The form of $E_L$ then implies the claim. 
\end{proof}

We now compute the $T$-equivariant $K$-classes of $\mathcal{S}_L^E$ and $\mathcal{Q}_L^E$. 

\begin{prop}\label{prop:envelopingK}
The equivariant $K$-classes of $\mathcal{S}_L^E$ and $\mathcal{Q}_L^E$ are given by
\begin{align*}
[\mathcal{S}_L^E]_{w} &= |\overline{B_w^{\max}} \cap w([n])|  + \sum_{i \in w([n]), i \not \in B_w^{\max}} T_i, \text{ and }  
[\mathcal{Q}_L^E]_{w} = n - |\overline{B_w^{\max}} \cap w([n])| + \sum_{i \in B_w^{\max} \cap {w([n])}} T_i .
\end{align*}

\end{prop}
\begin{proof}
Let $v$ be a vector in the interior of $C_w$.
We have noted that $\overline{B_w^{\max}}$ of $\D$ is equal to the $w$-minimal feasible set of $\D^{\perp}$.
Then, as in the proof of Lemma~\ref{lem:envelopingindep}, we have that  
\[
 \textstyle \face{v}{(P(\D^{\perp}) + \square - \be_{[n]})} =  \frac{1}{2} \be_{B_w(\D^\perp)} - \frac{1}{2} \be_{w([n])}= \frac{1}{2} \be_{\overline{B_w^{\max}}} - \frac{1}{2} \be_{w([n])}.
\]
In order to compute the localization of the pullback of $\mathcal{S}_{\mathrm{univ}}$, we find a preimage of $\face{v}{(P(\D^{\perp}) + \square - \be_{[n]})}$ in the polytope of the matroid represented by $E_L$. 
A preimage in $IP(\M^{\perp})$ of this vertex is $\mathbf{u}_{\overline{B_w^{\max}}} - \mathbf{u}_{\overline{B_w^{\max}} \cap w([n])}$. A preimage of this in the matroid polytope of the matroid represented by $E_L$ extends the independent set $\overline{B_w^{\max}} \setminus \overline{B_w^{\max}} \cap w([n])$ of $\M^{\perp}$ to a basis without adding any elements in $[2n]$. 
Proposition~\ref{prop:toricgrass} then implies that the localization of the pullback of $\mathcal{S}_{\mathrm{univ}}$ at the fixed point of $X_{B_n}$ corresponding to $w$ is 
$$|\overline{B_w^{\max}} \cap w([n])| + \sum_{i \in \overline{B_w^{\max}} \setminus \overline{B_w^{\max}} \cap w([n])} T_i = |\overline{B_w^{\max}} \cap w([n])| + \sum_{i \in \overline{B_w^{\max} \cap w([n])}} T_i.$$
Because $\mathcal{Q}_L^E$ is the dual of the pullback of $\mathcal{S}_{\mathrm{univ}}$, this gives the result for $\mathcal{Q}_L^E$.
We note that $[\mathcal{M}]_w = n + \sum_{i \in w([n])} T_i$. 
As $[\mathcal{S}_L^E] = [\mathcal{M}] - [\mathcal{Q}_L^E]$, the result for $[\mathcal{S}_L^E]$ follows.
\end{proof}

In particular, the equivariant $K$-classes of $[\mathcal{S}_L^E]$ and $[\mathcal{Q}_L^E]$ depend only on the delta-matroid associated to $L$.
For arbitrary delta-matroid $\D$, the proof of Proposition~\ref{prop:welldefined} immediately adapts to show that we may define \newword{enveloping tautological classes} $[\mathcal S_\D^E]$ and $[\mathcal Q_\D^E]$ in $K_T(X_{B_n})$ by the formulas in Proposition~\ref{prop:envelopingK}.
Note that the enveloping tautological classes $[\mathcal{S}_{\D}^E]^{\vee}$ and $[\mathcal{Q}_{\D}^E]^{\vee}$ have ``nice Chern roots'' in the sense discussed above \Cref{prop:niceChern}.

\begin{rem}
Arguing analogously to \cite[Proposition 5.6]{BEST}, one can show that any fixed polynomial in the tautological classes of delta-matroids or their Chern classes is a valuative invariant of delta-matroids in the sense of \cite{ESS}.
\end{rem}

\subsection{Intersection computations}\label{ssec:intersection}
We now compute several intersection numbers arising from the Chern and Segre classes of isotropic and enveloping tautological classes. 
We first do the computations with enveloping tautological classes, which are easier to work with because they are closely related to the exceptional isomorphisms $\phi^B$ and $\zeta^B$ introduced in Section~\ref{sec:HRR}. We then relate an intersection number of the Chern classes of the isotropic tautological classes to one involving enveloping tautological classes.

We begin by realizing both the interlace polynomial and the $U$-polynomial as intersection numbers of the enveloping tautological classes. Because the classes $[\mathcal{S}_{\D}^E]$ do not have any positivity properties, this does not give log-concavity properties for the interlace polynomial. But these results will form the basis for later intersection theory computations that prove Theorem~\ref{mainthm:logconc}. In \cite[Theorem 8.1]{EHL}, the analogous computation on $X_{St_n}$ yields the rank-generating function of a matroid. 

\begin{thm}\label{thm:interlace}
We have that $\int_{X_{B_n}} c([\mathcal{S}_{\D}^E], u) \cdot c([\mathcal{Q}_{\D}^E], v) = v^n \operatorname{Int}_{\D}(u/v)$. 
\end{thm}

\begin{proof}
To compute $\int_{X_{B_n}} c([\mathcal{S}_{\D}^E], u) \cdot c([\mathcal{Q}_{\D}^E], v)$, we look at the degree $n$ part of $c^T([\mathcal{S}_{\D}^E], u) \cdot c^T([\mathcal{Q}_{\D}^E], v)$. Let $S \in \ads_n$, and consider the  cone $\tau_S$ whose rays are $\{\be_i : i\in S\}$. Then $\tau_S$ is a maximal cone in the fan $(\Sigma_{B_1})^n$ of $(\mathbb{P}^1)^n$.
The linear function defined by $\be_S$ attains its maximum on a face $F$ of $P(\D)$, and every function in the interior of $\tau_S$ attains its maximum on a face of $F$ because every cone of $\Sigma_{B_n}$ which is contained in $\tau_S$ contains $\be_S$. Note any point $x$ of $F$ minimizes the distance to $\be_S$ from $P(\D)$. 

Note that $C_w \in \tau_S$ if and only if $S = w([n])$. For each $w \in \W$ with $S = w([n])$, we have that
$$c^T([\mathcal{S}_{\D}^E])_{w} = \prod_{i \in S, i \not \in B_{w}^{\max}} (1 +  t_i), \quad \text{and } c^T([\mathcal{Q}_{\D}^E])_{w} = \prod_{i \in S\cap B_w^{\max}} (1 +  t_i).$$
We see that the degree $n$ part of $c^T([\mathcal{S}_{\D}^E], u)_w \cdot c^T([\mathcal{Q}_{\D}^E], v)_w$  is 
\[(-1)^{|S \cap [\bar{n}]|} u^{d_{\D}(S)}v^{n - d_{\D}(S)}t_1 \dotsb t_n.\] 
Note that, for each $S \in \ads_n$, the piecewise polynomial function that is $(-1)^{|S \cap [\bar{n}]|} t_1 \dotsb t_n$ on $\tau_S$ and vanishes otherwise is $c_n^T(\bigoplus_{i \in [n]} \pi_i^*\mathcal{O}(1))$, where we give $\mathcal{O}(1)$ on the $i$th copy of $\mathbb{P}^1$ the $\mathcal{O}(1_{\infty})$ linearization if $i \in S$, and give it the $\mathcal{O}(1_{o})$ linearization if $\bar{i} \in S$. Proposition~\ref{prop:explicitu} gives
$$\int_{X_{B_n}} c([\mathcal{S}_{\D}^E], u) \cdot c([\mathcal{Q}_{\D}^E], v) = \sum_{S \in \ads_n} u^{d_{\D}(S)}v^{n - d_{\D}(S)} \int_{(\mathbb{P}^1)^n} c_n(\oplus \pi_i^* \mathcal{O}(1)) = v^n \operatorname{Int}_{\D}(u/v). \qedhere$$ 
\end{proof}

We prepare to do more computations by studying how enveloping tautological classes restrict to smaller type $B$ permutohedral varieties. The description of the fan of $\Sigma_{B_n}$ implies that the closure of each coordinate $\mathbb{G}_{\rm m}^{n-1} \subset T$ in $X_{B_n}$ can be identified with $X_{B_{n-1}}$. The inclusion is $\mathbb{G}_{\rm m}^{n-1}$-equivariant, so for each $i \in n$, we have a map $K_T(X_{B_n}) \to K_{\mathbb{G}_{\rm m}^{n-1}}(X_{B_{n-1}})$ given by the composition of the forgetful map $K_T(X_{B_n}) \to K_{\mathbb{G}_{\rm m}^{n-1}}(X_{B_n})$ and the restriction map. Recall that for a delta-matroid $\D$ and $I \subseteq [n]$, $\D(I)$ is the projection of $\D$ away from $I$.

\begin{prop}\label{prop:restrictionprojection}
The images of $[\mathcal{S}_{\D}^E]$, $[\mathcal{Q}_{\D}^E]$, and $[\mathcal{I}_{\D}]$ under the map $K_T(X_{B_n}) \to K_{\mathbb{G}_{\rm m}^{n-1}}(X_{B_{n-1}})$ are $1 + [\mathcal{S}_{\D(i)}^E]$, $1 + [\mathcal{Q}_{\D(i)}^E]$, and $1 + [\mathcal{I}_{\D(i)}]$ respectively. 
\end{prop}

\begin{proof}
Under the embedding $X_{B_{n-1}} \hookrightarrow X_{B_n}$, each $\mathbb{G}_{\rm m}^{n-1}$-fixed point of $X_{B_{n-1}}$ is the identity of the torus embedded into a $T$-fixed curve in $X_{B_n}$ on which $\mathbb{G}_{\rm m}^{n-1}$ acts trivially. We may compute the $\mathbb{G}_{\rm m}^{n-1}$-equivariant localization at this fixed point by computing the $T$-equivariant localization at any $T$-fixed point of this curve, and then applying the forgetful map $K_T(\mathrm{pt}) \to K_{\mathbb{G}_{\rm m}^{n-1}}(\mathrm{pt})$.  Then the result follows from the definition of the tautological classes. 
\end{proof}

\begin{prop}
We have that 
$$U_{\D}(u, v) = \int_{X_{B_n}} c(\boxplus \mathcal{O}(1), u)\cdot c([\mathcal{S}_{\D}^E], v)\cdot  c([\mathcal{Q}_{\D}^E]). $$
\end{prop}

\begin{proof}
The zero-locus of a general element of the complete linear system of $\pi_i^* \mathcal{O}(1)$ is $\overline{\{t \in T \colon t_i = \lambda\}}$ for some $\lambda \in \kk^*$. As these divisor are all $\mathbb{G}_m$-translates of the closure of $\mathbb{G}_{\rm m}^{n-1}$,  the class $[X_{B_{n-1}}] \in A^1(X_{B_n})$ represents $c_1(\pi_i^* \mathcal{O}(1))$. 
Letting $i$ vary, we see that $c(\boxplus  \mathcal{O}(1))$ is the sum of the Chow classes of the closures of the coordinate subtori of $T$. The closure of each coordinate subtorus of $T$ can be identified with a smaller $X_{B_{k}}$. 
By the projection formula and Proposition~\ref{prop:restrictionprojection}, we see that
\begin{equation*}\begin{split}
\int_{X_{B_n}} c(\boxplus \mathcal{O}(1), u) \cdot c([\mathcal{S}_{\D}^E], v) \cdot c([\mathcal{Q}_{\D}^E], 1) &= \sum_{I \subseteq [n]} u^{|I|} \int_{X_{B_{n - |I|}}} c([\mathcal{S}_{\D}^E], v)|_{X_{B_{n - |I|}}} \cdot c([\mathcal{Q}_{\D}^E], 1)|_{X_{B_{n - |I|}}} \\ 
&= \sum_{I \subseteq [n]} u^{|I|} \int_{X_{B_{n - |I|}}} c([\mathcal{S}_{\D(I)}^E], v) \cdot c([\mathcal{Q}_{\D(I)}^E], 1).
\end{split}\end{equation*}
The the result follows from Theorem~\ref{thm:interlace} and Proposition~\ref{prop:explicitu}.
\end{proof}

Recall that $\gamma$ is the first Chern class of the line bundle corresponding to the cross polytope $\Diamond$ and $s$ denotes the Segre class.
We now do the computation which underlies the proof of Theorem~\ref{mainthm:logconc}(\ref{eq:envelopinglorentzian}).

\begin{thm}\label{thm:intersectenveloping}
We have that 
$$\int_{X_{B_n}} s([\mathcal{Q}_{\D}^E]^{\vee}, z) \cdot  c([\mathcal{Q}_{\D}^E], w) \cdot \frac{1}{1 - y \gamma} \cdot c(\boxplus \mathcal{O}(1), x) =(y + w)^n U_{\D}\left(\frac{2z + x}{y + w}, \frac{y - z}{y + w}\right).$$
\end{thm}

The key tools in the proof are the two exceptional isomorphisms and the Hirzebruch--Riemann--Roch-type formulas that they satisfy, which are a manifestation of Serre duality. This allows us to show the equality of certain intersection numbers, and we leverage Theorem~\ref{thm:interlace} to compute more intersection numbers. 

\begin{proof}
We prove the theorem in three steps. \\
\textbf{Step 1}: we have that
$$\int_{X_{B_n}} s([\mathcal{Q}_{\D}^E]^{\vee}, z) \cdot  c([\mathcal{Q}_{\D}^E]) = U_{\D}(2z, -z).$$
Because $[\mathcal{S}_{\D}^E] + [\mathcal{Q}_{\D}^E] = [\mathcal{M}] = [\boxplus \mathcal{O}(1)^{\oplus 2}]$, we have that $c([\mathcal{S}_{\D}^E], z) \cdot c([\mathcal{Q}_{\D}^E], z) = c(\boxplus \mathcal{O}(1)^{\oplus 2}, z) = c(\boxplus  \mathcal{O}(1), 2z)$.
So 
$$s([\mathcal{Q}_{\D}^E]^{\vee}, z) = c([\mathcal{S}_{\D}^E], -z) \cdot c(\boxplus \mathcal{O}(1), 2z).$$
Then, using Proposition~\ref{prop:restrictionprojection}, we see that
\begin{equation*} \begin{split}
\int_{X_{B_n}}s([\mathcal{Q}_{\D}^E]^{\vee}, z) \cdot c([\mathcal{Q}_{\D}^E], w) &= \int_{X_{B_n}}c([\mathcal{S}_{\D}^E], -z) \cdot c(\boxplus \mathcal{O}(1), 2z) \cdot c([\mathcal{Q}_{\D}^E], w) \\
& = \sum_{I \subseteq [n]} (2z)^{|I|} \int_{X_{B_{n- |I|}}}c([\mathcal{S}_{\D(I)}^E], -z) \cdot c([\mathcal{Q}_{\D(I)}^E], w) \\
& = \sum_{I \subseteq [n]} (2z)^{|I|}  w^{n - |I|}  \operatorname{Int}_{\D(I)}(-z/w).
\end{split} \end{equation*}
Setting $w = 1$ and using Lemma~\ref{lem:sumprojection} gives the result. 
\\
\textbf{Step 2}: we have that
$$\int_{X_{B_n}} s([\mathcal{Q}_{\D}^E]^{\vee}, z) \cdot c([\mathcal{Q}_{\D}^E], w) \cdot  \frac{1}{1 - \gamma} =(1 + w)^n U_{\D}\left(\frac{2z}{1 + w}, \frac{1 - z}{1 + w}\right).$$
Let $[\square]$ be the class of the line bundle corresponding to the cube $\square=[0,1]^n$. From Lemma~\ref{lem:stellcommute} and \cite[Corollary 6.5(1)]{EHL}, we have that both $\phi^B([\square]) = c(\boxplus \mathcal{O}(1))$ and $\zeta^B([\square]) = c(\boxplus \mathcal{O}(1))$. Applying Proposition~\ref{prop:niceChern}, Proposition~\ref{prop:otherHRR}, and Theorem~\ref{mainthm:HRR}, we get that 
\begin{align*}
&\chi\left (\Big(\sum_{j \ge 0} \operatorname{Sym}^j [\mathcal{Q}_{\D}^E]^{\vee} z\Big)\, \Big(\sum_{i \ge 0} \wedge^i [\mathcal{Q}_{\D}^E]^{\vee} w\Big)\,[\square] \right) \\ 
&=\int_{X_{B_n}} \frac{1}{(1 - z)^n} \cdot  s\left([\mathcal{Q}_{\D}^E]^{\vee}, \frac{z}{z-1}\right) \cdot (w + 1)^n \cdot c\left([\mathcal{Q}_{\D}^E]^{\vee}, \frac{w}{1+w} \right)\cdot \frac{1}{1-\gamma} \\ 
&=\int_{X_{B_n}} \frac{1}{(1 - z)^n}s\left([\mathcal{Q}_{\D}^E], \frac{1}{1-z} \right)\cdot (w + 1)^n \cdot c\left([\mathcal{Q}_{\D}^E], \frac{1}{1+w} \right)\cdot c(\boxplus \mathcal{O}(1), 2).
\end{align*}
Equating the two right-hand sides, canceling, and replacing $w$ by $-\frac{w}{1+w}$ and $z$ by $\frac{z}{z-1}$, we obtain
\begin{equation*}\begin{split}
&\int_{X_{B_n}} s([\mathcal{Q}_{\D}^E]^{\vee}, z)\cdot c([\mathcal{Q}_{\D}^E], w)\cdot \frac{1}{1-\gamma} =\int_{X_{B_n}} s([\mathcal{Q}_{\D}^E], 1 - z)\cdot c([\mathcal{Q}_{\D}^E], w + 1)\cdot c(\boxplus \mathcal{O}(1), 2).
\end{split}\end{equation*}
Substituting in the result of Step 1 after homogenizing, we have that
$$\int_{X_{B_n}} s([\mathcal{Q}_{\D}^E]^{\vee}, z - 1)\cdot c([\mathcal{Q}_{\D}^E], 1 + w) = (1 + w)^n U_{\D} \left( \frac{2(z-1)}{1 + w}, \frac{1 - z}{1 + w} \right).$$
Therefore, using Lemma~\ref{lem:sumprojection}, we have that
\begin{equation*}\begin{split}
\int_{X_{B_n}} s([\mathcal{Q}_{\D}^E]^{\vee}, z)\cdot c([\mathcal{Q}_{\D}^E], w)\cdot \frac{1}{1 - \gamma} &= \sum_{I \subseteq [n]} 2^{|I|} (1 + w)^{n - |I|} U_{\D(I)}\left(\frac{2(z-1)}{1 + w}, \frac{1 - z}{1 + w}\right) \\
& = (1 + w)^n \sum_{I \subseteq E} \left (\frac{2}{1 + w} \right)^{|I|} U_{\D(I)} \left(\frac{2(z-1)}{1 + w}, \frac{1 - z}{1 + w}\right) \\
& = (1 + w)^n U_{\D}\left(\frac{2z}{1 + w}, \frac{1 - z}{1 + w}\right)
\end{split}\end{equation*}
\textbf{Step 3}: we now prove the result. 
We compute:
\begin{equation*}\begin{split}
&\int_{X_{B_n}} s([\mathcal{Q}_{\D}^E]^{\vee}, z)\cdot c([\mathcal{Q}_{\D}^E], w)\cdot \frac{1}{1 - y \gamma}\cdot c(\boxplus \mathcal{O}(1), x)  \\ 
&= \sum_{I \subseteq [n]} x^{|I|} \int_{X_{B_n}} s([\mathcal{Q}_{\D(I)}^E]^{\vee}, z)\cdot c([\mathcal{Q}_{\D(I)}^E], w)\cdot \frac{1}{1 - y \gamma} \\
& = \sum_{I \subseteq [n]} (y + w)^{n - |I|} x^{|I|} U_{\D(I)}\left(\frac{2z}{y + w}, \frac{y - z}{y + w}\right) \\
&= (y + w)^n U_{\D}\left(\frac{2z + x}{y + w}, \frac{y - z}{y + w}\right).\qedhere
\end{split}\end{equation*}
\end{proof}

\begin{thm}\label{thm:intersectisotropic}
Let $\D$ be a delta-matroid. We have that
$$\int_{X_{B_n}} c([\mathcal{I}_{\D}]^{\vee}, q)\cdot \frac{1}{1 - y \gamma}\cdot \prod_{i=1}^{n} (1 + x_i h_i) = (y + q)^n U_{\D}\left(\frac{x_1}{y + q}, \dotsc, \frac{x_n}{y + q}, \frac{y - q}{y + q}\right).$$
\end{thm}

Recall that $h_i = c_1(\pi_i^*\mathcal{O}(1))$, and note that $\prod_{i=1}^{n} (1 + x h_i) = c(\boxplus \mathcal{O}(1), x)$. We prove the above theorem by relating it to Theorem~\ref{thm:intersectenveloping}. We first recall the equivariant descriptions of $c^T([\mathcal{I}_{\D}]^{\vee})$. Recall that if $i \in [n]$, then $t_{\bar{i}} := - t_i$.  On a fixed point of $X_{B_n}$ corresponding to $w \in \W$, we have that 
$$c^T([\mathcal{I}_{\D}]^{\vee}, q)_w = \prod_{i \in B_w} (1 - t_iq) = \prod_{i \in \overline{B_w}} (1 + t_i q).$$

\begin{proof}
We claim that 
$$ s([\mathcal{Q}_{\D^{\perp}}^E]^{\vee}, q)\cdot c([\mathcal{Q}_{\D^{\perp}}^E], q)\cdot c(\boxplus \mathcal{O}(1), - 2q) =  c([\mathcal{I}_{\D}]^{\vee}, q).$$
Then Theorem~\ref{thm:intersectenveloping} implies that
$$\int_{X_{B_n}} c([\mathcal{I}_{\D}]^{\vee}, q)\cdot \frac{1}{1 - y \gamma} = \int_{X_{B_n}} s([\mathcal{Q}_{\D^{\perp}}^E]^{\vee}, q)\cdot c([\mathcal{Q}_{\D^{\perp}}^E], q)\cdot c(\boxplus \mathcal{O}(1), - 2q) = (y + q)^n \operatorname{Int}_{\D^{\perp}}\left( \frac{y - q}{y + q} \right).$$
Then, assuming the claim, the result follows using that $\operatorname{Int}_{\D}(v) = \operatorname{Int}_{\D^{\perp}}(v)$, Proposition~\ref{prop:restrictionprojection}, and the definition of the multivariate $U$-polynomial. 

Observe that
$$c^T([\mathcal{Q}_{\D^{\perp}}^E], q)_w = \prod_{i \in \overline{B_w} \cap w([n])} (1 + t_i q), \text{ and } s^T([\mathcal{Q}_{\D^{\perp}}^E]^{\vee}, q)_w = c^T([\mathcal{Q}_{\D^{\perp}}^E], -q)^{-1}_w = \prod_{i \in {B_w} \cap \overline{w([n])}} \frac{1}{1 + t_i q}.$$
On $\mathbb{P}^1$, the piecewise polynomial function which is $t$ on the cone $\{x < 0\}$ and $-t$ on the cone $\{x > 0\}$ is a linearization of $\mathcal{O}(-2)$.  Therefore, with this linearization, we have that
$$c^T(\boxplus \mathcal{O}(1), -2q)_w = \prod_{i \in \overline{w([n])}} (1 + t_{i}q).$$
Then the claim follows from multiplying the above expressions together. 
\end{proof}

\section{Log-concavity}\label{sec:logconc}

In this section, we prove Theorem~\ref{mainthm:logconc}. First we recall some definitions. Let $f \in \mathbb{R}[x_1, \dotsc, x_n]$ be a homogeneous polynomial of degree $d$. 
If $f = \sum a_{\mathbf{m}} x^{\mathbf{m}}$, then the \newword{normalization} of $f$, denoted $N(f)$, is the polynomial $\sum a_{\mathbf{m}} \frac{ x^{\mathbf{m}}}{\mathbf{m}!}$, where $\mathbf{m}! = m_1! \dotsb m_n!$ if $\mathbf{m} = (m_1, \dotsc, m_n)$. We call $f$ the \newword{denormalization} of $N(f)$. 
We say that $f$ is \newword{strictly Lorentzian} if the coefficient of every monomial of degree $d$ is positive, and every quadratic form obtained by taking $d - 2$ partial derivatives of $f$ is nondegenerate with signature $(+, -, \dotsc, -)$. We say that $f$ is \newword{Lorentzian} if it is a coefficientwise limit  of strictly Lorentzian polynomials. It follows from \cite[Example 2.26]{BH} and \cite[Theorem 2.10]{BH} that a denormalized Lorentzian polynomial has a log-concave unbroken array of coefficients. We now state a strengthening of (\ref{eq:isotropiclorentzian}) in Theorem~\ref{mainthm:logconc}.

\begin{thm}\label{thm:strengthened}
Let $\D$ be a delta-matroid which has an enveloping matroid. Then the polynomial 
\begin{equation}\label{eq:multivariable}
(y + q)^n U_{\D}\left(\frac{x_1}{y + q}, \dotsc, \frac{x_n}{y + q}, \frac{y - q}{y + q}\right)
\end{equation}
is denormalized Lorentzian. 
\end{thm}
By \cite[Lemma 4.8]{BLP}, this is indeed a strengthening of the statement that (\ref{eq:isotropiclorentzian}) is denormalized Lorentzian. Even when $\D$ has an enveloping matroid, we do not know if there is a denormalized Lorentzian evaluation of the multivariable $U$-polynomial that specializes to (\ref{eq:envelopinglorentzian}).
We have the following corollaries of Theorem~\ref{mainthm:logconc} and Theorem~\ref{thm:strengthened}.

\begin{cor}\label{cor:logconcave}
Let $\D$ be a delta-matroid which has an enveloping matroid. Then the coefficients of
$(y + 1)^n U_{\D}(0, \frac{y-1}{y+1} ) = (y+1)^n \operatorname{Int}_{\D}(\frac{y-1}{y+1})$ and
$U_{\D}(2u, -u)$ form a nonnegative log-concave sequence with no internal zeros, and in particular form a unimodal sequence. After multiplying the coefficient of $u^k$ in $U_{\D}(u, 0)$ or $U_{\D}(u, -1)$ by $k!$, the resulting sequence is a nonnegative log-concave sequence with no internal zeros. 
\end{cor}

\begin{proof}
To obtain the first two results, we set $x = 0, q = 1$ in (\ref{eq:isotropiclorentzian}) and set $x = y = 0, w=1$ in (\ref{eq:envelopinglorentzian}) respectively, and then apply \cite[Lemma 4.8]{BLP}.
To obtain the last two results, we normalize (\ref{eq:multivariable}) and set $y = q = 1/2, x_i = u$ and set $y = 0, q = 1, x_i = u$ respectively, and then apply \cite[Corollary 3.7]{BH}. 
\end{proof}

\begin{rem}
In \cite[Proposition 3.4, Theorem 3.8]{LarRank}, the third author showed that the coefficients of $U_{\D}(u, 0)$ count the number of independent set (i.e., subsets of feasible sets) of $\D$ by their cardinality, and the coefficients of $U_{\D}(u, -1)$ count the number of faces of a delta-matroid analogue of the broken circuit complex of a matroid. In particular, Corollary~\ref{cor:logconcave} gives an analogue of the log-concavity of the independence polynomial and the characteristic polynomial of a matroid \cite{AHK18}.
\end{rem}

\begin{rem}\label{rem:unimodalhistory}
For the adjacency delta-matroid $\D(G)$ of a graph $G$ (\Cref{eg:adjacency}), \cite{InterlaceBollobas} conjectured that the coefficients of $\operatorname{Int}_{\D(G)}(v-1)$ form a unimodal sequence, which was disproved by \cite{DP10}.  Both works conjectured that $\operatorname{Int}_{\D(G)}(v)$  has unimodal coefficients.
We note that $\operatorname{Int}_\D(v)$ may not have unimodal coefficients even when $\D$ is an even delta-matroid with a $D_n$ representation, like $\D(G)$.  See \Cref{eg:notunimodal} below. In \cite[Corollary 7.22]{FerroniSchroter}, Ferroni and Schr\"{o}ter gave an example of a matroid $\M$ such that $\operatorname{Int}_{P(\M)}(v)$ is not unimodal.
\end{rem}

\begin{eg}\label{eg:notunimodal}
Let $\mathrm{U}^\circ_{r,n}$ be the even delta-matroid on $[n,\overline n]$ whose feasible sets are
\[
\{S\cup ([\overline n]\setminus \overline S) : \text{$S\subseteq [n]$ with $|S|\leq r$ and $|S|\equiv r\ \operatorname{mod} 2$}\}.
\]
That is, the vertices of the polytope $P(\mathrm{U}_{r,n}^\circ)$ are obtained from $IP(\mathrm{U}_{r,n})$ by taking only the vertices corresponding to subsets with parity equal to that of $r$.  Then $\mathrm{U}^\circ_{r,n}$ has a $D_n$ representation by the row-span of the $n\times 2n$ matrix
\[
\begin{bmatrix}
I_r & A &\vline &B & 0 \\
0 & 0 &\vline &-A^t & I_{n-r}
\end{bmatrix}
\]
where $I_k$ is the $k\times k$ identity matrix, $A$ is a general $r\times (n-r)$ matrix, and $B$ is a general $r\times r$ skew-symmetric matrix.  In particular, $\mathrm{U}_{r,n}^\circ$ has an enveloping matroid.  Using the formula in \Cref{prop:explicitu}, we compute that the coefficients of $(1,v,v^2,v^3, \ldots)$ in $\operatorname{Int}_{\mathrm{U}_{m-3,2m}^\circ}(v)$ are
\[
\left(\sum_{\substack{0\leq i \leq m-3\\ i \equiv m-3 \ \operatorname{mod}2}}\binom{2m}{i}, \sum_{\substack{0\leq i \leq m-3\\ i \not\equiv m-3 \ \operatorname{mod}2}}\binom{2m}{i} + \binom{2m}{m-2}, \binom{2m}{m-1}, \binom{2m}{m}, \ldots \right)
\]
For large $m$, this sequence is not unimodal.  For instance, at $m = 10$ the sequence reads
\[(94184, 169766, 167960, 184756, \ldots).\]
In particular, the interlace polynomial of an even delta-matroids with a $D_n$ representation need not have unimodal or log-concave coefficients.
\end{eg}

\begin{rem}
The nonnegativity of the coefficients of $U_{\D}(2u, -u)$, which is part of the content of Corollary~\ref{cor:logconcave}, can be proven directly using the recursive definition of the $U$-polynomial. 
\end{rem}

\subsection{Motivation}
We exhibit the general strategy for constructing log-concave sequences from vector bundles, first used in \cite[Section 9]{BEST} and later placed into a general framework in \cite{EHL}.
We do this in the special case of showing that the coefficients of
$(y+1)^n \operatorname{Int}_{\D} (\frac{y-1}{y+1})$ are log-concave when $\D$ has an enveloping matroid.

Setting $x=0$ and $q=1$ in Theorem~\ref{thm:intersectisotropic}, we have the equality
$$\int_{X_{B_n}}c([\mathcal{I}_D]^{\vee})\cdot \frac{1}{1-y\gamma}=(y+1)^n \operatorname{Int}_{\D} \left (\frac{y-1}{y+1} \right).$$
Suppose first we are in the special case that $\D$ has a $B_n$ representation $L\subset \kk^{2n+1}$. The first step will be rewriting this intersection to involve Segre classes rather than Chern classes. As $\mathcal{I}_L$ is a subbundle of $\mathcal{O}^{\oplus 2n+1}_{X_{B_n}}$, by dualizing we obtain a short exact sequence
$$0\to \mathcal{K}_L\to \mathcal{O}^{\oplus 2n+1}_{X_{B_n}}\to \mathcal{I}_L^{\vee}\to 0$$
for some vector bundle $\mathcal{K}_L$. Then $c(\mathcal{I}_L^{\vee})=s(\mathcal{K}_L)$, and so 
$$\int_{X_{B_n}}c(\mathcal{I}_L^{\vee}) \,\frac{1}{1-y\gamma }=\int_{X_{B_n}} s(\mathcal{K}_L)\,\frac{1}{1-y\gamma }=\sum_{k=0}^ny^k\int_{X_{B_n}}s_{n-k}(\mathcal{K}_L)\gamma^{k}=\sum_{k=0}^n y^k\int_{\mathbb{P}(\mathcal{K}_L)} \delta^{2n-k}\gamma^k$$
where $\delta$ is the first Chern class of $\mathcal{O}(1)$ on $\mathbb{P}(\mathcal{K}_L)$.
The Khovanskii--Teissier inequality implies the coefficient sequence is log-concave. To establish this log-concavity beyond the case that $\D$ is $B_n$ representable, we note that we may rewrite the last equation as
$$\sum_{k=0}^{n} y^k\int_{\mathbb{P}(\mathcal{K}_L)} \delta^{2n-k}\gamma^k=\sum_{k=0}^n y^k \int_{X_{B_n} \times \mathbb{P}^{2n}}[\mathbb{P}(\mathcal{K}_L)]\delta^{2n-k}\gamma^k,$$
where $[\mathbb{P}(\mathcal{K}_L)] \in A^\bullet(X_{B_n}\times \mathbb{P}^{2n})=A^\bullet(X_{B_n})[\delta]/(\delta^{2n+1})$ is the fundamental class of $\mathbb{P}(\mathcal{K}_L) \subset X_{B_n} \times \mathbb{P}^{2n}$. 
We have the formula $[\mathbb{P}(\mathcal{K}_L)] = \sum_{i=0}^{n}c_{n-i}(\mathcal{I}_L^{\vee})\delta^{i}$.
The formula for this class makes sense for any delta-matroid, and one can formally define
$[\mathbb{P}(\mathcal{K}_{\D})]=\sum_{i=0}^n c_{n-i}([\mathcal{I}_\D]^{\vee})\delta^{i} \in A^{\bullet}(X_{B_n} \times \mathbb{P}^{2n})$. By Theorem~\ref{thm:intersectisotropic},  $\int_{X\times \mathbb{P}^{2n}}[\mathbb{P}(\mathcal{K}_\D)]\delta^{2n-k}\gamma^k$ still computes the coefficients of $(y+1)^n\operatorname{Int}_{\D}(\frac{y-1}{y+1})$.

In order to deduce log-concavity, we need to know that the Chow class $[\mathbb{P}(\mathcal{K}_{\D})]$ has Hodge-theoretic properties resembling those of an irreducible subvariety. The framework of \cite[Section 8.3]{EHL} constructs classes which are associated to any \emph{matroid} which have good Hodge-theoretic properties.\footnote{For technical reasons we actually work with classes in $A^\bullet(X_{B_n} \times \mathbb{P}^{2n-1})$ instead of $A^\bullet(X_{B_n} \times \PP^{2n})$ which more naturally extend to all rank $n$ matroids, but the underlying idea is the same.} 
The strategy is to relate the class to the {Bergman fan} of some matroid, which has good Hodge-theoretic properties by \cite{AHK18}. The notion of {valuativity} for invariants of matroids is used to reduce certain computations to the case of realizable matroids. 
When $\D$ has an enveloping matroid $\M$, we can use this to deduce that $[\mathbb{P}(\mathcal{K}_{\D})]$ has good Hodge-theoretic properties.

\subsection{Proof of log-concavity}

Before proving Theorem~\ref{mainthm:logconc}, we prove a log-concavity statement for an arbitrary matroid of rank $n$ on $[n, \bar{n}]$ (Theorem~\ref{thm:logconcmatroid}) by using the framework in \cite[Section 8.3]{EHL}, which is based on \cite[Section 9]{BEST}. Afterwards, we relate this log-concavity statement to Theorem~\ref{mainthm:logconc}. 
Using Proposition~\ref{prop:toricgrass}, we construct two types of vector bundles on $X_{B_n}$ that are associated to a realization of a matroid of rank $n$ on $[n, \bar{n}]$. First we give a definition (cf.\ \Cref{def:strongly valuative}). 

\begin{defn}
Let $A$ be an abelian group. 
A function \[\varphi \colon \{\text{matroids of rank $r$ on } [n]\} \to A\] is \textbf{valuative} if it factors through the map $\M \mapsto \ind{P(\M)}$. That is, for any matroids $\M_1, \ldots, \M_k$ and integers $a_1, \ldots, a_k$ such that $\sum a_i \ind{P(\M_i)} = 0$, we have that $\sum a_i \varphi(\M_i) = 0$.  
\end{defn}

Let $T$ act on $\kk^{4n}$ by $(t_1x_1, t_2x_2, \ldots, t_nx_n, t_1^{-1}x_{n+1}, \ldots, t_n^{-1}x_{2n}, x_{2n+1}, \ldots, x_{4n})$. Let $L \subset \kk^{2n}$ be a linear space of dimension $n$. Let 
$E_L$ be the image of $L^{\perp}$ under the diagonal embedding of $\kk^{2n}$ into $\kk^{4n}$
and consider the point $[E_L] \in Gr(n; 4n)$. The fan of the normalization of $\overline{T \cdot [E_L]}$ is the normal fan of $\operatorname{env}(IP(\M))$. Every edge of $\operatorname{env}(IP(\M))$ is parallel to $\be_i$ or $\be_i \pm \be_j$, so $\Sigma_{B_n}$ is a coarsening of the normal fan of $\operatorname{env}(IP(\M))$. Therefore there is a toric morphism $X_{\Sigma_{B_n}} \to Gr(n; 4n)$. Set $\mathcal{S}_{\mathrm{univ}}$ and $\mathcal{Q}_{\mathrm{univ}}$ to be the universal subbundle and quotient bundle respectively on $Gr(n; 4n)$.
Let $\tilde{\mathcal{K}}_L^E$ and $\tilde{\mathcal{Q}}_L^E$ be the duals of the pullbacks of $\mathcal{Q}_{\mathrm{univ}}$ and $\mathcal{S}_{\mathrm{univ}}$ respectively. 

\begin{lem}
For each $w \in \W$, let $I_{w}$ be any independent set of $\M^{\perp}$ such that any functional in the interior of $C_w$ achieves its minimum on the corresponding vertex of $\operatorname{env}(IP(\M^{\perp}))$. 
Then
$$[\tilde{\mathcal{Q}}_L^E]_{w} = n - |I_w \cap w([n])| + \sum_{i \in I_{w} \cap w([n])} T_i, \text{ and } [\tilde{\mathcal{K}}_L^E]_{w} = n + |I_w \cap w([n])| + \sum_{i \not \in w([n]) \cap I_{w}} T_i. $$
\end{lem}
Note that the classes $[\tilde{\mathcal{Q}}_L^E]$ and $[\tilde{\mathcal{K}}_L^E]$ only depend on the matroid $\M$ that $L$ represents. For any matroid $\M$ of rank $n$ on~$[n, \bar{n}]$, we define classes $[\tilde{\mathcal{Q}}_\M^E]$ and $[\tilde{\mathcal{K}}_\M^E]$ in $K_T(X_{B_n})$; the proof of Proposition~\ref{prop:welldefined} adapts to show that these are indeed well-defined. The proof of \cite[Proposition 5.6]{BEST} shows that any function that maps a matroid $\M$ of rank $n$ on $[n, \bar{n}]$ to a fixed polynomial expression in the Chern classes of $[\tilde{\mathcal{Q}}_{\M}^E]$ and $[\tilde{\mathcal{K}}_{\M}^E]$ is a valuative invariant of matroids of rank $n$ on $[n, \bar{n}]$. 

\medskip

We now construct analogues of isotropic tautological bundles. Consider a matroid $\M$ of rank $n$ on $[n, \bar{n}]$ represented by $L \subset \kk^{2n}$.
Then $L$ determines a $\kk$-valued point of $Gr(n; 2n)$. We have a $T$-action on $Gr(n; 2n)$ given by 
\[(t_1, \dotsc, t_n) \cdot (x_1, \dotsc, x_{n}, x_{\bar{1}}, \dotsc, x_{\bar{n}}) = (t_1 x_1, \dotsc, t_n x_n, t_1^{-1} x_{\bar{i}}, \dotsc, t_n^{-1} x_{\bar{n}}).\]
The fan of the normalization of $\overline{T \cdot [L]}$ is the toric variety with normal fan $\operatorname{env}(P(\M))$, which is a coarsening of $\Sigma_{B_n}$. This determines a morphism $X_{B_n} \to Gr(n; 2n)$; define $\tilde{\mathcal{K}}_{L}$ to be dual of the pullback of the universal quotient bundle $\mathcal{Q}_{\mathrm{univ}}$ under this map. Proposition~\ref{prop:toricgrass} implies the following lemma.

\begin{lem}
For $w \in \W$, let $B_w$ be a basis corresponding to any vertex in the preimage of the vertex of $\operatorname{env}(P(\M))$ that any functional in $C_w^{\circ}$ achieves its minimum on. Then
$$[\tilde{\mathcal{K}}_{L}]_w = \sum_{i \in {B_w}} T_i.$$
\end{lem}

Note that the above description of the equivariant $K$-class depends only on the matroid $\M$. 
Define $[\tilde{\mathcal{K}}_{\M}] \in K_T(X_{B_n})$ by the above formula for any $\M$; the proof of Proposition~\ref{prop:welldefined} adapts to show that these are indeed well-defined. The proof of \cite[Proposition 5.6]{BEST} shows that any function that maps a matroid $\M$ of rank $n$ on $[n, \bar{n}]$ to a fixed polynomial expression in the Chern classes of $[\tilde{\mathcal{K}}_{\M}]$ is a valuative invariant of matroids of rank $n$ on $[n, \bar{n}]$. We now use the framework \cite[Section 8.3]{EHL}, which establishes log-concavity properties for classes constructed in this way associated to loop-free and coloop-free matroids $\M$. Indeed, the above constructions give globally generated vector bundles associated to realizations of matroids of rank $n$ on $[n, \bar{n}]$. The Chern classes of these vector bundle depend only on the underlying matroid and depend valuatively on the matroid. Then \cite[Theorem 8.7]{EHL} gives the following result.

\begin{thm}\label{thm:logconcmatroid}
Let $\M$ be a loop-free and coloop-free matroid of rank $n$ on $[n, \bar{n}]$. Then the polynomials
$$\int_{X_{B_n}} s([\tilde{\mathcal{Q}}_{\M}^E]^{\vee}, z) \cdot s([\tilde{\mathcal{K}}_{\M}^E], w)\cdot \frac{1}{1 - y \gamma}\cdot c(\boxplus \mathcal{O}(1), x) \text{ and } \int_{X_{B_n}} s([\tilde{\mathcal{K}}_\M], q) \cdot \frac{1}{1 - y \gamma}\cdot \prod_{i=1}^{n} (1 + x_i h_i) $$
are denormalized Lorentzian. 
\end{thm}

\begin{proof}[Proof of Theorem~\ref{mainthm:logconc} and Theorem~\ref{thm:strengthened}]
We first do (\ref{eq:envelopinglorentzian}). Consider the case when $\D$ is loop-free and coloop-free. By Lemma~\ref{lem:envelopingloop-free}, the enveloping matroid $\M$ of $\D$ is loop-free and coloop-free. 
Then $[\tilde{\mathcal{Q}}_{\M}^E] = [\mathcal{Q}_{\D}^E]$, so $s([\tilde{\mathcal{Q}}_{\M}^E]^{\vee}, z) =  s([\mathcal{Q}_{\D}^E]^{\vee}, z)$. Also, $s([\tilde{\mathcal{K}}_{\M}^E], w) = c([\tilde{\mathcal{Q}}_{\M}^E], w) =  c([\mathcal{Q}_{\D}^E], w)$. We see that
\begin{equation*}\begin{split}
&\int_{X_{B_n}} s([\tilde{\mathcal{Q}}_{\M}^E]^{\vee}, z) \cdot s([\tilde{\mathcal{K}}_{\M}^E], w)\cdot \frac{1}{1 - y \gamma}\cdot c(\boxplus \mathcal{O}(1), x) \\ 
&=  \int_{X_{B_n}} s([\mathcal{Q}_{\D}^E]^{\vee}, z)\cdot c([\mathcal{Q}_{\D}^E], w)\cdot \frac{1}{1 - y \gamma}\cdot c(\boxplus \mathcal{O}(1), x)
= (y + w)^n U_{\D}\left(\frac{2z + x}{y + w}, \frac{y - z}{y + w}\right) 
\end{split}\end{equation*}
by Theorem~\ref{thm:intersectenveloping}.
So when $\D$ is loop-free and coloop-free, Theorem~\ref{thm:logconcmatroid} gives that the above polynomial is denormalized Lorentzian. In general, we can write $\D = \D' \times P(U_{0,k}) \times P(U_{\ell, \ell})$ for some $k$ and $\ell$, where $D'$ is loop-free and coloop-free. Using the behavior of the $U$-polynomial for delta-matroids with loops, we have that 
\begin{equation*}\begin{split}
&(y + w)^n U_{\D}\left(\frac{2z + x}{y + w}, \frac{y - z}{y + w}\right) =\\ 
&\left ((y + w)^{n - k - \ell} U_{\D'}\left(\frac{2z + x}{y + w}, \frac{y - z}{y + w}\right) \right)\cdot  \\
&\left( (y+w)^{k} U_{P(U_{0,k})}\left(\frac{2z + x}{y + w}, \frac{y - z}{y + w}\right) \right) \cdot \left( (y+w)^{\ell} U_{P(U_{\ell,\ell})}\left(\frac{2z + x}{y + w}, \frac{y - z}{y + w}\right) \right) \\
&= (y + w)^{n - k - \ell} U_{\D'}\left(\frac{2z + x}{y + w}, \frac{y - z}{y + w}\right)\cdot (z + 3y + w)^{k + \ell}
\end{split}\end{equation*}
As product of denormalized Lorentzian polynomials are denormalized Lorentzian \cite[Corollary 3.8]{BH}, we see that (\ref{eq:envelopinglorentzian}) is denormalized Lorentzian for all delta-matroids $\D$ that have an enveloping matroid. 

The proof of Theorem~\ref{thm:strengthened} is identical: one shows that, when $\M$ is an enveloping matroid of a loop-free and coloop-free delta-matroid $\D$,  
\begin{equation*}\begin{split}
\int_{X_{B_n}} s([\tilde{\mathcal{K}}_\M], q) \cdot \frac{1}{1 - y \gamma}\cdot \prod_{i=1}^{n} (1 + x_i h_i) &= \int_{X_{B_n}} c([\mathcal{I}_{\D}]^{\vee}, q)\cdot \frac{1}{1 - y \gamma}\cdot \prod_{i=1}^{n} (1 + x_i h_i) \\
&= (y + q)^n U_{\D}\left(\frac{x_1}{y + q}, \dotsc, \frac{x_n}{y + q}, \frac{y - q}{y + q}\right)
\end{split}\end{equation*}
by Theorem~\ref{thm:intersectisotropic}. One then deduces the general case using the behavior of the $U$-polynomial under products. 
\end{proof}

\begin{rem}
Our proof that \eqref{eq:envelopinglorentzian} is denormalized Lorentzian only requires that $\D$ has a sheltering matroid, as we only need that there is a matroid $\M$ with $\operatorname{env}(IP(\M)) = P(\D) + \square - \be_{[n]}$. See Remark~\ref{rem:shelter}.
\end{rem}

\bibliographystyle{alpha}
\bibliography{typeb.bib}

\end{document}